\documentclass[11pt]{article}
\usepackage{times}
\usepackage{star}
\usepackage{natbib}
\usepackage[linktocpage]{hyperref}
\usepackage{mathtools}
\usepackage{bbm}
\usepackage{color}
\usepackage[toc,page,header]{appendix}
\usepackage{minitoc}

\hypersetup{
	colorlinks=true,
	linkcolor=blue,
	citecolor=blue,
}
\usepackage{geometry}
\newcommand{\vol}{{\rm vol}}
\newcommand{\dvol}{{\rm dvol}}
\newcommand{\xk}{\mathfrak{x}}
\newcommand{\Exp}{\textnormal{Exp}}
\newcommand{\Log}{\textnormal{Log}}

\newcommand{\Cut}{{\rm Cut}}
\newcommand{\Ucut}{{\rm Ucut}}
\newcommand{\seg}{{\rm seg}}
\newcommand{\Iso}{{\rm Iso}}
\newcommand{\inj}{{\rm inj}}
\newcommand{\HH}{\mathbb{H}}
\newcommand{\TT}{\mathbb{T}}
\newcommand{\Ad}{{\rm Ad}}

\newcommand{\Hess}{{\rm Hess}}
\newcommand{\sn}{{\rm sn}}
\newcommand{\MLE}{{\rm MLE}}
\newcommand{\Frechet}{{Fr{$\rm \acute{e}$}chet$\ $}}

\newcommand{\FM}{{\rm FM}}
\newcommand{\Lap}{{\rm Lap}}

\newcommand{\med}{{\rm med}}
\newcommand{\vMF}{{\rm vMF}}
\newcommand{\KL}{{\rm KL}}
\newcommand{\Img}{{\rm Img}}

\title{Riemannian radial distributions on Riemannian symmetric spaces: Optimal rates of convergence for parameter estimation}
\author{Hengchao Chen\footnote{Department of Statistical Sciences, University of Toronto, hengchao.chen@mail.utoronto.ca.}}
\date{\today}

\begin{document}

\maketitle
\doparttoc 
\faketableofcontents 

\begin{abstract}
    Manifold data analysis is challenging due to the lack of parametric distributions on manifolds. To address this, we introduce a series of Riemannian radial distributions on Riemannian symmetric spaces. By utilizing the symmetry, we show that for many Riemannian radial distributions, the Riemannian $L^p$ center of mass is uniquely given by the location parameter, and the maximum likelihood estimator (MLE) of this parameter is given by an M-estimator. Therefore, these parametric distributions provide a promising tool for statistical modeling and algorithmic design. 
    
    In addition, our paper develops a novel theory for parameter estimation and minimax optimality by integrating statistics, Riemannian geometry, and Lie theory. We demonstrate that the  MLE achieves a convergence rate of root-$n$ up to logarithmic terms, where the rate is quantified by both the hellinger distance between distributions and geodesic distance between parameters. Then we derive a root-$n$ minimax lower bound for the parameter estimation rate, demonstrating the optimality of the MLE. Our minimax analysis is limited to the case of simply connected Riemannian symmetric spaces for technical reasons, but is still applicable to numerous applications. Finally, we extend our studies to Riemannian radial distributions with an unknown temperature parameter, and establish the convergence rate of the MLE. We also derive the model complexity of von Mises-Fisher distributions on spheres and discuss the effects of geometry in statistical estimation. 
\end{abstract}

\section{Introduction}\label{sec:1}

Manifold data appear in various forms across numerous applications, such as spherical data in bioinformatics \citep{banerjee2005clustering,mardia2009directional}, hyperbolic data in network analysis \citep{krioukov2010hyperbolic}, shape data in medical imaging \citep{kendall1984shape}, and covariance matrices in~brain computer interface~\citep{barachant2011multiclass}. The analysis of these geometric objects, known as manifold data analysis, poses a significant challenge due to their inherent nonlinearity, which limits the applicability of traditional Euclidean methods. In response, researchers have devised novel methods using tools from Riemannian geometry. Notable examples include \Frechet mean \citep{bhattacharya2003large}, geodesic regression \citep{cornea2017regression,fletcher2013geodesic}, principal geodesic analysis \citep{fletcher2004principal},  mixture models \citep{banerjee2005clustering}, and principal nested spheres \citep{jung2012analysis}. Moreover, numerous metric space data analysis methods, which apply to manifold data, are also developed by researchers. Notable examples include \Frechet regression \citep{petersen2019frechet},  \Frechet change point detection \citep{dubey2020frechet}, and \Frechet analysis of variance \citep{dubey2019frechet}.

However, due to the lack of parametric distributions on manifolds, most manifold data analysis methods are nonparametric \citep{patrangenaru2016nonparametric,bhattacharya2012nonparametric}. Such limitation has substantially hindered the development of statistical modeling and algorithm design. In response, this paper introduces a family of Riemannian radial distributions on Riemannian homogeneous spaces\footnote{Examples include Euclidean spaces, spheres, hyperbolic spaces, real and complex projective spaces, tori, the spaces of symmetric positive definite matrices, Grassmann manifolds, Stiefel manifolds, and more.}, which generalize von Mises-Fisher distributions on spheres~\citep{mardia2009directional} and Riemannian Gaussian distributions on Riemannian symmetric spaces~\citep{said2017gaussian}. Similar to radial distributions on Euclidean spaces, the Riemannian radial distribution on a Riemannian homogeneous space $\cM$ posits a density function that is proportional to:
\$
f(x;\alpha,\phi)\propto\tilde f(x;\alpha,\phi)=e^{-\phi(d_g(x,\alpha))},\quad x\in\cM,
\$
where $\phi$ is a monotone increasing function, $\alpha\in\cM$ is the location parameter, and $d_g$ is the geodesic distance in $\cM$. By applying Riemannian geometry, we can demonstrate that this parametric distribution possesses multiple favorable properties.
\begin{itemize}
    \item First, the density $f(x;\alpha,\phi)$ depends only  on the distance $d_g(x,\alpha)$ and becomes smaller when $x$ moves away from $\alpha$. Thus, it is natural to model random noise using this distribution.  
    \item Second, by using the homogeneity of $\cM$, we can show that the normalizing constant of $f$ is a constant independent of the parameter $\alpha$. Consequently, the maximum likelihood estimator (MLE) of $\alpha$, based on $n$ independent samples $\{x_i\}_{i=1}^n$ drawn from $f$,  is given by the following M-estimator:
    \$
    \hat\alpha^{\MLE}=\argmin_{\alpha}\sum_{i=1}^n\phi(d_g(x_i,\alpha)).
    \$
    This estimator can be efficiently solved by Riemannian optimization methods. 
    
    \item Finally, when $\cM$ is a Riemannian symmetric space and $\phi$ is strictly increasing, we can show that the Riemannian  $L^p$ center of mass \citep{afsari2011riemannian} of the distribution $f$ is uniquely given by the location parameter~$\alpha$. This non-trivial property characterizes the symmetry of  Riemannian radial distributions on Riemannian symmetric spaces. 
\end{itemize}
These properties make this parametric family of distributions a promising choice for statistical modeling and algorithm design, such as building regression or mixture models, and designing differential privacy mechanisms.



Besides introducing Riemannian radial distributions, our paper also develops a novel theory to address parameter estimation and minimax optimality\footnote{Such theory does not exist in the literature even for Riemannian Gaussian distributions and  von Mises-Fisher distributions. Thus, as special cases, these gaps are filled by our theory.}. First, we derive the convergence rate of the MLE of $\alpha$ by using empirical process theory and Riemannian geometry. Unlike most existing works that impose entropy conditions~\citep{petersen2019frechet,dubey2020frechet}, our paper imposes bounded parameters conditions, derives entropy estimates using the volume comparison theorem, and then deduces the convergence rate of the MLE by applying empirical process theory and examining the parameter identifiability. We show that the convergence rate of the MLE is root-$n$ up to logarithmic terms for a broad range of Riemannian radial distributions. Such rate matches the Euclidean cases. 

To assess the optimality of the MLE's convergence rate, we establish a minimax lower bound~for parameter estimation\footnote{This analysis is restricted to simply connected Riemannian symmetric spaces for technical reasons.}. For a variety of Riemannian radial distributions, 
such as Riemannian Gaussian distributions and  von Mises-Fisher distributions, 
the derived minimax lower bound is root-$n$, confirming the MLE's optimality up to logarithmic terms. Our analysis integrates the Fano's method and   geometric tools. Specifically, for simply connected noncompact Riemannian symmetric spaces, we will use the Hessian comparison theorem from Riemannian geometry, and for simply connected compact Riemannian symmetric spaces, we will use an integral formula and a Hessian formula for the squared distance function from Lie theory.

As an extension, our paper also studies Riemannian radial distributions with an unknown temperature parameter and establishes the convergence rate of the MLE, which is  root-$n$ in a variety of contexts. Additionally, we investigate von Mises-Fisher distributions on the unit $m$-sphere $\SSS^m$ and establish its optimal parameter estimation rate as $\Theta(\frac{m}{\sqrt{n}})$, where $\Theta(\cdot)$ omits constants independent of $m$ and $n$, and the numerator $m$ in the rate indicates the model complexity.

\subsection{Organization}

The rest of  this paper is structured as follows. We review differential geometry and empirical process theory in Section \ref{sec:2}. Then we introduce Riemannian radial distributions on Riemannian homogeneous spaces, and discuss their basic properties in Section \ref{sec:3}. In Section \ref{sec:4}, we study the MLE of the parameter in Riemannian radial distributions, and derive the rates of convergence of such estimator, measured by both the hellinger distance between distributions and the geodesic distance between parameters. In Section \ref{sec:5}, we derive the minimax lower bounds for the parameter estimation rates. Section \ref{sec:6-temperature} examines Riemannian radial distributions with an unknown temperature, and establishes the convergence rate of the MLE. Section \ref{sec:6} delves into the model complexity of a specific class of Riemannian radial distributions: the von Mises-Fisher distributions on spheres. 
Concluding remarks are given in Section \ref{sec:conclusion}. Proofs are provided in the Appendix. 


\section{Preliminary}\label{sec:2}

\subsection{Differential geometry}

In this section, we introduce basic concepts of differential geometry. Familiarity with smooth manifolds is assumed in the subsequent sections. For those less acquainted with the topic, an introduction to smooth manifolds is available in Appendix \ref{sec:smoothmanifolds} and \citet{lee2012introduction}. In the sequel, we will first review Riemannian manifolds in Section \ref{sec:211} and then review integration on manifolds in Section \ref{sec:212}. Following that, we will review Riemannian symmetric spaces and discuss basic properties and examples. Some details will be given in the Appendix. In addition, for a more complete treatment of the subject, we refer readers to \citet{do1992riemannian,helgason2001differential,petersen2006riemannian,cheeger1975comparison}. 

\subsubsection{Riemannian manifolds}\label{sec:211}

A Riemannian manifold $(\cM,g^{\cM})$ is a smooth manifold $\cM$ endowed with a Riemannian metric $g^{\cM}$. The metric is a smoothly varying family of inner products $g^{\cM}_x: T_{x}\cM\times T_{x}\cM\to\RR$, where $T_{x}\cM$ is the tangent space of $\cM$ at $x$. The metric enables us to measure geometric quantities on $\cM$, including curve lengths, distances, volumes, and curvatures. 

Let $(\cM,g^{\cM})$ be a Riemannian manifold. An isometry of $\cM$ is a diffeomorphism $F:\cM\to \cM$ that preserves the Riemannian metric, i.e., $g^{\cM}_{F(x)}(dF_x(v),dF_x(w))=g^{\cM}_x(v,w)$ for all $v,w\in T_{x}\cM$ and $x\in\cM$, where $dF_x$ is the differential of $F$ at $x$. By definition, an isometry of $\cM$ preserves all geometric quantities determined by the Riemannian metric. We will denote the set of all isometries of $\cM$ by ${\rm Iso}(\cM)$.

A mapping $\gamma:[a,b]\to\cM$ is a piecewise smooth curve in $\cM$ if $\gamma$ is continuous and there is a partition $a=a_1<\ldots<a_k=b$ of $[a,b]$ such that $\gamma|_{[a_i,a_{i+1}]}$ is smooth for $i=1,\ldots,k-1$. Given a piecewise smooth curve $\gamma$ in $\cM$, its length is measured by integrating the norm of the tangent vectors along the curve. For any two points $x,y\in\cM$, the distance $d_g(x,y)$ is given by the minimum length over all possible piecewise smooth curves between $x$ and $y$. This $d_{g}$ defines a metric on $\cM$ and $(\cM,d_g)$ forms a metric space. We say the manifold $(\cM,g^{\cM})$ is complete if the metric space $(\cM,d_g)$ is complete. 

A curve in $\cM$ is a geodesic if it locally minimizes the length between points. The Hopf-Rinow theorem states that a connected manifold $\cM$ is complete if and only if all geodesics extend indefinitely. Moreover,  in a connected complete manifold, any two points are connected by a length-minimizing geodesic. 

For any point $x\in\cM$ and vector $v\in T_{x}\cM$, there is a unique geodesic $\gamma_{v}(t)$ with $x$ as its initial location and $v$ as its initial velocity. The exponential map ${\rm Exp}_{x}:T_{x}\cM\to\cM$ is defined by ${\rm Exp}_{x}(v)=\gamma_{v}(1)$ when $\gamma_{v}(1)$ exists. Suppose $\cM$ is a connected complete manifold. Then the exponential map is well-defined on the whole tangent space $T_{x}\cM$. 
The segment domain of ${\rm Exp}_x$ is defined as
\$
{\rm seg}(x)=\left\{v\in T_{x}\cM\mid
d_g(x,\gamma_{v}(t))=t\norm{v},\forall t\in[0,1]
\right\},
\$
where $\norm{v}$ is the norm of the vector $v$. The Hopf-Rinow theorem implies that $\cM={\rm Exp}_{x}({\rm seg}(x))$. The interior of ${\rm seg}(x)$,
\$
{\rm seg}^0(x)=\left\{ sv\in T_{x}\cM\mid s\in[0,1),v\in {\rm seg}(x)\right\},
\$
is an open star-shaped domain in $T_{x}\cM$. On ${\rm seg}^0(x)$, the exponential map $\Exp_x$ is a diffeomorphism and its inverse is called the logarithm map, denoted by $\Log_{x}$. The set ${\rm seg}(x)- {\rm seg}^0(x)$ is called the cut locus of $x$ in $T_{x}\cM$ and the set ${\rm Cut}(x)\coloneqq\cM-\Exp_x({\rm seg}^0(x))$ is called the  cut locus of $x$ in $\cM$. Since $\cM$ is complete, ${\rm Cut}(x)$ is closed with measure zero
in $\cM$, meaning that $\Exp({\rm seg}^0(x))$ covers $\cM$ except for a null set. The distance function $r_x(y)=d_g(x,y)$ is smooth at $y$ if and only if $y\notin \Cut(x)\cup\{x\}$. The injectivity radius at $x$ is defined as 
\$
{\rm inj}(x)\coloneqq d_g(x,{\rm Cut}(x))=\min_{y\in {\rm Cut}(x)}d_g(x,y).
\$
The injectivity radius of $\cM$ is defined as ${\rm inj}(\cM)=\inf_{x\in\cM}{\rm inj}(x)$.

\subsubsection{Integration on manifolds}\label{sec:212}

To introduce integration on manifolds, we first need the concept of volume density. Let $(\cM,g^{\cM})$ be an $m$-dimensional Riemannian manifold, and $A$ be a compact set contained in a local chart $(U,\varphi)$ with the coordinate $\varphi(y)=(\xk_{1}(y),\ldots,\xk_{m}(y))$. The volume of $A$ is defined to be
\$
{\rm vol}(A)=\int_{\varphi(A)}\sqrt{G}\circ\varphi^{-1}d\xk^1\cdots d\xk^m,
\$
where $G={\rm det}(g_{ij})$, $g_{ij}=g^{\cM}(\frac{\partial}{\partial \xk_i},\frac{\partial}{\partial \xk_j})$, and $d\xk^1\cdots d\xk^m$ is the Lebesgue measure on $\RR^{m}$. This definition is independent of the choice of the coordinate chart. To define the volume of a general set $A$ which needs not to be in one coordinate chart, we use the partition of unity argument. More precisely, we pick a locally finite family of coordinate charts $(U_{\alpha},\varphi_{\alpha},\xk_{\alpha,1},\ldots,\xk_{\alpha,m})$ with $\bigcup_{\alpha} U_{\alpha}=\cM$ and a partition of unity $\{\rho_{\alpha}\}$ subordinate to this family of  charts. We set 
\$
{\rm vol}(A)=\sum_{\alpha}\int_{\varphi_{\alpha}(A\cap U_{\alpha})}\rho_{\alpha}\sqrt{G_{\alpha}}\circ\varphi_{\alpha}^{-1}d\xk^1_{\alpha}\cdots d\xk^m_{\alpha},
\$
as long as each integral in the sum exists. This leads to the following definition.

\begin{definition}[Volume density]\label{def:dvol}
    The Riemannian volume density $\dvol$ on $(\cM,g^{\cM})$ is defined as
    \$
    \dvol=\sum_{\alpha}\rho_{\alpha}\sqrt{G_{\alpha}}\circ\varphi_{\alpha}^{-1}d\xk^1_{\alpha}\cdots d\xk^m_{\alpha}.
    \$
\end{definition}

\begin{remark}
    Definition \ref{def:dvol} is independent of the choice of the family of local charts $\{(U_\alpha,\varphi_\alpha)\}$ and the choice of the partition of unity $\{\rho_\alpha\}$.
\end{remark}

The Riemannian volume density $\dvol$ enables us to integrate functions on $\cM$. Let $C_{c}^0(\cM)$ be the set of compactly supported continuous functions on $\cM$. For any $f\in C_{c}^{0}(\cM)$, its integral is given by
\$
\int_{\cM}f\dvol=\sum_{\alpha}\int_{U_{\alpha}}\rho_{\alpha}f\sqrt{G_{\alpha}}\circ\varphi_{\alpha}^{-1}d\xk^1_{\alpha}\cdots d\xk^m_{\alpha}.
\$
Let $C_{c}^{\infty}(\cM)$ be the set of compactly supported differentiable functions on $\cM$. For any $1\leq p< \infty$, one can define the $L^{p}$ norm and the $L^{\infty}$ norm on $C_c^{\infty}(\cM)$ via
\$
\norm{f}_{L^{p}}=\left(\int_{\cM}|f|^p\dvol\right)^{1/p}, \quad \norm{f}_{\infty}=\sup_x|f(x)|,
\$
The completion of $C_{c}^{\infty}(\cM)$ under the $L^p$ norm is the $L^p$ space, $L^{p}(\cM)$. Similarly, one can define the $L^{\infty}$ space, $L^{\infty}(\cM)$. A function $f:\cM\to\RR$ is said to be $L^p$ integrable or bounded if its $L^p$ norm or $L^{\infty}$ norm is finite. Given two functions $f,h:\cM\to\RR$, the $L^p$ distance and $L^\infty$ distance between them are defined as $d_p(f,h)\coloneqq\norm{f-h}_{L^p}$ and $d_{\infty}(f,h)\coloneqq\norm{f-h}_{\infty}$, respectively. A density function on $\cM$ is a non-negative function with a unit $L^1$ norm. Given two density functions $f,h:\cM\to\RR$, the hellinger distance between them is given by
\$
d_h(f,h)\coloneqq\norm{\sqrt{f}-\sqrt{h}}_{L^2}=\left(\int_{\cM}(\sqrt{f}-\sqrt{h})^2\dvol\right)^{1/2}.
\$
The Kullbeck-Leibler (KL) divergence between two density functions $f$ and $h$ is given by
\$
D_{\KL}(f\ \|\ h)\coloneqq \int_{\cM}\ln\left(\frac{f}{h}\right)f\dvol.
\$
Our paper will use these distances and divergences in our statistical analysis on manifolds.

To evaluate integrals on manifolds, one can use the coordinate representation over a local chart. One typical choice is to use the normal coordinate system determined by the exponential map. Let $(\cM,g^{\cM})$ be an $m$-dimensional complete Riemannian manifold and $x\in\cM$. The normal coordinate chart at $x$ is given by $(\Exp_x({\rm seg}^0(x)),\Log_x)$, where ${\rm seg}^0(x)$ is the interior of the segment domain of $\Exp_x$. On the normal coordinate chart, the volume density $\dvol$ can be expressed as follows:
\$
\dvol=\sqrt{G}\circ\Exp_xd\xk^1\cdots d\xk^m,
\$
where $G={\rm det}(g_{ij})$, $g_{ij}=g^{\cM}(\frac{\partial}{\partial \xk_i},\frac{\partial}{\partial \xk_j})$, 
and $d\xk^1\cdots d\xk^m$ is the Lebesgue measure on $\RR^{m}$. Using the polar coordinates and replacing $(\xk^1,\ldots,\xk^m)$ by $(r,\Theta)$,
one can rewrite $\dvol$ as
\#
\dvol=\lambda(r,\Theta)drd\Theta,\quad\forall (r,\Theta)\in {\rm seg}^0(x),\label{equ:lambda}
\#
where $\lambda(r,\Theta)=r^{m-1}\sqrt{G}\circ\Exp_x$ is defined over ${\rm seg}^0(x)$ and $d\Theta$ is the usual surface measure on the unit sphere $\SSS^{m-1}$. 
For convenience, we set $\lambda(r,\Theta)=0$ outside ${\rm seg}^0(x)$. Since ${\rm Cut}(x)$ is of measure zero in $\cM$, for any real-valued function $f$, we have
\#
\int_{\cM}f\dvol&=\int_{\Exp_x({\rm seg}^0(x))}f\dvol=\int_{{\rm seg^0(x)}}f^{\flat}\lambda(r,\Theta)drd\Theta=\int_{T_{x}\cM}f^{\flat}\lambda(r,\Theta)drd\Theta,\label{equ:integral}
\#
where $f^{\flat}=f\circ\Exp_x$. To evaluate the above integral, one can use Theorem \ref{thm:2.3} to replace $\lambda(r,\Theta)$ with simpler functions. 

\begin{theorem}\label{thm:2.3}
    Let $(\cM,g^{\cM})$ be an $m$-dimensional complete Riemannian manifold. Suppose the sectional curvatures of $\cM$ lie within the interval $[\kappa_{\min},\kappa_{\max}]$. Let $\lambda(r,\Theta)$ be given by \eqref{equ:lambda}. Then for all $(r,\Theta)\in{\rm seg}^0(x)$, we have
    \$
    \sn_{\kappa_{\max}}^{m-1}(r)\leq\lambda(r,\Theta)\leq \sn_{\kappa_{\min}}^{m-1}(r),
    \$
    where $m$ is the dimension of $\cM$,  and 
    \begin{gather}\label{equ:sn}
    \sn_{\kappa}(r)=\left\{
    \begin{array}{lc}
        \frac{\sin(\sqrt{\kappa }r)}{\sqrt{\kappa}}\mathbbm{1}_{r\leq\frac{\pi}{\sqrt{\kappa}}}, & \textnormal{if }\kappa>0, \\
        r,& \textnormal{if }\kappa=0,   \\
        \frac{\sinh(\sqrt{-\kappa }r)}{\sqrt{-\kappa}}, &\textnormal{if }\kappa<0. 
    \end{array}
    \right.
    \end{gather}
    Since $\lambda(r,\Theta)=0$ outside ${\rm seg}^0(x)$, the inequality $\lambda(r,\Theta)\leq \sn_{\kappa_{\min}}^{m-1}(r)$ holds for all $(r,\Theta)$. 

\end{theorem}

\begin{proof}
    This immediately follows from Theorem 27, Chapter 6 in \citet{petersen2006riemannian}.
\end{proof}

Notably, when $\cM$ is of constant curvature $\kappa$, then $\lambda(r,\Theta)=\sn_{\kappa}^{m-1}(r)$ for all $(r,\Theta)\in\seg^0(x)$. Consequently, Theorem \ref{thm:2.3} essentially gives a volume comparison between a general Riemannian manifold and manifolds of constant curvatures. This theorem is useful throughout this paper.

\subsubsection{Riemannian symmetric spaces}\label{sec:213}

In this section, we will delve into a benign class of Riemannian manifolds called Riemannian symmetric spaces. A Riemannian manifold $(\cM,g^\cM)$ is said to be symmetric if for each $x\in\cM$, there is an isometry $F$ of $\cM$ fixing $x$ and acting on the tangent space $T_x\cM$ as minus the identity. Such isometry $F$ is called an involution as its square $F\circ F$ is the identity. Riemannian symmetric spaces are  homogeneous, thus satisfying all properties in Proposition \ref{prop:homogeneous}. 
Note that a Riemannian manifold $\cM$ is homogeneous if for any $x,y\in\cM$, there is an isometry $F\in \Iso(\cM)$ such that $F(x)=y$.





\begin{proposition}\label{prop:homogeneous}
    A Riemannian homogeneous space $(\cM,g^{\cM})$ satisfies the following properties.
    \begin{itemize}
        \item $\cM$ is a complete Riemannian manifold. 
        \item The sectional curvatures of $\cM$ lie within a bounded interval $[\kappa_{\min},\kappa_{\max}]$.
        \item The injectivity radius $\inj(\cM)$ of $\cM$ is larger than zero.

        \item There exists a constant $c_0>0$ such that for any $x,y\in\cM$ with $d_g(x,y)<c_0$, $x$ and $y$ are connected by a unique length-minimizing geodesic.

        \item Let $\dvol$ be the volume density on $\cM$ and $f\in L^1(\cM)$ be an integrable function. Then
        \#\label{equ:homo}
        \int_{\cM}f\dvol=\int_{\cM}f\circ F\dvol,\quad \forall F\in\Iso(\cM).
        \#
    \end{itemize}
\end{proposition}

In addition to the above properties, Riemannian symmetric spaces possess some distinct characteristics. Proposition \ref{prop:symmetric} provides one of such characteristics. This property is useful in examining the population $\cL^p$ center of mass of the Riemannian radial distributions on symmetric spaces. These are elaborated in Section \ref{sec:3}. 

\begin{proposition}\label{prop:symmetric}
    Suppose $(\cM,g^{\cM})$ is a Riemannian symmetric space. Then for any $x\neq y\in\cM$, there exists an isometry $F\in\Iso(\cM)$ such that $F(x)=y$ and $F\circ F$ is the identity. 
\end{proposition}

Simply connected Riemannian symmetric spaces admit more benign properties. First, any simply connected Riemannian symmetric space can be expressed as a product of a Hadamard Riemannian symmetric space\footnote{A Hadamard manifold is a complete and  simply connected Riemannian manifold that has nonpositive curvatures.} and a simply connected compact Riemannian symmetric space, as outlined in the following proposition.  It is thus natural to consider these two types of Riemannian symmetric spaces separately. 

\begin{proposition}
\label{thm:decomposition}

Let $\cM$ be a simply connected Riemannian symmetric space. Then $\cM$ is a product
\$
\cM=\cM_{H}\times \cM_{C},
\$
where $\cM_H$ is a Riemannian symmetric space that is also a Hadamard manifold, and $\cM_C$ is a simply connected compact Riemannian symmetric space. 
    
\end{proposition}

\begin{proof}
    This immediately follows from Proposition 4.2 and Theorem 3.1 in \citet{helgason2001differential}. 
\end{proof}

Hadamard Riemannian symmetric spaces and simply connected compact Riemannian symmetric spaces exhibit distinct geometries. 
Notably, Hadamard Riemannian symmetric spaces are diffeomorphic to Euclidean spaces and have empty cut loci. In contrast, compact Riemannian symmetric spaces have non-empty cut loci. The existence of such non-empty cut locus brings additional challenges to our minimax analysis in Section \ref{sec:5}. Advanced tools from Lie theories will be needed, and we postpone these materials to Appendix \ref{sec:apd-a2}.



Finally,  to conclude this section, we provide examples of Riemannian symmetric spaces (RSS) and Riemannian homogeneous spaces, along with discussions on their corresponding applications. These examples include Hadamard RSSs, simply connected compact RSSs, non-simply connected RSSs, and Riemannian homogeneous spaces that are not symmetric. By introducing these examples, we aim to provide readers with a deeper understanding of the applicability of our theories. 

\begin{example}
The following are examples of Hadamard RSSs.
\begin{itemize}
    \item The Euclidean space $\RR^m$ is a Hadamard RSS. The space of symmetric positive definite (SPD) matrices endowed with the log-Euclidean metric is an example of the Euclidean spaces~\citep{arsigny2007geometric}. It has applications in medical imaging \citep{arsigny2006log}, brain connectivity~\citep{you2021re}, and computer vision~\citep{huang2015log}.
    
    \item The hyperbolic space $\HH^m$ is a Hadamard RSS. It has been used to model complex networks \citep{krioukov2010hyperbolic} or word embeddings~\citep{nickel2017poincare}, because it can naturally fit an exponentially growing number of nodes. 

    \item The space of SPD matrices endowed with the affine-invariant metric is a Hadamard RSS which is non-Euclidean \citep{moakher2005differential,said2017riemannian,terras2012harmonic}. It is useful in medical imaging \citep{pennec2006riemannian}, computer vision \citep{tuzel2008pedestrian}, and brain-computer interface~\citep{barachant2011multiclass}.
    
    \item The product space of several Hadamard RSSs is  a Hadamard RSS.
\end{itemize}
    
\end{example}

\begin{example}The following are  examples of simply connected compact RSSs.
    \begin{itemize}
        \item The $m$-sphere $\SSS^m$ ($m\geq 2$) is a simply connected compact RSS. It is well-suited to model $\ell_2$ normalized feature vectors and has received much attention in various fields~\citep{mardia2009directional}. 
        
        \item The complex projective space $\CC\PP^m$ is a simply connected compact RSS. It is well-suited for  modeling 2D landscape shape data  \citep{kendall1984shape} and is applicable in archaeology, medical imaging, biology, and more \citep{dryden2016statistical}.  

        \item The complex Grassmann manifold ${\rm Gr}(k,\CC^m)$, representing $k$-dimensional subspaces in $\CC^m$, is a simply connected compact RSS. 

        \item The product space of several simply connected compact RSSs is  a simply connected compact RSS.
    \end{itemize}
\end{example}

\begin{example}The following are examples of non-simply connected RSSs.
    \begin{itemize}
        \item The circle $\SSS^1$ and $m$-dimensional torus $\TT^m$ are non-simply connected RSSs. 
        \item The real projective space $\RR\PP^m$ is a non-simply connected RSS, which is useful in axial data analysis~\citep{bhattacharya2005large}.

        \item The Grassmann manifold ${\rm Gr}(k,\RR^m)$, representing $k$-dimensional subspaces in $\RR^m$, is a non-simply connected RSS when $k\geq 1$ and $m-k>1$. It is useful in Riemannian optimization \citep{edelman1998geometry} and computer vision \citep{turaga2008statistical}.
        \item The product space of a non-simply connected RSS with any RSS is a non-simply connected RSS. 
    \end{itemize}
\end{example}

\begin{example}
    The following are  examples of Riemannian homogeneous spaces that are not symmetric.
    \begin{itemize}
        \item Stiefel manifolds are Riemannian homogeneous spaces but not symmetric. They are suitable for modeling frames and are useful in Riemannian optimization \citep{edelman1998geometry}.

        \item The product space of  homogeneous spaces is still homogeneous but not necessarily symmetric.
    \end{itemize}
\end{example}

    

    

\subsection{Empirical process theory}

Besides geometry, this paper also uses empirical process theory \citep{van1996weak} and minimax analysis \citep{wainwright2019high} to derive the optimal rates of convergence. Both analyses involve the concepts of entropies introduced in this section. The first two concepts, metric entropy and packing number, are defined for all metric spaces. 
Essentially, they   measure the same complexity as shown in Lemma \ref{lma:b1}. The metric entropy is  typically useful in establishing the upper bound and the packing number is useful in deriving the lower bound.  

\begin{definition}[Metric entropy]
    Let $(\cX,d)$ be a metric space and $\cA\subseteq \cX$. A $\delta$-cover of $\cA$ is a set $\{x_i\}_{i=1}^N\subseteq \cX$ such that for any $y\in\cA$, there exists an $i\in [N]$ such that $d(x_i,y)\leq \delta$. The $\delta$-covering number $\cN(\delta,\cA,d)$ is the cardinality of the smallest $\delta$-cover of $\cA$. The $\delta$-metric entropy is the logarithm of the $\delta$-covering number, denoted by $H(\delta,\cA,d)=\log\cN(\delta,\cA,d)$.  
\end{definition}

\begin{definition}[Packing number]
    Let $(\cX,d)$ be a metric space and $\cA\subseteq\cX$. A $\delta$-packing of $\cA$ is a set $\{x_i\}_{i=1}^N\subseteq \cA$ such that $d(x_{i},x_j)>\delta$ for all distinct $i,j\in[N]$. The $\delta$-packing number $M(\delta,\cA,d)$ is the cardinality of the largest $\delta$-packing. 
\end{definition}

\begin{lemma}[Lemma 5.5 in \cite{wainwright2019high}]\label{lma:b1}
    For all $\delta>0$, the packing number and covering numbers are related as follows
    \$
    M(2\delta,\cA,d)\leq \cN(\delta,\cA,d)\leq M(\delta,\cA,d).
    \$
\end{lemma}

Besides metric entropy, to upper bound the convergence rate of an M-estimator, we also need to introduce  bracketing entropy. This is defined for a functional space $\cF=\{f:\cM\to\RR\}$ equipped with a distance function $d$. 

\begin{definition}[Bracketing entropy]
    Consider the functional space $\cF$ equipped with the distance function $d$.  Given two functions $l,u\in\cF$ with $l\leq u$, the bracket $[l,u]$ is the set of all functions $f\in\cF$ with $l\leq f\leq u$. An $\epsilon$-bracket is a bracket $[l,u]$ with $d(l,u)\leq \epsilon.$ Let $\cG\subseteq\cF$ be some functional class. The $\epsilon$-bracketing number $\cN_{B}(\epsilon,\cG,d)$ is the minimum number of $\epsilon$-brackets needed to cover $\cG$. The bracketing entropy is the logarithm of the bracketing number, denoted by $H_{B}(\epsilon,\cG,d)=\log\cN_{B}(\epsilon,\cG,d)$.
\end{definition}



A key aspect of our paper is  establishing desired bounds on specific entropies, which yields the desired convergence rates. Our analysis integrates statistics, differential geometry, and Lie theory, offering a deeper understanding on  manifold data analysis.

\section{Riemannian radial distributions}\label{sec:3}

In this section, we will define Riemannian radial distributions on Riemannian homogeneous spaces. We will begin by establishing the conditions under which these distributions are well-defined. Then we will explore fundamental properties of Riemannian radial distributions and give some examples. Our findings will show that Riemannian radial distributions form a promising parametric family of distributions, meriting deeper exploration. 

To proceed, we let $(\cM,g^{\cM})$ be a Riemannian homogeneous space. Given a point $\alpha\in\cM$ and a continuous, increasing function $\phi:[0,\infty)\to[0,\infty)$, we can define the following  function on $\cM$:
\#\label{equ:3.1}
\tilde f(x;\alpha,\phi)=e^{-\phi(d_g(x,\alpha))},\quad \forall x\in\cM,
\#
where $d_g(\cdot,\cdot)$ is the geodesic distance on $\cM$. Suppose $\tilde f$ is integrable on $\cM$, then we can normalize it to obtain a density function $f(x;\alpha,\phi)$ on $\cM$. Such density function only depends on the distance $d_g(x,\alpha)$ between $x$ and $\alpha$, and decreases as the distance $d_g(x,\alpha)$ increases. Thus, we refer to such density as a Riemannian radial distribution. The following proposition presents sufficient conditions under which $\tilde f$ is integrable and hence the Riemannian radial distribution is well-defined. 

\begin{proposition}\label{prop:3.1}
    Suppose $(\cM,g^{\cM})$ is a connected Riemannian homogeneous space with sectional curvatures bounded within $[\kappa_{\min},\kappa_{\max}]$. Then the function $\tilde f$, as defined in \eqref{equ:3.1}, is integrable over $\cM$ in any  of the following cases: 
    \begin{enumerate}
        \item[(1)] $\cM$ is compact.
        \item[(2)] $\cM$ is noncompact and $\phi$ satisfies
    \#\label{equ:3.2}
    \int_0^{\infty}e^{-\phi(r)}\cdot \sn_{\kappa_{\min}}^{m-1}(r)dr<\infty,
    \#
    where $m$ is the dimension of $\cM$ and $\sn_{\kappa_{\min}}(\cdot)$ is given by \eqref{equ:sn}. 
    \end{enumerate} 
\end{proposition}

\begin{remark}
    It is worth discussing the second case where $\cM$ is noncompact. In this case, $\kappa_{\min}$ is nonpositive, ensuring that $\sn_{\kappa_{\min}}(\cdot)$ is well-defined over $\RR_+$ and condition \eqref{equ:3.2} is well-defined. If $\kappa_{\min}=0$, then condition \eqref{equ:3.2} reduces to 
    \$
    \int_0^\infty e^{-\phi(r)}r^{m-1}dr<\infty.
    \$
    Conversely, if $\kappa_{\min}<0$,   condition \eqref{equ:3.2} becomes
    \$
    \int_0^\infty e^{-\phi(r)}\cdot e^{\sqrt{-\kappa_{\min}}(m-1)r}dr<\infty.
    \$
    The differences between these two cases are significant. For instance, let us consider $\phi(r)=\beta r^q$ for some constant $\beta>0$ and $q>0$. When $\kappa_{\min}=0$, condition \eqref{equ:3.2} is satisfied for all $q>0$. However, for $\kappa_{\min}<0$, this condition only holds for $q\geq 1$. Moreover, when $q=1$, this condition is only met if $\beta>\sqrt{-\kappa_{\min}}(m-1)$. These sharp differences between the cases $\kappa_{\min}=0$ and $\kappa_{\min}<0$ stem from the varying  rates of volume growth of balls in different manifolds $\cM$, highlighting the impact of geometry. 
    

\end{remark}

\begin{condition}
\label{cond:3.1}
    We assume the following conditions hold.
    \begin{enumerate}
        \item[(1)] $(\cM,g^{\cM})$ is a connected Riemannian homogeneous space and we denote its maximum radius  by $r_{\cM}=\sup_{x,y\in\cM}d_g(x,y)$. 
        \item[(2)] $\phi$ is a nonnegative, increasing, and continuous function defined over $[0,r_{\cM}]$.\footnote{Throughout this paper,  the interval $[0,r_{\cM}]$ is always interpreted as $[0,\infty)$ when $r_{\cM}=\infty$.} 
        \item[(3)] $\tilde f(x;\alpha,\phi)$, as defined by \eqref{equ:3.1},
    is integrable over $\cM$ for some $\alpha\in\cM$.
    \end{enumerate}  
\end{condition}

Going forward, we will always assume that Condition \ref{cond:3.1}  is satisfied. Under this condition, we will first establish some basic properties of a Riemannian radial distribution in Proposition \ref{prop:3.2}. In this proposition, we show that if $\tilde f(x;\alpha,\phi)$ is integrable over $\cM$ for a point $\alpha\in\cM$, then $\tilde f(x;\alpha,\phi)$ is integrable over $\cM$ for all $\alpha\in\cM$. Moreover, we show that the integral of $\tilde f(x;\alpha,\phi)$ is a constant independent of $\alpha$, which we denote by
\#\label{equ:Z}
Z(\phi)=Z(\alpha,\phi)\coloneqq\int_{\cM}\tilde f(x;\alpha,\phi)\dvol(x). 
\#
By normalizing $\tilde f$, we can obtain the  density function of a Riemannian radial distribution as follows:
\#\label{equ:f}
f(x;\alpha,\phi)=\frac{1}{Z(\phi)}e^{-\phi(d_g(x,\alpha))},\quad\forall x\in\cM.
\# 
Suppose $\phi$ is fixed and known, then the MLE for $\alpha$, based on $n$ independent samples $\{x_i\}_{i=1}^n$ from $f$, is given by the following M-estimator
\#\label{equ:MLE}
\alpha^{\MLE}=\argmin_{\alpha\in\cM}\sum_{i=1}^n\phi(d_g(x_{i},\alpha)).
\#
This property follows from the fact that the normalizing constant $Z(\phi)$ is independent of $\alpha$.

\begin{proposition}\label{prop:3.2}
    Assuming Condition \ref{cond:3.1} is satisfied and using the  same notation, we have
    \begin{enumerate}
        \item[(1)] $\tilde f(x;\tilde \alpha,\phi)$ is integrable over $\cM$ for all $\tilde \alpha \in\cM$. 

        \item[(2)] The normalizing constant $Z(\alpha,\phi)=\int_\cM \tilde f(x;\alpha,\phi)\dvol(x)$ is independent of $\alpha$. 

        \item[(3)] Let $f(x;\alpha,\phi)$ be the density function given by \eqref{equ:f}, where $\phi$ is considered fixed and known. Then the MLE of $\alpha$, based on $n$ independent samples $\{x_i\}_{i=1}^n$ from $f$,  is given by  \eqref{equ:MLE}.
    \end{enumerate}
\end{proposition}



Proposition \ref{prop:3.2} is rooted in the homogeneity of $\cM$, yet it does not fully capture the symmetry of $\cM$ when $\cM$ is a Riemannian symmetric space. Indeed, if $\cM$ is a Riemannian symmetric space, we can establish more favorable properties. To see this, recall that the $L^p$ center of mass with respect to a density function $f$ on $\cM$ is given by
\#\label{equ:Lp}
\alpha^{p}=\argmin_{\alpha\in\cM}\int_{\cM}d_g^p(x,\alpha)f(x)\dvol(x).
\#
Then in Proposition \ref{prop:3.3}, we establish that under mild conditions, the $L^p$ center of mass with respect to the density  $f(x;\alpha,\phi)$ is uniquely given by the location parameter $\alpha$. 

\begin{proposition}\label{prop:3.3}
    Assume $(\cM,g^{\cM})$ is a Riemannian symmetric space and Condition \ref{cond:3.1} holds. Let $f(x;\alpha,\phi)$ be the density given by \eqref{equ:f} and $p\in(0,\infty)$. If $\phi$ is strictly increasing and the integral 
    \$
    \int_{\cM}d^p_g(x,\alpha) f(x;\alpha,\phi)\dvol(x)
    \$
    is finite, then the $L^p$ center of mass with respect to $f(x;\alpha,\phi)$ is uniquely given by the parameter $\alpha$. 
    
\end{proposition}

\begin{remark}
    Proposition \ref{prop:3.3} is a remarkable result, as the uniqueness of the $L^p$ center of mass with respect to a distribution on a manifold is typically nontrivial. Previous results either assume that the manifold is Hadamard or the distribution is supported in a local area~\citep{afsari2011riemannian,karcher1977riemannian}. In sharp contrast, our results drop the assumption of a localized support and cover all Riemannian symmetric spaces. Our results hold for a wide class of Riemannian radial distributions, highlighting that symmetric distributions on symmetric spaces can possess a unique $L^p$ center of mass. One may extend such high-level ideas to the study of Gibbs distributions on symmetric spaces,  and enrich the results in \cite{said2021riemannian}.
\end{remark}

Proposition \ref{prop:3.2} and  \ref{prop:3.3} suggest that  Riemannian radial distributions form a promising parametric class of distributions on Riemannian homogeneous spaces. Therefore, exploring their statistical properties is of considerable importance. One of the most important statistical problems is to estimate the parameter $\alpha$ given $n$ independent samples $\{x_i\}_{i=1}^n$ from $f(x;\alpha,\phi)$. While we can employ the MLE in \eqref{equ:MLE}, it remains unknown, due to the non-Euclidean geometry, 
\begin{enumerate}
    \item[(1)] What is the convergence rate of the MLE in terms of the sample size $n$? 
    \item[(2)] Is the convergence rate achieved by the MLE optimal in the minimax sense?
\end{enumerate}
Answering these two questions will be the main focus of this paper. We will derive the convergence rate of the MLE in Section \ref{sec:4}, and investigate the minimax lower bound on the parameter estimation rate in Section \ref{sec:5}. Together, we will obtain the optimal rate of convergence for parameter estimation and demonstrate the optimality of the MLE in a wide range of contexts.



Before moving on, let us present several examples of Riemannian radial distributions, including Riemannian Gaussian, Laplacian, and uniform distributions on Riemannian homogeneous spaces, and the von Mises-Fisher distributions on spheres. Keeping these examples in mind will be beneficial, as they can enhance understanding of the theories developed in the following sections.


\begin{example}\label{eg:3.5}
    Here are examples of Riemannian radial distributions and associated properties. 
    \begin{enumerate}
        \item[(1)] The Riemannian Gaussian distribution is  the Riemannian radial distribution with $\phi(r)=\beta r^2$ for some constant $\beta>0$.
        The density of such distribution is given by
        \$
        f_{N}(x;\alpha,\beta)=\frac{1}{Z_{N}(\beta)}e^{-\beta d_g^2(x,\alpha)},\quad\textnormal{where}\ Z_{N}(\beta)=\int_{\cM}e^{-\beta d_g^2(x,\alpha)}\dvol(x).
        \$
        Given $n$ independent samples $\{x_i\}_{i=1}^n$ from $f_N$, the MLE of $\alpha$ is the sample \Frechet mean:
        \$
        \alpha^{\FM}=\argmin_{\alpha\in\cM}\sum_{i=1}^nd_g^2(x_i,\alpha).
        \$

        \item[(2)] The Riemannian Laplacian distribution is the Riemannian radial distribution with $\phi(r)=\beta r$ for some sufficiently large $\beta>0$. The density of such distribution is given by
        \$
        f_{\Lap}(x;\alpha,\beta)=\frac{1}{Z_{\Lap}(\beta)}e^{-\beta d_g(x,\alpha)},\quad \textnormal{where}\ Z_{\Lap}(\beta)=\int_\cM e^{-\beta d_g(x,\alpha)}\dvol(x).
        \$
        Given $n$ independent samples $\{x_i\}_{i=1}^n$ from $f_{\Lap}$, the MLE of $\alpha$ is the sample \Frechet median:
        \$
        \alpha^{\med}=\argmin_{\alpha\in\cM}\sum_{i=1}^nd_g(x_i,\alpha).
        \$


        \item[(3)] When $\cM$ is a compact Riemannian homogeneous space and $\phi(r)=0$ for all $r$, the associated Riemannian radial distribution is called the uniform distribution. Such distribution $f(x;\alpha,\phi)$ remains identical for all $\alpha\in\cM$, rendering the estimation of $\alpha$ impossible.
    \end{enumerate}
\end{example}

\begin{example}\label{eg:3.6}
    The von Mises-Fisher distribution  on a unit $m$-sphere $\SSS^m$ is determined by the following density function \citep{mardia2009directional}:
    \$
    f_{\vMF}(x;\alpha,\beta)=\frac{1}{Z_{\vMF}(\beta)}e^{\beta\inner{x}{\alpha}},\quad Z_{\vMF}(\beta)=\int_{\SSS^m}e^{\beta\inner{x}{\alpha}}\dvol(x),\quad \beta>0,
    \$
    where a point in $\SSS^m$ is denoted by an $\ell_2$ normalized vector $x\in\RR^{m+1}$ and $\inner{\cdot}{\cdot}$ is the inner product in $\RR^{m+1}$. It is easy to show that $\inner{x}{\alpha}=\cos d_g(x,\alpha)$, where $d_g$ is the geodesic distance on $\SSS^m$. Thus, the von Mises-Fisher distribution is a special case of the Riemannian radial distribution with $\phi(r)=\beta(1-\cos r)$, $r\in[0,\pi]$ and some $\beta>0$.
\end{example}

\section{Maximum likelihood estimation}\label{sec:4}

Consider a Riemannian homogeneous space $(\cM,g^{\cM})$ and assume Condition \ref{cond:3.1} holds. Let $f(x;\alpha,\phi)$ be the density in \eqref{equ:f} with $\phi$ being fixed and known. The goal of this section is to derive the convergence rate for the MLE of $\alpha$, based on $n$ independent samples $\{x_i\}_{i=1}^n$ from $f$. Since the MLE is an M-estimator, we will use empirical process theory to derive its convergence rate. Our first step is to give an entropy estimate for the studied functional class in Section \ref{sec:4.1}. Then, we will use empirical process theory to derive the convergence rate of the MLE in Section \ref{sec:4.2}, where the rate is measured by the hellinger distance between the true distribution and the estimated distribution. After that, we will examine the identifiability of the parameter and derive the parameter estimation rate in Section \ref{sec:4.3}.







\subsection{Entropy estimates}\label{sec:4.1}

Our first main theorem aims to establish entropy estimates for the following functional class
\#\label{equ:F}
\cF=\{f(x;\alpha,\phi)\mid\alpha\in\cB_{\cM}(\alpha^*,D)\},
\#
where $\phi$ is considered fixed and known, and $\cB_{\cM}(\alpha^*,D)$ is the geodesic ball defined by
\#\label{equ:ball}
\cB_{\cM}(\alpha^*,D)=\{\alpha\in\cM\mid d_g(\alpha,\alpha^*)<D\}.
\#
Our analysis requires $\alpha$ to be bounded, a standard condition even in the Euclidean cases. Moreover, for a valid entropy estimate, we impose additional assumptions on $\phi$. 


\begin{condition}\label{cond:4.1}
    Let $(\cM,g^{\cM})$ be an $m$-dimensional Riemannian homogeneous space  with  $r_{\cM}=\sup_{x,y\in\cM}d_g(x,y)$.  Let $\phi$ be a continuous function on $[0,r_\cM]$ satisfying Condition \ref{cond:3.1}. We assume the following conditions hold.
\begin{enumerate}
        \item[(1)] The function $e^{-\phi(r)}$ is Lipschitz continuous on $[0,r_\cM]$ with a Lipschitz constant $L$.
        \item[(2)]  When $\cM$ is noncompact, the following function 
        \$
        h(r)=e^{-\phi(r/2)}\sn_{\kappa_{\min}}^{m-1}(r)
        \$ 
        is integrable over $[0,\infty)$, where $\kappa_{\min}<0$ is a lower bound on the sectional curvatures of $\cM$ and $\sn_{\kappa_{\min}}(\cdot)$ is given by \eqref{equ:sn}. Moreover, for all sufficiently small $\eta$, 
    \$
    \int_{B_\eta}^{\infty}h(r)dr\leq c_1 \eta^{c_2},
    \$
   where $B_\eta=\frac{\log(1/\eta)}{2(m-1)\sqrt{-\kappa_{\min}}}$ and $c_1,c_2>0$ are constants independent of $\eta$.
    \end{enumerate} 
\end{condition}

\begin{remark}\label{remark:4.2}
    A variety of functions $\phi$ satisfy Condition \ref{cond:4.1}. For example, $\phi(r)=\beta r^{p}$ with $p>1$ and $\beta>0$, and $\phi(r)=\beta r$ with a sufficiently large $\beta$  all satisfy Condition \ref{cond:4.1}. When $\cM=\SSS^m$ is the unit $m$-sphere,  $\phi(r)=\beta(1-\cos(r))$ also satisfies this condition. 
    Thus, 
    subsequent results, such as Theorem \ref{thm:4.1} and Corollary \ref{corollary:4.2}, hold for Riemannian radial distributions associated with these $\phi$.  
\end{remark}

Under Condition \ref{cond:4.1}, we can derive the entropy estimates of $\cF$ as follows. 

\begin{theorem}\label{thm:4.1}
    Let $(\cM,g^{\cM})$ be a Riemannian homogeneous space and $\phi$ a function such that Condition \ref{cond:4.1} holds. Let $f(x;\alpha,\phi)$ be the density in \eqref{equ:f} and $\cF$ be the functional class in \eqref{equ:F}. Then for sufficiently small $\epsilon$, we have
    \$
    \log\cN(\epsilon,\cF,d_{\infty})&\lesssim\log(\frac{1}{\epsilon}),\\
    \log\cN_B(\epsilon,\cF,d_{1})&\lesssim\log(\frac{1}{\epsilon}),\\
    \log\cN_{B}(\epsilon,\cF,d_h)&\lesssim\log(\frac{1}{\epsilon}),
    \$
    where $d_{\infty}$, $d_1$, $d_h$ represent the $L^{\infty}$, $L^1$, and hellinger distance, respectively, and $\lesssim$ omits constants independent of $\epsilon$. 
\end{theorem}

\begin{proof}
    We will prove this theorem in three steps. First, we construct an $\epsilon$-net $\cS=\{\alpha_i\}_{i=1}^N$ of the set $\cX=\cB_{\cM}(\alpha^*,D)$. By Lemma \ref{lma:c1}, this $\epsilon$-net $\cS$ can be chosen such that $N\lesssim \epsilon^{-m}$ for sufficiently small $\epsilon$, where $\lesssim$ omits constants independent of $\epsilon$. 

    Next, we use $\cS$ to construct the following net
    \#\label{equ:SF}
    \cS_{\cF}=\{f(x;\alpha_i,\phi)\mid\alpha_i\in\cS\}.
    \#
    For any $\alpha\in\cX$, there is an $\alpha_i\in\cS$ such that $d_g(\alpha,\alpha_i)\leq \epsilon$. By the Lipschitz property of $f(x;\alpha,\phi)$ with respect to its parameter $\alpha$, i.e., Lemma \ref{lma:c2}, we have
    \$
    d_\infty(f(x;\alpha,\phi),f(x;\alpha_i,\phi))\leq C\epsilon,
    \$
    where $C$ is a constant independent of $\epsilon$. Consequently, $\cS_{\cF}$ in \eqref{equ:SF} is a $C\epsilon$-net of $\cF$, defined in \eqref{equ:F}, and  
    \$
    \cN(C\epsilon,\cF,d_\infty)\leq |\cS_{\cF}|\lesssim\epsilon^{-m}.
    \$
    By rescaling $\epsilon$, we obtain the following metric entropy estimate of $\cF$:
    \$
    \log\cN(\epsilon,\cF,d_{\infty})\lesssim\log(\frac{1}{\epsilon}),
    \$
    where $\lesssim$ omits constants independent of $\epsilon$.
    This proves the first entropy inequality in the theorem. 

    To proceed, we analyze the bracketing entropy $\log\cN_{B}(\epsilon,\cF,d_1)$ and $\log\cN_B(\epsilon,\cF,d_h)$. 
    Observe that for any $f(x;\alpha,\phi)\in\cF$, it holds that
    \#\label{equ:4.4}
    0\leq f(x;\alpha,\phi)\leq \frac{1}{Z(\phi)},\quad\forall x\in\cM.
    \#
    In addition, if $d_{g}(x,\alpha^*)\geq 2D$, then for all $\alpha\in\cB_{\cM}(\alpha^*,D)$, we have
    \#\label{equ:4.5}
    d_g(x,\alpha)\geq d_g(x,\alpha^*)-d_g(\alpha,\alpha^*)\geq \frac{d_g(x,\alpha^*)}{2}.
    \#
    Consequently, for all $x$ with $d_g(x,\alpha^*)\geq 2D$, we have
    \#\label{equ:4.6}
    f(x;\alpha,\phi)\leq \frac{1}{Z(\phi)}e^{-\phi(d_g(x,\alpha^*)/2)},
    \#
    where we use \eqref{equ:4.5} and the fact that $\phi$ is increasing. Combining \eqref{equ:4.4} and \eqref{equ:4.6}, we conclude that
    \$
    H(x)=\left\{
    \begin{array}{ll}
        \frac{1}{Z(\phi)}e^{-\phi(d_g(x,\alpha^*)/2)}, &\textnormal{if }d_g(x,\alpha^*)>2D,  \\
        \frac{1}{Z(\phi)}, &\textnormal{otherwise}, 
    \end{array}\right.
    \$
    is an envelope for $\cF$, that is, $f(x;\alpha,\phi)\leq H(x)$ for all $x\in\cM$ and $\alpha\in\cX$. Now construct brackets $[l_i,u_i]$ as follows:   
    \$
    l_i=\max\{f_i-\eta,0\},\quad u_i=\min\{f_i+\eta,H\},
    \$
    where $\{f_i\}_{i=1}^N$ is the $\eta$-net of $\cF$ under $d_\infty$. It is clear that $\cF\subseteq\cup_{i=1}^N[l_i,u_i]$ and 
    \$
    u_i-l_i\leq \min\{2\eta,H\}.
    \$
    As a result, for any $B>0$, 
    \#\label{equ:4.7}
    d_1(u_i,l_i)\coloneqq \int_{\cM}(u_i-l_i)\dvol \leq 2\eta\cdot\vol(\cB_{\cM}(\alpha^*,B))+\int_{d_g(x,\alpha^*)> B}H(x)\dvol(x).
    \#
    To analyze this upper bound, we consider two cases, $\cM$ is compact or noncompact, separately.
    \begin{itemize}
        \item  {\bf Case 1: $\cM$ is compact.} In this case, by taking $B>\sup_{x\in\cM}d_g(x,\alpha^*)$, we obtain
    \$
    d_1(u_i,l_i)\leq 2\eta\cdot\vol(\cM).
    \$
    This implies that for sufficiently small $\eta$,
    \$
    \cN_B(2\eta\cdot\vol(\cM),\cF,d_1)\leq \cN(\eta,\cF,d_{\infty})\lesssim \eta^{-m}.
    \$
    where $\lesssim$ omits constants independent of $\eta$.
    By taking $\epsilon=2\eta\cdot\vol(\cM)$, we obtain that for sufficiently small $\epsilon$,
    \$
    \log\cN_B(\epsilon,\cF,d_1)\lesssim \log(\frac{1}{\epsilon}),
    \$
    where $\lesssim$ omits constants independent of $\epsilon$. Since $\cN_B(\sqrt{\epsilon},\cF,d_h)\leq\cN_B(\epsilon,\cF,d_1)$, we have 
    \$
    \log\cN_B(\epsilon,\cF,d_h)\lesssim \log(\frac{1}{\epsilon})
    \$ 
    for sufficiently small $\epsilon$, where $\lesssim$ omits constants independent of $\epsilon$.

    \item {\bf Case 2: $\cM$ is noncompact.} In this case, we take $B>2D$. The upper bound in \eqref{equ:4.7} is then reduced to
    \$
    d_1(u_i,l_i)\leq 2\eta\cdot\vol(\cB_{\cM}(\alpha^*,B))+\frac{1}{Z(\phi)}\int_{d_g(x,\alpha^*)>B}e^{-\phi(d_g(x,\alpha^*)/2)}\dvol(x).
    \$
    Using the polar coordinate expression \eqref{equ:integral} of the integral, we obtain that
    \#\label{equ:4.8}
    d_1(u_i,l_i)\leq2\eta\cdot\vol(\cB_{\cM}(\alpha^*,B))+\frac{1}{Z(\phi)}\int_{\SSS^{m-1}}\int_{B}^{\infty}e^{-\phi(r/2)}\lambda(r,\Theta)drd\Theta.
    \#
    Let $\kappa_{\min}<0$ be a lower bound on the sectional curvatures of $\cM$ that is used in Condition \ref{cond:4.1}, and $m$ the dimension of $\cM$. Then by the volume comparison theorem, Theorem \ref{thm:2.3}, we have
    \#
    \vol(\cB_{\cM}(\alpha^*,B))&\leq \vol(\SSS^{m-1})\cdot \int_{0}^B\sn_{\kappa_{\min}}^{m-1}(r)dr\notag\\
    &\leq C_1e^{\sqrt{-\kappa_{\min}}(m-1)B},\label{equ:4.9}
    \#
    where $\sn_{\kappa_{\min}}(\cdot)$ is given by \eqref{equ:sn} and $C_1$ is a constant independent of $B$. Furthermore, by Theorem \ref{thm:2.3}, we have
    \#
    \int_{\SSS^{m-1}}\int_{B}^{\infty}e^{-\phi(r/2)}\lambda(r,\Theta)drd\Theta&\leq \vol(\SSS^{m-1})\cdot\int_{B}^{\infty}e^{-\phi(r/2)}\sn_{\kappa_{\min}}^{m-1}(r)dr\notag\\
    &\leq C_2\int_{B}^{\infty}h(r)dr,\label{equ:4.10}
    \#
    where $h(r)=e^{-\phi(r/2)}\sn_{\kappa_{\min}}^{m-1}(r)$ and $C_2$ is a constant independent of $B$. By combining \eqref{equ:4.8}, \eqref{equ:4.9}, and \eqref{equ:4.10}, we obtain 
    \$
    d_1(u_i,l_i)\leq 2C_1\eta e^{\sqrt{-\kappa_{\min}}(m-1)B}+\frac{C_2}{Z(\phi)}\int_B^\infty h(r)dr.
    \$
    Taking $B=B_\eta=\frac{\log(1/\eta)}{2(m-1)\sqrt{-\kappa_{\min}}}$, and using Condition \ref{cond:4.1}, we conclude that for all sufficiently small $\eta$, 
    \$
    d_1(u_i,l_i)\leq 2C_1\eta^{1/2}+\frac{c_1C_2}{Z(\phi)}\eta^{c_2}\leq C\eta^{c},
    \$
    where $c_1,c_2,C,c>0$ are constants independent of $\eta$. This implies that for sufficiently small $\eta$, 
    \$
    \cN_{B}(C\eta^c,\cF,d_1)\leq \cN(\eta,\cF,d_\infty)\lesssim\eta^{-m},
    \$
    where $\lesssim$ omits constants independent of $\eta$. By taking $\epsilon=C\eta^c$, we obtain that for sufficiently small $\epsilon$, 
    \$
    \log\cN_B(\epsilon,\cF,d_1)\lesssim \log(\frac{1}{\epsilon}),
    \$
    where $\lesssim$ omits constants independent of $\epsilon$.
    Again, since $\cN_B(\sqrt{\epsilon},\cF,d_h)\leq\cN_B(\epsilon,\cF,d_1)$, we have 
    \$
    \log\cN_B(\epsilon,\cF,d_h)\lesssim \log(\frac{1}{\epsilon})
    \$ 
    for sufficiently small $\epsilon$, where $\lesssim$ omits constants independent of $\epsilon$. 
    \end{itemize}
The proof of the theorem is complete by combining the above two cases. 
\end{proof}

\subsection{Distribution estimation}\label{sec:4.2}

Once we obtain the bracketing entropy estimate of the functional class $\cF$ in Theorem \ref{thm:4.1}, we can  use empirical process theory to derive the convergence rate of the MLE for the parameter $\alpha$. Specifically, consider a true distribution $f(x;\alpha^{\tr},\phi)$ with the parameter $\alpha^{\tr}\in\cX\coloneqq\cB_{\cM}(\alpha^*,D)$. Then the MLE for $\alpha$, based on $n$ independent samples $\{x_i\}_{i=1}^n$ drawn from $f(x;\alpha^{\tr},\phi)$, is given by 
\#\label{equ:MLE-hat}
    \hat\alpha^{\MLE}=\argmin_{\alpha\in\cX} \sum_{i=1}^n\phi(d_g(x_i,\alpha)).
\#
To derive the convergence rate of this MLE, we use the hellinger distance between the true distribution $f(x;\alpha^{\tr},\phi)$ and the estimated distribution $f(x;\hat\alpha^{\MLE},\phi)$. Corollary \ref{corollary:4.2} shows that this  rate is root-$n$ up to logarithmic terms, matching the classical results in the Euclidean cases. 

\begin{corollary}\label{corollary:4.2}
    Suppose $(\cM,g^{\cM})$ is a Riemannian homogeneous space  and Condition \ref{cond:4.1} holds. Let $f(x;\alpha^{\tr},\phi)$ be the true density in \eqref{equ:f} with $\alpha^{\tr}\in \cB_{\cM}(\alpha^*,D)$. Let $\hat\alpha^{\MLE}$ be the MLE  for $\alpha$, based on $n$ independent samples $\{x_i\}_{i=1}^n$ from $f$, defined in \eqref{equ:MLE-hat}.   
    Then for sufficiently large $n$, it holds with probability at least $1-ce^{-c\log^2n}$ that
    \#
    d_h(f(x;\alpha^{\tr},\phi),f(x;\hat\alpha^{\MLE},\phi))\lesssim \frac{\log n}{\sqrt{n}},\label{equ:cor44}\\
    d_1(f(x;\alpha^{\tr},\phi),f(x;\hat\alpha^{\MLE},\phi))\lesssim \frac{\log n}{\sqrt{n}},\label{equ:cor44-2}
    \#
    where $d_h$ and $d_1$ are the hellinger distance and  the $L^1$ distance, respectively, $c$ is a constant, and $\lesssim$ omits constants independent of $n$. 
\end{corollary}

\begin{proof}
    This follows from empirical process theory and the bracketing entropy estimates in Theorem \ref{thm:4.1}. Specifically, by the bracketing entropy estimates in Theorem \ref{thm:4.1},  the bracketing entropy integral satisfies that
    \$
    J_B(\delta,\cF,d_h)\coloneqq\int_0^\delta \sqrt{\log\cN_{B}(u,\cF,d_h)}du\lesssim\int_0^\delta \log^{1/2}(\frac{1}{u})du,
    \$
    for sufficiently small $\delta$, where $\lesssim$ omits constants independent of $\delta$. By Theorem 2,  \citet{wong1995probability}, we conclude that with probability at least $1-ce^{-c\log^2n}$, 
    \$
    d_h(f(x;\alpha^{\tr},\phi),f(x;\hat\alpha^{\MLE},\phi))\lesssim\frac{\log n}{\sqrt{n}},
    \$
    where $\lesssim$ omits constants independent of $n$ and $c$ is a universal constant. This proves \eqref{equ:cor44}.  Since the $L^1$ distance is upper bounded by twice the hellinger distance, the inequality \eqref{equ:cor44-2}  follows from the inequality \eqref{equ:cor44}. This proves the corollary.  
\end{proof}

\subsection{Parameter estimation}\label{sec:4.3}

This section aims to derive the convergence rate of the MLE in terms of the geodesic distance between parameters. The key component is to examine the identifiability of the parameter $\alpha$. Clearly, if $\phi$ is a constant, as in the  uniform distribution in Example \ref{eg:3.5},  the parameter is non-identifiable and thus the parameter estimation is impossible. To avoid this undesirable situation, we further assume that $\phi$ is strictly increasing and continuously differentiable. With these additional assumptions, we can establish the identifiability of the parameter $\alpha$ in the function $f(x;\alpha,\phi)$. Furthermore, we show that the geodesic distance convergence rate is upper bounded by the rate of $L^1$ distance convergence. These results are formally stated in Theorem \ref{thm:4.3}. 


        

\begin{condition}\label{cond:4.5}
    Let $(\cM,g^{\cM})$ be a Riemannian homogeneous space with $r_{\cM}=\sup_{x,y}d_g(x,y)$. Assume $\phi$ is a function satisfying Condition \ref{cond:4.1}. In addition, we assume that $\phi$ is strictly increasing on $[0,r_\cM]$ and continuously differentiable on $(0,r_{\cM})$.
\end{condition}

\begin{theorem}\label{thm:4.3}
    Suppose $(\cM,g^{\cM})$ is a Riemannian homogeneous space and $\phi$ is a function satisfying Condition \ref{cond:4.5}. Let $f(x;\alpha,\phi)$ be the density in \eqref{equ:f} and $\cX$ be a compact set in $\cM$. Let $\alpha_0\in\cX$ be a fixed point. Then for all $\alpha\in\cX$, the following inequality holds:
    \$
    d_g(\alpha,\alpha_0)\leq C\cdot d_1(f(x;\alpha,\phi),f(x;\alpha_0,\phi)),
    \$
    where $d_1$ is the $L^1$ distance and $C>0$ is a constant independent of $\alpha$.
\end{theorem}

\begin{proof}
    By Lemma \ref{lma:c6}, we have
    \#\label{equ:4.13}
    \lim_{\epsilon\to 0}\inf\left\{{D(\alpha,\alpha_0)}/{d_g(\alpha,\alpha_0)}\mid d_g(\alpha,\alpha_0)<\epsilon\right\}>0,
    \#
    where
    \$
    D(\alpha,\alpha_0)=d_1(f(x;\alpha,\phi),f(x;\alpha_0,\phi)).
    \$
    Therefore, there exists a positive constant $\epsilon_0$ such that for all $\alpha$ with $d_g(\alpha,\alpha_0)<\epsilon_0$, the following inequality holds:
    \#\label{equ:4.15}
    d_g(\alpha,\alpha_0)\leq C_1\cdot D(\alpha,\alpha_0),
    \#
    where $C_1>0$ is a constant independent of $\alpha$. To prove the theorem, we then show that
    \#\label{equ:4.16}
    \inf_{\alpha\in\cX-\cB_{\cM}(\alpha_0,\epsilon_0)}D(\alpha,\alpha_0)>0,
    \#
    where $\cB_{\cM}(\alpha_0,\epsilon_0)=\{x\in\cM\mid d_g(x,\alpha)<\epsilon_0\}$ is a geodesic ball. Suppose on the contrary that
    \$
    \inf_{\alpha\in\cX-\cB_{\cM}(\alpha_0,\epsilon_0)}D(\alpha,\alpha_0)=0.
    \$
    Then there exists a sequence $\{\alpha_n\}_{n=1}^{\infty}\subseteq\cX-\cB_{\cM}(\alpha_0,\epsilon_0)$ such that 
    \$
    D(\alpha_n,\alpha_0)\to 0.
    \$
    Since $\cX-\cB_{\cM}(\alpha_0,\epsilon_0)$ is a compact set, we can apply Lemma \ref{lma:c3} to obtain that $d_g(\alpha_n,\alpha_0)\to0$. This contradicts the fact that $d_g(\alpha_n,\alpha_0)\geq \epsilon_0$, and thus we prove the result \eqref{equ:4.16}. Combining this with \eqref{equ:4.15}, we immediately obtain the theorem.  
\end{proof}

Theorem \ref{thm:4.3} allows us to transform the hellinger/$L^1$ distance convergence rate into the parameter estimation rate. In particular, by combining Theorem \ref{thm:4.3} with Corollary \ref{corollary:4.2}, we can demonstrate that the parameter estimation rate of the MLE is root-$n$ up to logarithmic terms. This result matches the classical results in the Euclidean cases and is presented in Corollary \ref{corollary:4.4}. 

\begin{corollary}\label{corollary:4.4}
    Assume $(\cM,g^{\cM})$ is a Riemannian homogeneous space and $\phi$ satisfies Condition \ref{cond:4.5}. Let $f(x;\alpha^{\tr},\phi)$ be the true density in \eqref{equ:f} with $\alpha^{\tr}\in\cB_{\cM}(\alpha^*,D)$. Let $\hat\alpha^{\MLE}$ be the MLE  for $\alpha$, based on $n$ independent samples $\{x_i\}_{i=1}^n$ from $f$, defined in \eqref{equ:MLE-hat}.   
    Then for sufficiently large $n$, it holds with probability at least $1-ce^{-c\log^2n}$ that
    \$
    d_g(\hat\alpha^{\MLE},\alpha^{\tr})\lesssim \frac{\log n}{\sqrt{n}},
    \$
    where $c$ is a constant independent of $n$ and $\lesssim$ omits constants independent of $n$.
\end{corollary}

\begin{proof}
    This immediately follows from Theorem \ref{thm:4.3} and Corollary \ref{corollary:4.2}.
\end{proof}

\begin{remark}
    It can be seen that the conditions in Corollary \ref{corollary:4.4}  hold for a wide range of Riemannian radial distributions, such as Riemannian Gaussian distributions, Riemannian Laplacian distributions ($\phi(r)=\beta r$ with a sufficiently large $\beta$), and von Mises-Fisher distributions. Thus, the conclusions in Corollary \ref{corollary:4.4} also hold for these Riemannian radial distributions. 
\end{remark}

\section{Minimax lower bounds for parameter estimation rates}\label{sec:5}

In the previous section, we derived the parameter estimation rate of the MLE for $\alpha$. To determine the optimality of this rate, this section establishes a lower bound for such rate. Our analysis employs the classical minimax framework, which will be revisited in Section \ref{sec:minimax}. Then we will give a high-level intuition behind our minimax analysis in Section \ref{sec:intuition}, highlighting the role of symmetry in our problem. Then in Section \ref{sec:5.1} and \ref{sec:5.2}, we will rigorously establish the minimax lower bounds in the case of simply connected Riemannian symmetric spaces. These minimax lower bounds will match the upper bounds established in Section \ref{sec:4} up to logarithmic terms, thus establishing the optimality of the MLE for a wide range of Riemannian radial distributions.

\subsection{Minimax framework}\label{sec:minimax}

First, let us review the classical minimax analysis framework \citep{wainwright2019high} within the context of our problem. Consider a Riemannian symmetric space $(\cM,g^{\cM})$ and assume that $\phi$ is a function satisfying Condition \ref{cond:3.1}. For any $\alpha\in\cM$, denote by $f(x;\alpha,\phi)$  the density in \eqref{equ:f}. Our objective is to estimate the parameter $\alpha$, based on samples $\{x_i\}_{i=1}^n$ independently drawn from the distribution $f(x;\alpha,\phi)$ for some unknown $\alpha\in\cX\coloneqq\cB_{\cM}(\alpha^*,D)$. Given any estimator $\hat\alpha$, which is  a measurable function mapping from $\cM^n$ to $\cM$, we define its risk as 
\$
\EE[d_g(\hat\alpha(x_1,\ldots,x_n),\alpha)],
\$
where the expectation is taken over the samples $\{x_i\}_{i=1}^n$. Also, we define the minimax risk of the parameter estimation problem as 
\#\label{equ:5.1}
\cR_n(\phi)=\inf_{\hat\alpha}\sup_{\alpha\in\cX}\EE[d_g(\hat\alpha(x_1,\ldots,x_n),\alpha)],
\#
where we take the supremum over all $\alpha\in\cX$ and the infimum over all estimators $\hat\alpha$. The goal  is to establish a lower bound for this minimax risk under certain conditions on $\phi$. 

One useful technique in bounding the minimax risk is through the Fano's inequality from information theory. Given a finite set $\cV\subseteq\cX$, we say it is $2\delta$-separated if $d_g(\alpha,\alpha')\geq 2\delta$ for all distinct pairs $\alpha,\alpha'\in\cV$. Given any such $2\delta$-separated set, the Fano's inequality states that the minimax risk \eqref{equ:5.1} has the following lower bound:
\#\label{equ:5.2}
\cR_n(\phi)\geq \delta\cdot\left\{1-\frac{I_{\cV}+\log 2}{\log|\cV|}\right\},
\#
where $|\cV|$ is the cardinality of the set $\cV$, and 
\$
I_{\cV}=\frac{1}{|\cV|}\sum_{\alpha\in\cV}D_{\KL}(P^n_{\alpha}\| \bar P).
\$
Here $P_{\alpha}$ represents the distribution $f(x;\alpha,\phi)$, $P_{\alpha}^n$ denotes the $n$-product of the distribution $P_{\alpha}$, $\bar P$ denotes the mixture distribution $\frac{1}{|\cV|}\sum_{\alpha\in\cV}P_{\alpha}^n$, and $D_{\KL}$ is the KL divergence. By the convexity of the KL divergence, we have that
\#
I_{\cV}&\leq \frac{1}{|\cV|^2}\sum_{\alpha,\alpha'\in\cV}D_{\KL}(P_{\alpha}^n\|P_{\alpha'}^n)\notag\\
&\overset{({\rm i})}{=}n\cdot\frac{1}{|\cV|^2}\sum_{\alpha,\alpha'\in\cV}D_{\KL}(P_\alpha\|P_{\alpha'})\notag\\
&\leq n \cdot \sup_{\alpha,\alpha'\in\cV}D_{\KL}(P_{\alpha}\|P_{\alpha'}),\label{equ:5.3}
\#
where (i) uses the equality $D_{\KL}(P_{\alpha}^n\|P_{\alpha'}^n)=nD_{\KL}(P_{\alpha}\|P_{\alpha'})$. By combining \eqref{equ:5.3} with \eqref{equ:5.2}, we obtain the following lower bound for the minimax risk:
\#\label{equ:5.4}
\cR_n(\phi)\geq \delta\cdot\left\{1-\frac{n\cdot\sup_{\alpha,\alpha'\in\cV}D_{\KL}(P_{\alpha}\|P_{\alpha'})+\log 2}{\log|\cV|}\right\}.
\#
In concrete problems, we will establish the lower bound for $\cR_n(\phi)$ by selecting a suitable set $\cV$ and determining upper bounds on $D_{\KL}(P_{\alpha}\|P_{\alpha'})$ for all pairs $\alpha,\alpha'\in\cV$. 

\subsection{Symmetry}\label{sec:intuition}

To proceed, let us introduce our intuitions and emphasize the role of symmetry in our problem. Let $(\cM,g^{\cM})$ be a Riemannian symmetric space and assume $\phi$ satisfies Condition \ref{cond:3.1}. Our goal is to derive a lower bound on $\cR_{n}(\phi)$, defined in \eqref{equ:5.1}. For this purpose, we choose $\cV=\{\alpha,\alpha(2\delta),\alpha(-2\delta)\}$, where $\alpha\in\cX$ and $\alpha(t)$ is a unit-speed geodesic in $\cM$ with $\alpha(0)=\alpha$. We assume $\delta$ is sufficiently small such that $\cV\subseteq\cX$ and the distance between $\alpha(2\delta)$ and $\alpha(-2\delta)$ is $4\delta$. Thus, $\cV$ is $2\delta$-separated and we can use \eqref{equ:5.4} to derive a lower bound for $\cR_n(\phi)$. 

We need to establish an upper bound on $\sup_{\alpha,\alpha'\in\cV}D_{\KL}(P_\alpha\|P_{\alpha'})$, where $P_{\alpha}$ denotes the density $f(x;\alpha,\phi)$. For example, let us consider the pair $(\alpha,\alpha(2\delta))$. By definition of $f$, we can rewrite the KL divergence between $P_{\alpha}$ and $P_{\alpha(2\delta)}$ as follows:
\#\label{equ:DKL}
D_{\KL}(P_{\alpha}\|P_{\alpha(2\delta)})=\frac{1}{Z(\phi)}\cdot (I(2\delta)-I(0)),
\#
where $I(t),t\in\RR$ is defined as
\#\label{equ:5.6}
I(t)=\int_{\cM}\phi(d_g(x,\alpha(t)))e^{-\phi(d_g(x,\alpha))}\dvol(x).
\#
Here $\alpha(t)$ is the unit-speed geodesic passing $\alpha(0)=\alpha$ and $\alpha(2\delta)$. Our valuable observation is the following symmetry result. 

\begin{proposition}[Symmetry]\label{prop:5.1}
    Consider a Riemannian symmetric space $(\cM,g^{\cM})$  and assume Condition \ref{cond:3.1} holds. Let $\alpha(t)$ be a geodesic in $\cM$ with $\alpha(0)=\alpha$, and define $I(t)$ as the integral \eqref{equ:5.6}. Then $I(t)=I(-t)$ holds for all $t\in\RR$. 
\end{proposition}

\begin{proof}
    Since $\cM$ is a Riemannian symmetric space, there is an isometry $F$ of $\cM$ such that $F\circ F$ is the identity, $F(\alpha)=\alpha$,  and $F(\alpha(t))=\alpha(-t)$ for all $t\in\RR$, where $\alpha(t)$ represents a geodesic with $\alpha(0)=\alpha$. Using this $F$, we can show that
    \$
    I(t)&=\int_{\cM}\phi(d_g(x,\alpha(t)))e^{-\phi(d_g(x,\alpha))}\dvol(x)\\
    &\overset{({\rm i})}{=}\int_{\cM}\phi(d_g(F(x),\alpha(t)))e^{-\phi(d_g(F(x),\alpha))}\dvol(x)\\
    &\overset{({\rm ii})}{=}\int_{\cM}\phi(d_g(x,F(\alpha(t))))e^{-\phi(d_g(x,F(\alpha)))}\dvol(x)\\
    &\overset{({\rm iii})}{=}\int_{\cM}\phi(d_g(x,\alpha(-t)))e^{-\phi(d_g(x,\alpha))}\dvol(x)\\
    &=I(-t),
    \$
    where (i) uses the property \eqref{equ:homo}, (ii) uses the facts that $F\in\Iso(\cM)$ and $F\circ F$ is the identity, and (iii) follows from the definition of $F$. 
\end{proof}

Proposition \ref{prop:5.1} shows that $I(t)=I(-t)$ for all $t\in\RR$. Thus, if $I(t)$ has  desired differentiability, then 
\#\label{equ:5.7}
I(2\delta)-I(0)=O(\delta^2)
\#
for sufficiently small $\delta$. Combining this with \eqref{equ:DKL}, we obtain $D_{\KL}(P_{\alpha}\|P_{\alpha(2\delta)})=\cO(\delta^2)$ for sufficiently small $\delta$. If we can derive similar results for all pairs $\alpha,\alpha'\in\cV$, then we actually prove that for sufficiently small $\delta$,
\$
\cR_n(\phi)\geq \delta\cdot\left\{1-\frac{C n\delta^2+\log2}{\log 3}\right\},
\$
 where $C>0$ is a universal constant. It implies that $\cR_n(\phi)\geq C_1n^{-1/2}$ for some universal constant $C_1>0$ and sufficiently large $n$, by taking  $Cn\delta^2=0.01$. Up to logarithmic terms, this lower bound matches the upper bound we established in Section \ref{sec:4}, and thus proving the optimality of the MLE. 

 In the subsequent sections, we will delve into rigorous proofs of the above intuitions. Some of the unsolved challenges are
 \begin{itemize}
     \item the integrability of the integral $I(t)$;
     \item the differentiability of $I(t)$;
     \item the upper bound \eqref{equ:5.7}.
 \end{itemize}
 To address these issues, we will use Riemannian geometry and Lie theory. For technical reasons\footnote{The cut locus on a simply connected Riemannian symmetric space is more tractable.}, we will only consider simply connected Riemannian symmetric spaces. Due to distinct proof techniques, the noncompact and compact cases will be addressed separately in Section \ref{sec:5.1} and \ref{sec:5.2}. The Lie theory will only be  used in the compact case.


\subsection{Simply connected noncompact Riemannian symmetric spaces}\label{sec:5.1}

Now, we consider a simply connected noncompact Riemannian symmetric space $(\cM,g^{\cM})$ and let $\phi$ be a function  satisfying Condition \ref{cond:3.1}. By the decomposition result, Proposition \ref{thm:decomposition}, we can write $\cM$ as a product space $\cM=\cM_{H}\times\cM_{C}$, where $\cM_{H}$ is a Hadamard Riemannian symmetric space and $\cM_C$ is a simply connected compact Riemannian symmetric space. Moreover, the dimension of $\cM_H$ is positive since $\cM$ is noncompact.  
By using this decomposition, we can write each point $x\in\cM$ as $(x^H,x^C)$, where $x^H$ and $x^C$ represent the Hadamard and compact components, respectively. 


Our goal is to establish a lower bound on the minimax risk $\cR_n(\phi)$. For this purpose, we select a set $\cV=\{\alpha,\alpha(2\delta),\alpha(-2\delta)\}$, where $\alpha(t)=(\alpha^H(t),\alpha^C)$ is a unit-speed geodesic with a constant compact component. We assume that $\alpha(0)=\alpha\in\cX$ and $\delta$ is sufficiently small such that $\cV\subseteq\cX$ and $d_g(\alpha(2\delta),\alpha(-2\delta))=4\delta$. Such choice of a curve is feasible as $\cM_H$ has a positive dimension, 
and this choice will significantly simplify our analysis due to the empty cut locus on $\cM_H$. 

Our analysis will follow the intuitions in Section \ref{sec:intuition}. First, by the Fano's method and inequality \eqref{equ:5.4}, we can reduce the problem to estimating the following quantity:
\#\label{equ:5.8-DKL}
\sup_{\alpha,\alpha'\in\cV}D_{\KL}(P_{\alpha}\|P_{\alpha'}),
\#
where $P_{\alpha}$ denotes the distribution $f(x;\alpha,\phi)$. Then we will use Proposition \ref{prop:5.1} and the differentiability of $I(t)$ in \eqref{equ:5.6} to upper bound this quantity. Our proofs will need the following conditions on $\phi$.

\begin{condition}\label{cond:5.2}
     Consider a simply connected noncompact Riemannian symmetric space $(\cM,g^{\cM})$. Let $\kappa_{\min}<0$ be a lower sectional curvature bound of $\cM$ and $m$ the dimension of $\cM$. Assume the following conditions hold.
     \begin{enumerate}
         \item[(1)] $\phi$ satisfies Condition \ref{cond:3.1}.
         
         \item[(2)] The integral $\int_{0}^\infty\phi(2r)e^{-\phi(r)}\sn_{\kappa_{\min}}^{m-1}(r)dr$ is finite, where $\sn_{\kappa_{\min}}(r)$ is given by \eqref{equ:sn}. 
         \item[(3)] $\phi$ is second-order continuously differentiable on $(0,\infty)$.
         \item[(4)] $\phi$ is convex and the integral $\int_{0}^\infty\phi'(2r)e^{-\phi(r)}\sn_{\kappa_{\min}}^{m-1}(r)dr$ is finite.
         \item[(5)] When $m=1$, the integral 
         \$
         \int_0^{\infty}\phi''(r)e^{-\phi(\max\{r-1,0\})}dr
         \$ 
         is finite.
         When $m>1$, the following integral
         \$
         \int_{0}^{\infty}\left(\phi''(r)+3\phi'(r)\frac{\sn_{\kappa_{\min}}'(r)}{\sn_{\kappa_{\min}}(r)}\right)e^{-\phi(\max\{r-1,0\})}\sn_{\kappa_{\min}}^{m-1}(r)dr
         \$ 
         is finite.
     \end{enumerate}
\end{condition}

\begin{remark}
    Many functions $\phi$ satisfy Condition \ref{cond:5.2}, including $\phi(r)=\beta r^p$ with $p>1$ and $\beta>0$, and $\phi(r)=\beta r$ with a sufficiently large $\beta$. Therefore, Proposition \ref{prop:5.3} and Theorem \ref{thm:5.5} below hold for Riemannian radial distributions associated with these $\phi$. 
\end{remark}

Under Condition \ref{cond:5.2}, we can establish an upper bound for the quantity in \eqref{equ:5.8-DKL}, as presented in Proposition \ref{prop:5.3}. 

\begin{proposition}\label{prop:5.3}
    Let $(\cM,g^{\cM})$ be a simply connected noncompact Riemannian symmetric space. Assume $\phi$ is a function satisfying Conditions \ref{cond:5.2}. Let $\alpha(t)=(\alpha^H(t),\alpha^C)$ be a unit-speed geodesic in $\cM$ with a constant compact component. Then for a sufficiently small $\delta_0>0$, we have
    \$
    D_{\KL}(P_{\alpha(t_1)}\|P_{\alpha(t_2)})\leq C|t_1-t_2|^2,
    \$
    for all  $t_1,t_2\in\RR$ with $|t_1-t_2|\leq\delta_0$, where $C$ is a constant independent of the choice of $t_1,t_2$.
\end{proposition}

\begin{proof}
    By Lemma \ref{lma:d1}, the KL divergence $D_{\KL}(P_{\alpha(t_1)}\|P_{\alpha(t_2)})$ is finite for all $t_1,t_2\in\RR$. Moreover, we can rewrite it as follows
    \#\label{equ:5.8}
    D_{\KL}(P_{\alpha(t_1)}\|P_{\alpha(t_2)})=\frac{1}{Z(\phi)}(I_{t_1}(\Delta)-I_{t_1}(0)),
    \# 
    where $\Delta=t_2-t_1$, $I_{t_1}(t)$ is a function on $\RR$ defined by
    \$
    I_{t_1}(t)=\int_{\cM}\phi(d_g(x,\alpha_{t_1}(t)))e^{-\phi(d_g(x,\alpha_{t_1}(0)))}\dvol(x),
    \$
    and $\alpha_{t_1}(t)=\alpha(t+t_1)$. Note that $\alpha_{t_1}(t)$ is also a unit-speed geodesic in $\cM$ with a constant compact component. Therefore, without loss of generality, we may assume $t_1=0$ and denote $I_0(t)$ by $I(t)$. 
    By Lemma \ref{lma:d3}, there exists a small constant $\delta_0>0$ such that for any $|t|\leq \delta_0$,  the following bound holds:
    \$
    |I(t)-I(0)|\leq C_0t^2, 
    \$
    where $C_0$ is a constant independent of $t$. Combining this with \eqref{equ:5.8}, we have that for all $|t|\leq \delta_0$,
    \$
    D_{\KL}(P_{\alpha(0)}\|P_{\alpha(t)})\leq Ct^2,
    \$
    where $C$ is a constant independent of $t$. By choosing a suitable geodesic $\alpha$ as we discussed, we can prove that for all $t_1,t_2\in\RR$ with $|t_1-t_2|\leq\delta_0$, 
    \$
    D_{\KL}(P_{\alpha(t_1)}\|P_{\alpha(t_2)})\leq C(t_1-t_2)^2,
    \$
    where $C$ is a constant independent of $t_1$ and $t_2$. This proves the proposition. 
\end{proof}

By using Proposition \ref{prop:5.3}, we can now establish a lower bound on the minimax risk $\cR_n(\phi)$ in the following theorem. This minimax lower bound is of order $n^{-1/2}$, matching the parameter estimation rates of the MLE in Corollary \ref{corollary:4.4} up to logarithmic terms.   

\begin{theorem}\label{thm:5.5}
    Let $(\cM,g^{\cM})$ be a simply connected noncompact Riemannian symmetric space, and assume $\phi$ to be a function satisfying Condition \ref{cond:5.2}. Let $\cR_n(\phi)$ be the minimax risk defined in \eqref{equ:5.1}. Then for sufficiently large $n$, we have
    \$
    \cR_n(\phi)\geq Cn^{-1/2},
    \$
    where $C>0$ is a constant independent of $n$. 
\end{theorem}

\begin{proof}
    Pick $\cV=\{\alpha,\alpha(2\delta),\alpha(-2\delta)\}$, where $\alpha(t)$ is a unit-speed geodesic with $\alpha(0)=\alpha\in\cX$ and a constant compact component. Let $\delta$ be sufficiently small such that $d_g(\alpha(2\delta),\alpha(-2\delta))=4\delta$ and $\cV\subseteq\cX$. Then $\cV$ is a $2\delta$-separated set in $\cX$, and by inequality \eqref{equ:5.4}, we have
    \$
    \cR_n(\phi)\geq \delta\cdot\left\{1-\frac{n\cdot \sup_{\alpha,\alpha'}D_{\KL}(P_{\alpha}\|P_{\alpha'})+\log 2}{\log 3}\right\}.
    \$
    By Proposition \ref{prop:5.3}, we have for sufficient small $\delta$,
    \$
    \sup_{\alpha,\alpha'}D_{\KL}(P_{\alpha}\| P_{\alpha'})\leq C_0 \delta^2,
    \$
    where $C_0$ is a constant independent of $\delta$. Therefore,
    \$
    \cR_n(\phi)\geq \delta\cdot \{1-\frac{C_0n\delta^2+\log 2}{\log 3}\}
    \$
    for sufficient small $\delta$. When $n$ is sufficiently large, we can take $C_0n\delta^2=0.01$, and then obtain
    \$
    \cR_{n}(\phi)\geq Cn^{-1/2},
    \$
    where $C$ is a constant independent of $n$. This proves the theorem.
\end{proof}

\subsection{Simply connected compact Riemannian symmetric spaces}\label{sec:5.2}

Now let us consider a simply connected compact Riemannian symmetric space $(\cM,g^{\cM})$. Assume $\phi$ is a function satisfying Condition \ref{cond:3.1}. To establish the lower bound on the minimax risk $\cR_n(\phi)$, defined in \eqref{equ:5.1}, we select a set $\cV=\{\alpha,\alpha(2\delta),\alpha(-2\delta)\}$, where $\alpha(t)$ is a unit-speed geodesic in $\cM$. We choose $\delta$ to be sufficiently small such that $d_g(\alpha(2\delta),\alpha(-2\delta))=4\delta$. Then we shall use the analysis in Section \ref{sec:5.2} to obtain the minimax lower bound. Specifically, by the Fano's method and inequality, we can reduce the problem to estimating the following quantity, \#\label{equ:5.8-DKL-compact}
\sup_{\alpha,\alpha'\in\cV}D_{\KL}(P_{\alpha}\|P_{\alpha'}),
\#
where $P_{\alpha}$ denotes the distribution $f(x;\alpha,\phi)$. To proceed, we will need the following conditions on $\phi$. 

\begin{condition}\label{cond:5.6}
    Consider a simply connected compact Riemannian symmetric space $(\cM,g^{\cM})$. Let $r_{\cM}=\sup_{x,y}d_g(x,y)$ be the maximum radius of $\cM$ and $m\geq 2$ the dimension of $\cM$. We assume the following conditions hold.
    \begin{enumerate}
        \item[(1)] $\phi$  is a Lipschitz continuous function on $[0,r_{\cM}]$ satisfying Condition \ref{cond:3.1}.
        \item[(2)] $\phi$ is second-order continuously differentiable on $(0,r_{\cM})$ and $\phi'$ is bounded on $(0,r_{\cM})$. 
        \item[(3)] The integral $\int_0^{r_{\cM}}|\phi''(r)|r^{m-1}dr$ is finite. 
    \end{enumerate}
\end{condition}

\begin{remark}
    Many functions satisfy Condition \ref{cond:5.6}. For example, $\phi(r)=\beta r^p$ with $p>1$ and $\beta\geq 0$, and $\phi(r)=\beta r$ with a sufficiently large $\beta$ all satisfy Condition \ref{cond:5.6}. In addition, when $\cM=\SSS^m$ is the $m$-sphere, the function $\phi(r)=\beta(1-\cos(r))$ also satisfies Condition \ref{cond:5.6}. Consequently, the following Proposition \ref{prop:5.8} and Theorem \ref{thm:5.9} hold for Riemannian radial distributions associated with these $\phi$. 
\end{remark}

Under Condition \ref{cond:5.6}, we can derive a desirable upper bound on the quantity in \eqref{equ:5.8-DKL-compact}. This is presented in the following proposition.

\begin{proposition}\label{prop:5.8}
    Let $(\cM,g^{\cM})$ be a simply connected compact Riemannian symmetric space and $\phi$ a function satisfying Condition  \ref{cond:5.6}. Let $\alpha(t)$ be a unit-speed geodesic in $\cM$. Then for a sufficiently small $\delta_0>0$, we have
    \$
    D_{\KL}(P_{\alpha(t_1)}\|P_{\alpha(t_2)})\leq C|t_1-t_2|^2,
    \$
    for any $t_1,t_2\in\RR$ with $|t_1-t_2|\leq \delta_0$, where $C$ is a constant independent of the choice of $t_1,t_2$. 
\end{proposition}

\begin{proof}
    By Lemma \ref{lma:d.2.2.1}, the KL divergence $D_{\KL}(P_{\alpha(t_1)}\|P_{\alpha(t_2)})$ is finite for all $t_1,t_2\in\RR$. It can be rewritten as follows
    \#\label{equ:5.11-compact}
    D_{\KL}(P_{\alpha(t_1)}\|P_{\alpha(t_2)})=\frac{1}{Z(\phi)}(I_{t_1}(\Delta)-I_{t_1}(0)),
    \# 
    where $\Delta=t_2-t_1$, $I_{t_1}(t)$ is a function on $\RR$ defined by
    \$
    I_{t_1}(t)=\int_{\cM}\phi(d_g(x,\alpha_{t_1}(t)))e^{-\phi(d_g(x,\alpha_{t_1}(0)))}\dvol(x),
    \$
    and $\alpha_{t_1}(t)=\alpha(t+t_1)$. Here $\alpha_{t_1}(t)$ is also a unit-speed geodesic in $\cM$. Thus, 
    by Lemma \ref{lma:d.2.2.3}, there exists a small constant $\delta_0>0$ such that for any $|\Delta|\leq \delta_0$,  the following bound holds:
    \$
    |I_{t_1}(\Delta)-I_{t_1}(0)|\leq C_0\Delta^2, 
    \$
    where $C_0$ is a constant independent of $t$. Combining this with \eqref{equ:5.11-compact}, we have that for all $|\Delta|\leq \delta_0$,
    \$
    D_{\KL}(P_{\alpha(t_1)}\|P_{\alpha(t_2)})\leq C\Delta^2,
    \$
    where $C$ is a constant independent of $t_1$ and $t_2$. This proves the proposition. 
\end{proof}

By using Proposition \ref{prop:5.8}, we can now establish a lower bound on the minimax risk $\cR_n(\phi)$ in the following theorem. This minimax lower bound is of order $n^{-1/2}$, matching the parameter estimation rates of the MLE in Corollary \ref{corollary:4.4} up to logarithmic terms. 

\begin{theorem}\label{thm:5.9}
    Let $(\cM,g^{\cM})$ be a simply connected compact Riemannian symmetric space and $\phi$ a function satisfying Condition \ref{cond:5.6}. Define $\cR_n(\phi)$ as the minimax risk in \eqref{equ:5.1}.  Then for sufficiently large $n$, we have
    \$
    \cR_n(\phi)\geq Cn^{-1/2},
    \$
    where $C>0$ is a constant independent of $n$.
\end{theorem}

\begin{proof}
    Pick $\cV=\{\alpha,\alpha(2\delta),\alpha(-2\delta)\}$, where $\alpha(t)$ is a unit-speed geodesic in $\cM$ with $\alpha(0)=\alpha$. Let $\delta$ be sufficiently small such that 
    \$
    d_g(\alpha(2\delta),\alpha(-2\delta))=4\delta.
    \$
    Then $\cV$ is a $2\delta$-separated set in $\cM$, and by inequality \eqref{equ:5.4}, we have
    \$
    \cR_n(\phi)\geq \delta\cdot\left\{1-\frac{n\cdot \sup_{\alpha,\alpha'}D_{\KL}(P_{\alpha}\|P_{\alpha'})+\log 2}{\log 3}\right\}.
    \$
    By Proposition \ref{prop:5.8}, we have for sufficient small $\delta$,
    \$
    \sup_{\alpha,\alpha'}D_{\KL}(P_{\alpha}\| P_{\alpha'})\leq C_0 \delta^2,
    \$
    where $C_0$ is a constant independent of $\delta$. Therefore,
    \$
    \cR_n(\phi)\geq \delta\cdot \{1-\frac{C_0n\delta^2+\log 2}{\log 3}\}
    \$
    for sufficient small $\delta$. When $n$ is sufficiently large, we can take $C_0n\delta^2=0.01$, and then obtain
    \$
    \cR_{n}(\phi)\geq Cn^{-1/2},
    \$
    where $C$ is a constant independent of $n$. This proves the theorem.
\end{proof}


    

    


\section{Riemannian radial distributions with an unknown temperature}\label{sec:6-temperature}

In previous sections, we study the 
Riemannian radial distribution $f(x;\alpha,\phi)$ with a fixed and known function $\phi$, where the primary goal is to estimate the unknown location parameter $\alpha$. In this section, we aim to extend this problem to a more practical setting where the Riemannian radial distribution has an additional unknown temperature parameter. Specifically, for  a general Riemannian homogeneous space $(\cM,g^{\cM})$, we shall study the following parametric family of distributions on $\cM$: 
\#\label{equ:6.1-f}
f(x;\alpha,\beta,\phi)=\frac{1}{Z(\beta,\phi)}e^{-\beta\phi(d_g(x,\alpha))},\quad \forall x\in\cM,
\#
where $\phi$ is a function known {\it a priori}, $\alpha\in\cM$ and $\beta>0$ are the unknown location and temperature parameters, respectively, and $Z(\beta,\phi)=\int_{\cM}e^{-\beta\phi(d_g(x,\alpha))}\dvol(x)$ is the normalizing constant. Our objective is to investigate the MLE for the parameters $\alpha$ and $\beta$, and derive their rates of convergence. Combined with the minimax lower bounds established in Section \ref{sec:5}\footnote{The minimax lower bounds in Section \ref{sec:5} adapt to the case where the temperature parameter is unknown, since Section \ref{sec:5} addresses the simpler case where the temperature parameter is fixed and known.}, we shall obtain the optimal parameter estimation rates and demonstrate the optimality of the MLE in a variety of contexts. 

The rest of this section proceeds as follows. Section \ref{sec:6.1} presents basic properties of  Riemannian radial distributions with an unknown temperature parameter, such as the properties of the  MLE. In Section \ref{sec:6.2}, \ref{sec:6.3}, and \ref{sec:6.4}, we  study the convergence rates of the MLE. Specifically, Section \ref{sec:6.2} gives an entropy estimate of a certain functional class, Section \ref{sec:6.3} derives the distribution estimation rates using the empirical process theory, and Section \ref{sec:6.4} establishes the parameter estimation rates. 

\subsection{Maximum likelihood estimation}\label{sec:6.1}

Let $(\cM,g^{\cM})$ be a Riemannian homogeneous space  and $\phi$ a function such that $\phi_{\beta_{\min}}$ satisfies Condition \ref{cond:3.1}, where $\phi_\beta(r)=\beta\phi(r)$ and $\beta_{\min}>0$ is a constant. Then for any $\beta\geq \beta_{\min}$, the function $\phi_{\beta}$  satisfies Condition \ref{cond:3.1} and  $f(x;\alpha,\beta,\phi)$ is well-defined, since  $\phi_{\beta}\leq \phi_{\beta_{\min}}$ for all $\beta\geq\beta_{\min}$. To proceed, we will always assume this condition, summarized below, is satisfied.  


\begin{condition}\label{cond:6.1.phi}
    Let $(\cM,g^{\cM})$ be a Riemannian homogeneous space. We assume that $\phi$ is a function such that $\phi_{\beta_{\min}}(\cdot)=\beta_{\min}\phi(\cdot)$ satisfies Condition \ref{cond:3.1} for some constant $\beta_{\min}>0$. 
\end{condition}

Proposition \ref{prop:6.1}  derives the MLEs for $\alpha$ and $\beta$, based on $n$ independent samples drawn from the distribution $f(x;\alpha,\beta,\phi)$. Notably, the MLE of $\alpha$ is the same as when $\beta$ is known. Thus, estimating $\alpha$ via Riemannian optimization is  equally straightforward regardless of whether $\beta$ is known or not.

\begin{proposition}\label{prop:6.1}
    Let $(\cM,g^{\cM})$ be a Riemannian homogeneous space and $\phi$ a function satisfying Condition \ref{cond:6.1.phi} for a constant $\beta_{\min}>0$. Let $f(x;\alpha,\beta,\phi)$ be a true distribution with an unknown $\alpha\in\cX\subseteq\cM$ and $\beta\geq\beta_{\min}$. Given $n$  samples $\{x_i\}_{i=1}^n$ drawn independently from $f(x;\alpha,\beta,\phi)$, the MLEs for $\alpha$ and $\beta$ are given by
    \$
    \hat\alpha^{\MLE}&=\argmin_{\alpha\in\cX}\sum_{i=1}^n\phi(d_g(x_i,\alpha)),\\
    \hat\beta^{\MLE} &=\argmin_{\beta}\beta\sum_{i=1}^n\phi(d_g(x_i,\hat\alpha^{\MLE}))+n\log Z(\beta,\phi),
    \$
    where $Z(\beta,\phi)$ is the normalizing constant in the distribution $f(x;\alpha,\beta,\phi)$. 
\end{proposition}

\subsection{Entropy estimate}\label{sec:6.2}

To derive the convergence rate of the MLEs for $\alpha$ and $\beta$, we first provide entropy estimates for the following functional class 
\#\label{equ:function-6.2}
\cF=\{f(x;\alpha,\beta,\phi)\mid \alpha\in\cX,\beta\in[\beta_{\min},\beta_{\max}]\},
\#
where $\cX=\cB_{\cM}(\alpha^*,D)$ is a geodesic ball in $\cM$ and $0<\beta_{\min}\leq \beta_{\max}<\infty$ are lower and upper bounds for $\beta$. For an effective entropy estimate, we impose the following conditions on $\phi$. 

\begin{condition}\label{cond:6.3}
    Let $(\cM,g^{\cM})$ be an $m$-dimensional Riemannian homogeneous space with $r_{\cM}=\sup_{x,y\in\cM}d_g(x,y)$ and $\phi$ a function satisfying Condition \ref{cond:6.1.phi} for    $\beta_{\min}>0$. We assume the following conditions are satisfied.  
    \begin{itemize}
        \item[(1)] $\phi$ is differentiable on $[0,r_{\cM}]$ and  $|\phi'(r)e^{-\beta_{\min}\phi(r)}|$ is bounded on $[0,r_{\cM}]$ by a constant $L$.
        \item[(2)] When $\cM$ is noncompact, the following two functions
        \$
        h(r)&=e^{-\beta_{\min}\phi(r/2)}\sn_{\kappa_{\min}}^{m-1}(r),\\
        h_1(r)&=\phi(r)e^{-\beta_{\min}\phi(r)}\sn_{\kappa_{\min}}^{m-1}(r),
        \$
        are integrable over $[0,\infty)$, where $\kappa_{\min}<0$ is a lower bound on the sectional curvatures of $\cM$ and $\sn_{\kappa_{\min}}(\cdot)$ is given by \eqref{equ:sn}. In addition, for all sufficiently small $\eta$,
        \$
        \int_{B_\eta}^\infty h(r)dr\leq c_1\eta^{c_2},
        \$
        where $B_\eta=\frac{\log(1/\eta)}{2(m-1)\sqrt{-\kappa_{\min}}}$ and $c_1,c_2>0$ are constants independent of $\eta$. 
    \end{itemize}
\end{condition}

\begin{remark}\label{remark:6.4}
    A broad range of functions $\phi$ satisfy Condition \ref{cond:6.3} for a given $\beta_{\min}>0$. For example, $\phi(r)=\nu r^p$ with $p>1$ and $\nu>0$, and $\phi(r)=\nu r$ with a sufficiently large $\nu$ all satisfy Condition \ref{cond:6.3}. When $\cM=\SSS^{m}$ is the $m$-sphere, $\phi(r)=\nu(1-\cos(r))$ with $\nu>0$ also satisfies this condition. Consequently, the following Theorem \ref{thm:6.2} and Corollary \ref{corollary:6.4} hold for Riemannian radial distributions associated with these $\phi$. 
\end{remark}

Under Condition \ref{cond:6.3}, we can establish the entropy estimates of $\cF$ as follows. 

\begin{theorem}\label{thm:6.2}
    Assume $(\cM,g^{\cM})$ is a Riemannian homogeneous space and $\phi$ is a function satisfying Condition  \ref{cond:6.3} for $\beta_{\min}>0$. Let $f(x;\alpha,\beta,\phi)$ be the density in \eqref{equ:6.1-f} and $\cF$ the functional class in \eqref{equ:function-6.2}. Then for sufficiently small $\epsilon$, the following entropy estimates hold:
    \$
    \log\cN(\epsilon,\cF,d_{\infty}) &\lesssim\log(\frac{1}{\epsilon}),\\
    \log\cN_{B}(\epsilon,\cF,d_1)&\lesssim \log(\frac{1}{\epsilon}),\\
    \log\cN_{B}(\epsilon,\cF,d_h)&\lesssim \log(\frac{1}{\epsilon}),
    \$
    where $d_{\infty},d_1,d_h$ represent the $L^\infty$, $L^1$, and hellinger distance, respectively, and $\lesssim$ omits constants independent of $\epsilon$. 
\end{theorem}

    

\subsection{Distribution estimation}\label{sec:6.3}

Once we obtain the entropy estimate, we can derive the convergence rate of the MLEs of $\alpha$ and $\beta$ using empirical process theory. 
Specifically, we assume  $f(x;\alpha^{\tr},\beta^{\tr},\phi)$ is the true distribution with $\alpha^{\tr}\in\cX\coloneqq\cB_{\cM}(\alpha^*,D)$ and $\beta^{\tr}\in[\beta_{\min},\beta_{\max}]$. Let $\hat\alpha^{\MLE}$ and $\hat\beta^{\MLE}$ be the MLEs for $\alpha$ and $\beta$ based on $n$ independent samples drawn from $f$. Then  the convergence rate of such MLEs, measured by the  hellinger distance between the true and estimated distributions, is derived in Corollary \ref{corollary:6.4}. It shows that this rate is root-$n$ up to logarithmic terms, matching the classical results in the Euclidean cases. 

\begin{corollary}\label{corollary:6.4}
    Suppose $(\cM,g^{\cM})$ is a Riemannian homogeneous space and $\phi$ satisfies Condition \ref{cond:6.3} for $\beta_{\min}>0$. Let $f(x;\alpha^{\tr},\beta^{\tr},\phi)$ be the true density with $\alpha^{\tr}\in\cX$ and $\beta^{\tr}\in[\beta_{\min},\beta_{\max}]$. Let $\hat\alpha^{\MLE}$ and $\hat\beta^{\MLE}$ be the MLEs for $\alpha$ and $\beta$, based on $n$ independent samples drawn from $f$. Then for sufficiently large $n$, it holds with probability at least $1-ce^{-c\log^2n}$ that
    \#
    d_h(f(x;\alpha^{\tr},\beta^{\tr},\phi),f(x;\hat\alpha^{\MLE},\hat\beta^{\MLE},\phi))&\lesssim\frac{\log n}{\sqrt{n}},\label{equ:6.3}\\
    d_1(f(x;\alpha^{\tr},\beta^{\tr},\phi),f(x;\hat\alpha^{\MLE},\hat\beta^{\MLE},\phi))&\lesssim\frac{\log n}{\sqrt{n}},\label{equ:6.4}
    \#
    where $d_h$ and $d_1$ are the hellinger distance and the $L^1$ distance, respectively, $c$ is a constant, and $\lesssim$ omits constants independent of $n$. 
\end{corollary}

\subsection{Parameter estimation}\label{sec:6.4}


To proceed, we derive the parameter estimation rates of the MLEs for $\alpha$ and $\beta$. The key result is the following  theorem, which controls the distance between parameters by the $L^1$ distance  between distributions. This identifiability result allows us to transform the hellinger/$L^1$ distance convergence rate in Corollary \ref{corollary:6.4} into the parameter estimation rates. 

\begin{condition}\label{cond:6.7}
    Assume $(\cM,g^{\cM})$ is a Riemannian homogeneous space with $r_{\cM}=\sup_{x,y}d_g(x,y)$. Assume $\phi$ satisfies Condition \ref{cond:6.3} for some $\beta_{\min}>0$. In addition, we assume that $\phi$ is strictly increasing on $[0,r_{\cM}]$ and continuously differentiable on $(0,r_{\cM})$.
\end{condition}

\begin{theorem}\label{thm:6.4}
    Suppose $(\cM,g^{\cM})$ is a Riemannian homogeneous space and $\phi$ satisfies Condition \ref{cond:6.7} for  $\beta_{\min}>0$. Let $\cX\subseteq\cM$ be a compact set and $\beta_{\max}\geq \beta_{\min}$ be a finite constant. Fix $\alpha_0\in\cX$ and $\beta_0\in[\beta_{\min},\beta_{\max}]$. Then for any $\alpha\in\cX$ and $\beta\in[\beta_{\min},\beta_{\max}]$, the following bounds hold: 
    \$
    d_g(\alpha,\alpha_0)+|\beta-\beta_0|&\leq C\cdot d_1(f(x;\alpha,\beta,\phi),f(x;\alpha_0,\beta_0,\phi)),
    \$
    where $d_1$ is the $L^1$ distance and $C>0$ is a constant independent of $\alpha$ and $\beta$. 
\end{theorem}

By combining Theorem \ref{thm:6.4} and Corollary \ref{corollary:6.4}, we shall obtain the parameter estimation rate of the MLEs for $\alpha$ and $\beta$ as follows. 

\begin{corollary}\label{corollary:6.8}
    Suppose $(\cM,g^{\cM})$ is a Riemannian homogeneous space and $\phi$ satisfies Condition \ref{cond:6.7} for  $\beta_{\min}>0$. Let $f(x;\alpha^{\tr},\beta^{\tr},\phi)$ be the true density with $\alpha\in\cX=\cB_{\cM}(\alpha^*,D)$ and $\beta^{\tr}\in[\beta_{\min},\beta_{\max}]$. Let $\hat\alpha^{\MLE},\hat\beta^{\MLE}$ be the MLEs for $\alpha$ and $\beta$, based on $n$ independent samples drawn from $f$. Then for sufficiently large $n$, it holds with probability at least $1-ce^{-c\log^2n}$ that
    \$
    d_g(\hat\alpha^{\MLE},\alpha^{\tr})+|\hat\beta^{\MLE}-\beta^{\tr}|&\lesssim\frac{\log n}{\sqrt{n}},
    \$
    where $c$ is a constant independent of $n$ and $\lesssim$ omits constants independent of $n$.
\end{corollary}

\begin{proof}
    This corollary immediately follows from Theorem \ref{thm:6.4} and Corollary \ref{corollary:6.4}.
\end{proof}

\begin{remark}
    Examples in Remark \ref{remark:6.4} satisfy the conditions in Corollary \ref{corollary:6.8}. Thus, the parameter estimation rates derived in Corollary \ref{corollary:6.8} hold for the MLEs associated with those distributions. 
\end{remark}

\section{Model complexity for von Mises-Fisher distributions}\label{sec:6}

In previous sections, we show for a wide range of Riemannian radial distributions that  the optimal parameter estimation rate is root-$n$ up to logarithmic terms. This rate aligns with the classical results in Euclidean cases, indicating that the sample size $n$ plays a similar role in determining estimation accuracy in manifold settings. However, such theory does not tell us what the model complexity is, such as the dependence on the dimension $m$ of the manifold. Thus, it gives limited insights on how  geometry may affect the complexity of a statistical model. Motivated by this issue, this section uses von Mises-Fisher distributions on spheres as an example\footnote{Characterizing the tight model complexity of general Riemannian radial distributions on Riemannian symmetric spaces is challenging, since it depends on both the manifold structure and the function $\phi$. Therefore, we do not touch the general cases in this paper, and leave relevant study to future works.}, and establishes its model complexity.

\subsection{Formulation}

Let $\cM=\SSS^m$ be the unit $m$-sphere, and recall that the von Mises-Fisher distribution on $\SSS^m$ is given by
\#\label{equ:vMF-sec7}
f_{\vMF}(x;\alpha,\beta)=\frac{1}{Z_{\vMF}(\beta)}e^{\beta \cos d_g(x,\alpha)},\quad Z_{\vMF}(\beta)=\int_{\SSS^m}e^{\beta \cos d_g(x,\alpha)}\dvol(x),
\#
where $d_g$ is the geodesic distance, $\alpha\in\SSS^m$ is the location parameter, and $\beta$ is the dispersion parameter. For simplicity, we assume that $\beta$ is known {\it a priori} and our aim is to estimate $\alpha$. Then the MLE of $\alpha$, based on $n$ independent samples $\{x_i\}_{i=1}^n$  drawn from $f_{\vMF}(x;\alpha,\beta)$, is given by the following M-estimator:
\$
\hat\alpha^{\MLE}=\argmax_{\alpha\in\SSS^m}\sum_{i=1}^n\cos d_g(x_i,\alpha).
\$
It is natural to ask what the convergence rate of the MLE is and how it depends on the dimension $m$ of the sphere. In addition, one may ask whether this dependence is optimal in the minimax sense. Answering these questions is the main objective of this section. Specifically,  we will establish the convergence rate of the MLE in Section \ref{sec:7.1}. We will then  derive a compatible minimax lower bound in Section \ref{sec:7.2}, which demonstrates the optimality of the MLE in the minimax sense. 



\subsection{Maximum likelihood estimation}\label{sec:7.1}


This section derives the convergence rate of the MLE. To that end, we first provide entropy estimates for the following functional class:
\#\label{equ:functional-class-vMF}
\cF=\{f_{\vMF}(x;\alpha,\beta)\mid\alpha\in\SSS^m\}.
\#
The result is presented in the following theorem. 

\begin{theorem}\label{thm:7.1}
    Let  $f_{\vMF}(x;\alpha,\beta)$ be the von Mises-Fisher distribution on $\SSS^m$ and $\cF$ the functional class in \eqref{equ:functional-class-vMF}. Then for any sufficiently small $\epsilon$, the following entropy estimates hold:
    \$
    \log\cN(\epsilon,\cF,d_{\infty})\leq 2m\log(\frac{1}{\epsilon}), \\
    \log\cN_B(\epsilon,\cF,d_{1})\leq 2m\log(\frac{1}{\epsilon}),\\
    \log\cN_B(\epsilon,\cF,d_{h})\leq 2m\log(\frac{1}{\epsilon}),
    \$
    where $d_\infty$, $d_1$, and $d_h$ denote the $L^{\infty}$, $L^1$, and hellinger distance, respectively. 
\end{theorem}


After we obtain the entropy estimate in Theorem \ref{thm:7.1}, we can apply empirical process theory to derive the convergence rate of the MLE. This is presented in the following corollary. 

\begin{corollary}\label{corollary:7.2}
    Let $f_{\vMF}(x;\alpha,\beta)$ be the von Mises-Fisher distribution on $\SSS^m$, and $\hat\alpha^{\MLE}$  the MLE of $\alpha$ based on $n$ independent samples drawn from $f_{\vMF}(x;\alpha^{\tr},\beta)$. Then for sufficiently large $n$, it holds with probability at least $1-ce^{-cm\log^2n}$ that
    \#
    d_h(f(x;\hat\alpha^{\MLE},\beta),f(x;\alpha^{\tr},\beta))&\leq \frac{C\sqrt{m}\log n}{\sqrt{n}},\label{equ:dh-cor-7.2}\\
    d_1(f(x;\hat\alpha^{\MLE},\beta),f(x;\alpha^{\tr},\beta))&\leq \frac{2C\sqrt{m}\log n}{\sqrt{n}},\label{equ:d1-cor-7.2}
    \#
    where $d_h$ and $d_1$ represent the hellinger and $L^1$ distances, respectively, and $c,C>0$ are constants independent of $n$ and $m$.  
\end{corollary}


Furthermore, by studying the identifiability of the parameter $\alpha$ in $f_{\vMF}(x;\alpha,\beta)$, we obtain the parameter estimation rate of the MLE as follows. 

\begin{theorem}\label{thm:7.3}
    Let $f_{\vMF}(x;\alpha,\beta)$ be the von Mises-Fisher distribution on $\SSS^m$, and $\hat\alpha^{\MLE}$  the MLE of $\alpha$ based on $n$ independent samples drawn from $f_{\vMF}(x;\alpha^{\tr},\beta)$. Then for sufficiently large $n$, it holds with probability at least $1-ce^{-cm\log^2n}$ that
    \$
    d_g(\hat\alpha^{\MLE},\alpha^{\tr})\leq \frac{Cm\log n}{\sqrt{n}},
    \$
    where $c,C>0$ are constants independent of $n$ and $m$. 
\end{theorem}

\begin{remark}
    Observe that the hellinger distance rate in Corollary \ref{corollary:7.2} and the parameter estimation rate in Theorem \ref{thm:7.3} differ by a root-$m$ term. Such difference stems from the parameter  identifiability for the von Mises-Fisher distribution, which highlights the effects of geometry on the precision of statistical estimates.  
\end{remark}

\subsection{Minimax lower bounds}\label{sec:7.2}

To assess the optimality of the model complexity established in Theorem \ref{thm:7.3}, we derive a compatible minimax lower bound  using the Fano's method. Specifically, Let $f_{\vMF}(x;\alpha,\beta)$ be the von Mises-Fisher distribution on $\SSS^m$. Let $\{x_i\}_{i=1}^n$ be $n$ independent samples drawn from $f_{\vMF}(x;\alpha,\beta)$ with an unknown $\alpha\in\SSS^m$. We define the minimax risk $\cR_{n,\vMF}(\beta)$ of the parameter estimation problem as follows:
    \#\label{equ:risk-sec73}
    \cR_{n,\vMF}(\beta)=\inf_{\hat\alpha}\sup_{\alpha}\EE[d_g(\hat\alpha(x_1,\ldots,x_n),\alpha)],
    \#
    where we take the infimum over all estimators $\hat\alpha$, take the supremum over all feasible parameters $\alpha\in\SSS^m$, and take the expectation over the independent samples $\{x_i\}_{i=1}^n$. Our objective is to lower bound this minimax risk, and our main result is the following theorem.

\begin{theorem}\label{thm:7.4}
     Let $f_{\vMF}(x;\alpha,\beta)$ be the von Mises-Fisher distribution defined on the $m$-sphere $\SSS^m$. Let $\cR_{n,\vMF}(\beta)$ be the minimax risk defined in \eqref{equ:risk-sec73}. Then for sufficiently large $n$, we have
     \$
     \cR_{n,\vMF}(\beta)\geq Cmn^{-1/2},
     \$
     where $C>0$ is a constant independent of $n$ and $m$. 
\end{theorem}

By comparing the minimax lower bound with the convergence rate of the MLE in Theorem \ref{thm:7.3}, we conclude that the optimal convergence rate is $\tilde\Theta(mn^{-1/2})$, where $\tilde\Theta(\cdot)$ omits logarithmic terms, and that the MLE is optimal in the minimax sense. Also, by comparing this rate with the Gaussian distributions in the Euclidean space $\RR^m$, whose minimax optimal rate is $\Theta(\sqrt{m}/\sqrt{n})$, we find that there is an additional root-$m$ term for the von Mises-Fisher distribution. This demonstrates that the geometry can affect the statistical estimation accuracy, and it is promising to draw more conclusions on the impacts of geometry on statistics in future research. 





\section{Conclusion}\label{sec:conclusion}

Manifold data analysis is challenging partially due to the lack of parametric distributions on manifolds. In response, this paper introduces a series of Riemannian radial distributions on Riemannian symmetric spaces. By using homogeneity and symmetry, we establish favorable properties of these parametric distributions, making them a promising candidate for statistical modeling and algorithm design. In addition, we introduce a novel theory for parameter estimation and minimax optimality, leveraging statistics, Riemannian geometry, and Lie theory. We show that for a wide range of Riemannian radial distributions, the convergence rate of the MLE is root-$n$ up to logarithmic terms. We also derive a root-$n$ minimax lower bound for the parameter estimation rate when $\cM$ is a simply connected Riemannian symmetric space. This demonstrates the optimality of the MLE in a variety of contexts. Extensions to (1) Riemannian radial distributions with an unknown temperature parameter and (2)  model complexity of von Mises-Fisher distributions are also provided in this paper. To conclude, let us highlight several promising directions for future exploration.  
\begin{itemize}
    \item It is promising to use Riemannian radial distributions to build sophisticated statistical models and design efficient algorithms. Potential examples include regression models \citep{cornea2017regression}, mixture models \citep{banerjee2005clustering}, and differential privacy mechanisms \citep{reimherr2021differential,jiang2023gaussian}. 

    \item It is intriguing to develop the concept of {\it covariance matrix} on Riemannian symmetric spaces. In Euclidean spaces, Gaussian distributions with general covariance matrices are much more useful than those with isotropic covariance matrices. Therefore, extending Riemannian radial distributions to include more complex dependence structures is a promising direction. 

    \item It is interesting to investigate the role of Riemannian radial distributions in hypothesis testing over manifolds. Such procedures may provide efficient tools for  uncertainty quantification. 

    \item Technically, it is interesting to examine  minimax lower bounds for parameter estimation when $\cM$ is not a simply connected Riemannian symmetric space. New theoretical tools are needed as the structure of the cut locus is more complicated. Besides, it is promising to investigate the model complexity of more complex statistical models on manifolds. This will reveal how geometry may affect the model complexity.

\end{itemize}

\bibliographystyle{apalike}
\bibliography{references}

\newpage

\appendix
\part{Appendix} 
\parttoc

\section{Differential geometry}

This section gives an additional introduction to differential geometry. Section \ref{sec:smoothmanifolds} introduces basic concepts in smooth manifolds, and we refer readers to \citet{lee2012introduction} for a more complete introduction. 
Section \ref{sec:apd-a2} reviews simply connected compact Riemannian symmetric spaces from the perspective of Lie theory. This introduction will follow the review presented in \cite{said2021riemannian}, and we refer readers to \cite{helgason2001differential} for more  details.

\subsection{Smooth manifolds}\label{sec:smoothmanifolds}


A smooth manifold $\cM$ of dimension $m$ is a topological manifold $\cM$ equipped with a maximal family of injective mappings $\varphi_{\alpha}:U_{\alpha}\subseteq\cM\to\RR^{m}$ such that \\
\indent (1) $\varphi_\alpha$ is a homeomorphism from an open set $U_{\alpha}$ to an open subset $\varphi(U_{\alpha})\subseteq\RR^m$.\\    
\indent (2) For any $\alpha,\beta$ with a non-empty $W\coloneqq U_{\alpha}\cap U_{\beta}$, the mapping $\varphi_{\beta}\circ\varphi_{\alpha}^{-1}$ is a differentiable mapping of $\varphi_{\alpha}(W)$ onto $\varphi_{\beta}(W)$.\\
\indent (3) $\bigcup_{\alpha}U_{\alpha}=\cM$.\\
The collection $\{(U_\alpha,\varphi_\alpha)\}$ satisfying the above properties is called a smooth structure on $\cM$.
The pair $(U_{\alpha},\varphi_{\alpha})$ is referred to as an open chart or a local coordinate system on $\cM$. If $x\in U_{\alpha}$ and $\varphi_{\alpha}(x)=(\xk_1(x),\ldots,\xk_m(x))$, the set $U_{\alpha}$ is called a  coordinate neighborhood of $x$ and the numbers $\xk_i(x)$ are called local coordinates of $x$. 

Let $\cM$ and $\cN$ be smooth manifolds. A mapping $f:\cM\to\cN$ is called differentiable at $x\in\cM$ if given a local chart $(V,\psi)$ at $f(x)$ there exists a local chart $(U,\varphi)$ at $x$ such that $ f(U)\subseteq V$ and $\psi\circ f\circ\varphi^{-1}:\varphi (U)\to\psi(f(U))$ is differentiable at $\varphi(x)$. The mapping $f$ is called  differentiable if it is differentiable at every point $x\in\cM$. A differentiable mapping $f$ is called a diffeomorphism if it is bijective and its inverse $f^{-1}$ is differentiable. Let $C^{\infty}(\cM,\cN)$ be the set of all differentiable mappings of $\cM$ onto $\cN$. In particular, when $\cN=\RR$, $C^{\infty}(\cM,\RR)$ is the set of all differentiable functions on $\cM$, which we denote by $C^{\infty}(\cM)$. In addition, we let $C^{\infty}_x(\cM)$ be the set of all real-valued functions on $\cM$ that are differentiable at $x\in\cM$. 

Let $\cM$  be a smooth manifold. A tangent vector at $x\in\cM$ is a linear mapping $v: C^{\infty}_x(\cM)\to\RR$ such that $v(fg)=f(x)v(g)+g(x)v(f)$ for all $f,g\in C^{\infty}_x(\cM)$. The set $T_{x}\cM$ of all tangent vectors at $x$ forms an $m$-dimensional vector space, called the tangent space to $\cM$ at $x$. Let $(U,\varphi)$ be a local chart  at $x$ and $\varphi(y)=(\xk_1(y),\ldots,\xk_m(y))$ be the coordinate representation. Then $\{\frac{\partial}{\partial\xk_i}|_{x}\}_{1\leq i\leq m}$ gives a set of basis vectors of the tangent space $T_{x}\cM$, where $\frac{\partial}{\partial \xk_{i}}$ is defined by
\$
\frac{\partial}{\partial\xk_i}|_{y}f=\frac{\partial\left(f\circ \varphi^{-1}\right)}{\partial \xk_i}|_{\varphi(y)},\quad\forall f\in C^{\infty}_y(\cM),\quad\forall y\in U.
\$
Let $\cN$ be a smooth manifold and $\varphi:\cM\to\cN$ be a differentiable mapping. The differential of $\varphi$ at $x$ is a linear mapping $d\varphi_x:T_{x}\cM\to T_{\varphi(x)}\cN$ such that
$d\varphi_x(v)(f)=v(f\circ \varphi)$ holds for all $v\in T_{x}\cM$ and $f\in C^{\infty}(\cN)$. 

Let $\cM$ be a smooth manifold. A family of open sets $U_{\alpha}\subseteq\cM$ with $\bigcup_{\alpha} U_{\alpha}=\cM$ is called locally finite if every  $x\in\cM$ has a neighborhood $W$ such that $W\cap V_{\alpha}$ is non-empty for only a finite number of indices. The  support of $\rho:\cM\to\RR$ is the closure of the set of points where $\rho$ is different from zero. A family of differentiable functions $\rho_{\alpha}:\cM\to\RR$ is called a smooth partition of unity if:\\
\indent (1) For all $\alpha$, $\rho_{\alpha}\geq0$ and the support of $\rho_{\alpha}$ is contained in a coordinate neighborhood $U_{\alpha}$ of a smooth structure $\{(U_{\beta},\varphi_{\beta})\}$ of $\cM$.\\
\indent (2) The family $\{U_{\alpha}\}$ is locally finite.\\
\indent (3) $\sum_{\alpha}\rho_{\alpha}(x)=1$ for all $x\in\cM$.\\
The following theorem guarantees the existence of the smooth partition of unity. 

\begin{theorem}[Theorem 5.6, Chapter 0 in \citet{do1992riemannian}]

    A smooth manifold $\cM$ has a smooth partition of unity if and only if every connected component of $\cM$ is Hausdorff and has a countable basis.

\end{theorem}




\subsection{Simply connected compact Riemannian symmetric spaces}\label{sec:apd-a2}

This section reviews simply connected compact Riemannian symmetric spaces from the perspective of Lie theory. It allows us to present several geometric formulae which are useful in the minimax analysis in Section \ref{sec:5.2}. In the review, we will first present the connection to Lie theory in Section \ref{sec:a21}. We will then introduce an integral formula in Section \ref{sec:a22}, followed by an examination of the squared distance function, including the formulae for its gradient and Hessian, in Section \ref{sec:a23}. Finally, in Section \ref{sec:a24}, we will discuss some results on the union of the cut loci. 
For clarity, our review will focus on the results while omitting technical derivations. For further details, we direct readers to \citet{said2021riemannian} and \cite{helgason2001differential}.


\subsubsection{Connection to Lie theory}\label{sec:a21}


Suppose $(\cM,g^{\cM})$ is a simply connected compact Riemannian symmetric space. Then the identity component $G$ of the isometry group of $\cM$ is a compact, simply connected Lie group. Given some $o\in\cM$, the isotropy subgroup $K$ of $o$ in $G$ is a closed subgroup of $G$. In addition, $\cM\simeq G/K$ as a Riemannian homogeneous space. 

The Riemannian geometry of $\cM$ can be given in algebraic terms. Let $\mathfrak{g}$ and $\mathfrak{t}$ be the Lie algebras of $G$ and $K$, and $B$ the Killing form on $\mathfrak{g}$. Then $\mathfrak{g}=\mathfrak{t}+\mathfrak{p}$ is a direct sum with $\mathfrak{p}$ being the orthogonal complement of $\mathfrak{t}$ with respect to $B$. The tangent space $T_o\cM$ can be naturally identified with $\mathfrak{p}$, and the Riemannian metric of $\cM$ is determined (up to re-scaling) by 
\#\label{equ:a1}
g^{\cM}_o(u,v)=-B(u,v),\quad\forall u,v\in T_{o}\cM\simeq \mathfrak{p}.
\#
Moreover, the Riemannian curvature tensor $R_o$ is given by
\#\label{equ:a2}
R_o(u,v)w=-[[u,v],w],\quad\forall u,v,w\in T_{o}\cM\simeq \mathfrak{p},
\#
where $[\cdot,\cdot]$ is the Lie bracket. The right hand side of \eqref{equ:a2} is in $\mathfrak{p}$ as $\cM$ is a Riemannian symmetric space. We can define $T_u$ for each $u\in\mathfrak{p}$ as the following operator
\#\label{equ:a3}
T_u: v\in\mathfrak{p}\to R_o(u,v)u=[u,[u,v]]\in\mathfrak{p}.
\#

Both \eqref{equ:a1} and \eqref{equ:a3} can be simplified using a root system. 
Let $\mathfrak{a}$ be a maximal Abelian subspace of $\mathfrak{p}$, and $\Delta_+$ the set of positive roots
of $\mathfrak{g}$ with respect to $\mathfrak{a}$. Then we can rewrite  $u\in\mathfrak{p}$ as $u=\Ad(k)a$ for some $k\in K$ and $a\in \mathfrak{a}$, where $\Ad$ is the adjoint representation. With this notation, the Riemannian metric of $\cM$ satisfies
\#\label{equ:a4}
g^{\cM}_o(u,u)=\sum_{\lambda\in\Delta_+}m_\lambda (\lambda(a))^2,
\#
where $m_\lambda$ is the multiplicity of the root $\lambda$. In addition, the operator $T_u$ satisfies 
\#\label{equ:a5}
T_u(v)=-\sum_{\lambda\in\Delta_+}(\lambda(a))^2\Pi_\lambda^k(v),\quad \Pi_\lambda^k=\Ad(k)\circ \Pi_\lambda\circ \Ad(k^{-1}),
\#
where $\Pi_{\lambda}$ is an orthogonal projector from $\mathfrak{p}$ onto the eigenspace of $T_{a}$ associated with the eigenvalue $-(\lambda(a))^2$. Define $\Pi_{\mathfrak{a}}$ as the orthogonal projector from $\mathfrak{p}$ onto $\mathfrak{a}$ and $\Pi_{\mathfrak{a}}^k=\Ad(k)\circ\Pi_{\mathfrak{a}}\circ\Ad(k^{-1})$. These definitions are useful in deriving the integral formula and the formula for the Hessian of the squared distance function, which we will discuss in Section \ref{sec:a22} and Section \ref{sec:a23}. 

\subsubsection{Integral formula}\label{sec:a22}

This section introduces a useful formula for the integral $\int_\cM f(x)\dvol(x)$. To evaluate this integral, it is desirable to parameterize $x\in\cM$ by the so-called polar coordinates $k\in K$ and $a\in \mathfrak{a}$, 
\$
x=\Exp_o(\Ad(k)a),
\$
where $\Exp_o$ is the Riemannian exponential mapping. However, this parametrization is not unique. To turn it into a unique parametrization, we let $K_{\mathfrak{a}}$ be the centralizer of $\mathfrak{a}$ in $K$ and set $S=K/K_{\mathfrak{a}}$. Also, we let $C_+$ be the set of $a\in\mathfrak{a}$ such that $\lambda(a)\in(0,\pi)$ for each $\lambda\in\Delta_+$. Then, it is useful to consider the following mapping,
\#\label{equ:a6}
\varphi(s,a)=\Exp_o\circ\beta(s,a), \quad (s,a)\in S\times \bar C_+,
\#
where $\beta(s,a)=\Ad(s)a$ and $\bar C_+$ is the closure of $C_+$. Indeed, $\varphi$ maps $S\times \bar C_+$ onto the whole area of $\cM$, and is a diffeomorphism from $S\times C_+$ to the set $\cM_r$ of regular values of $\varphi$. The set $\cM-\cM_r$ is of measure zero. Thus, given a measurable function $f$ on $\cM$, the integral of $f$ can be written as
\$
\int_{\cM}f(x)\dvol(x)=\int_{\cM_r}f(x)\dvol(x).
\$
Since $\varphi$ is a diffeomorphism, this integral can be further rewritten as
\$
\int_{\cM}f(x)\dvol(x)=\int_{C_+}\int_Sf(s,a)D(a)dad\omega(s),
\$
where $f(s,a)=f\circ \varphi(s,a)$, $D(a)$ is the Jacobian determinant of $\varphi$, and $d\omega$ is the invariant volume density on $S$ induced from $K$. It has been derived in \cite{said2021riemannian} that
\$
D(a)=\prod_{\lambda\in\Delta_+}(\sin \lambda(a))^{m_\lambda}.
\$
Thus, we obtain the following integral formula on $\cM$:
\#\label{equ:a7}
\int_{\cM}f(x)\dvol(x)=\int_{C_+}\int_Sf(s,a)\prod_{\lambda\in\Delta_+}(\sin \lambda(a))^{m_\lambda}dad\omega(s).
\#

\subsubsection{The squared distance function}\label{sec:a23}

Fix $x\in\cM$, and consider the function $h_x(y)=\frac{1}{2}d^2(x,y)$. This function is $C^2$ differentiable near $o$ if $x\notin \Cut(o)$, where $\Cut(o)$ denotes the cut locus of $o$. The objective of this section is to derive the closed form formulae for the gradient and Hessian of the function $h_x$ at $o$. Recall that given a function $f$ with desired differentiability, 
\begin{itemize}
    \item The gradient $\nabla_y f$ of a function
    $f$ at $y$ is a tangent vector in $T_y\cM$ such that
\#\label{equ:grad}
g^{\cM}_y(\nabla_y f,u)=u(f),\quad \forall u\in T_y\cM.
\#
\item The Hessian $\Hess_y f$ of a function $f$ at $y$ is given by 
\#\label{equ:Hess}
\Hess_yf(u,v)=g^{\cM}_y(\nabla_u\nabla f,v),\quad \forall u,v\in T_y\cM,
\#
where $\nabla f$ is the gradient field of $f$ defined near $y$, and $\nabla_u\nabla f$ is the covariant derivative of $\nabla f$ in the direction of $u$. 
\end{itemize}
To express the gradient $G_o(x)$ and Hessian $H_o(x)$ of the function $h_x$ at $o$, we write $x=\varphi(s,a)$ in the notation of \eqref{equ:a6}. If $x\notin\Cut(o)$, then we have
\#
G_o(x)&=-\beta(s,a),\label{equ:a8}\\
H_o(x)&=\Pi_\mathfrak{a}^s+\sum_{\lambda\in\Delta_+}\lambda(a)\cot \lambda(a)\Pi_\lambda^s,\label{equ:a9}
\#
where $\Pi_\mathfrak{a}^s$ and $\Pi_\lambda^s$ are the orthogonal projectors defined in \eqref{equ:a5} and $\cot$ is the cotangent function. 
Both the integral formula \eqref{equ:a7}  and Hessian formula \eqref{equ:a9} are derived using the closed form solutions of the Jacobi equations. We refer readers to \cite{said2021riemannian} for more details.

\subsubsection{Cut locus}\label{sec:a24}

Lastly, we present a lemma on the cut locus which is useful in our minimax analysis. We know that the cut locus of a point $\Cut(x)$ is of measure zero. Lemma \ref{lma:a2} below strengthens this statement in the case of simply connected compact Riemannian symmetric spaces. Rather than one cut locus, it considers the union of cut loci along a geodesic. 

\begin{lemma}\label{lma:a2}
    Assume $(\cM,g^{\cM})$ is a simply connected compact Riemannian symmetric space. Let $\gamma:I\to\cM$ be a geodesic defined on a compact interval $I$. Let $\Ucut(\gamma)$ denote  the union of all cut loci $\Cut(\gamma(t))$ for $t\in I$. Then $\Ucut(\gamma)$ is a set of measure zero. 
\end{lemma}

The assumption of simply connectedness is essential in this lemma. For example, if $\cM$ is a real projective space, which is not simply connected, then the statement in Lemma \ref{lma:a2} does not hold. 


%

\section{Proofs of results in Section \ref{sec:3}}

\subsection{Proof of Proposition \ref{prop:3.1}}

\begin{proof}
    We will treat both cases separately. \begin{itemize}
        \item {\bf Case 1.} When $\cM$ is compact, any bounded and measurable function on $\cM$ is integrable. Since $\phi$ is a nonnegative, continuous, and increasing function, the function $\tilde f$ is bounded and measurable. Thus, it is integrable on $\cM$. 
        \item {\bf Case 2.} When $\cM$ is noncompact, we can use Theorem \ref{thm:2.3} to prove the results. As $\tilde f$ is nonnegative and measurable, to prove its integrability, it suffices to show that 
    \$
    \int_{\cM}\tilde f(x;\alpha,\phi)\dvol(x)<\infty.
    \$
    To prove this, we will use the polar coordinate chart at $\alpha$, and rewrite the integral as
    \$
    \int_\cM\tilde f(x;\alpha,\phi)\dvol(x)&=\int_{T_\alpha\cM} \tilde f^{\flat}\lambda(r,\Theta)drd\Theta\\
    &\overset{({\rm i})}{=}\int_{T_\alpha\cM}e^{-\phi(r)}\lambda(r,\Theta)drd\Theta,
    \$
    where $\tilde f^{\flat}=\tilde f\circ\Exp_\alpha$ with $\Exp_\alpha$ being the exponential map, and (i) uses the definition of $\tilde f$. Since the sectional curvatures of $\cM$ are larger than $\kappa_{\min}$, we have  $\lambda(r,\Theta)\leq \sn_{\kappa_{\min}}^{m-1}(r)$ by Theorem~\ref{thm:2.3}. It follows that 
    \$
    \int_\cM\tilde f(x;\alpha,\phi)\dvol(x)&\leq \int_{T_\alpha\cM}e^{-\phi(r)}\cdot \sn_{\kappa_{\min}}^{m-1}(r)drd\Theta\\
    &=\vol(\SSS^{m-1})\cdot\int_{0}^{\infty}e^{-\phi(r)}\cdot\sn_{\kappa_{\min}}^{m-1}(r)dr<\infty,
    \$
    where $\SSS^{m-1}$ is the unit $(m-1)$-sphere and the last inequality follows from  condition \eqref{equ:3.2}. 
    \end{itemize}
    The proof is complete by combining these two cases.
\end{proof}

\subsection{Proof of Proposition \ref{prop:3.2}}

\begin{proof}
    Suppose $\tilde f(x;\alpha,\phi)$ is integrable over $\cM$ for some $\alpha\in\cM$. We will first show that $\tilde f(x;\tilde\alpha,\phi)$ is integrable over $\cM$ for all $\tilde\alpha\in\cM$. 
    Since $\tilde f(x;\tilde\alpha,\phi)$ is nonnegative and measurable, it suffices to show that its integral over $\cM$ is finite. Indeed, since $\cM$ is a Riemannian homogeneous space, there is an isometry $F\in\Iso(\cM)$ such that $F(\alpha)=\tilde\alpha$. As a result,
    \$
    \int_{\cM}\tilde f(x;\tilde\alpha,\phi)\dvol(x)&=\int_{\cM}e^{-\phi(d_g(x,\tilde\alpha))}\dvol(x)\\
    &\overset{({\rm i})}{=}\int_{\cM}e^{-\phi(d_g(x,F(\alpha)))}\dvol(x)\\
    &\overset{({\rm ii})}{=}\int_{\cM}e^{-\phi(d_g(F^{-1}(x),\alpha))}\dvol(x)\\
    &\overset{({\rm iii})}{=}\int_{\cM}e^{-\phi(d_g(x,\alpha))}\dvol(x)<\infty,
    \$
    where (i) uses the fact that $\tilde\alpha=F(\alpha)$, (ii) uses the fact that $d_g(x,F(\alpha))=d_g(F^{-1}(x),\alpha)$ when $F$ is an isometry, and (iii) uses equality \eqref{equ:homo} in Proposition \ref{prop:homogeneous}. This proves the first and the second properties in Proposition \ref{prop:3.2}. Since the normalizing constant 
    \$
    Z(\alpha,\phi)=\int_{\cM}\tilde f(x;\alpha,\phi)\dvol(x)
    \$ 
    is independent of $\alpha$, we immediately obtain the third property in Proposition \ref{prop:3.2} using the definition of the MLE. This concludes the proof. 
\end{proof}

\subsection{Proof of Proposition \ref{prop:3.3}}

\begin{proof}
    To prove this statement, it suffices to show that if 
    \#\label{equ:b1}
    \int_{\cM} d_g^p(x,\tilde \alpha)f(x;\alpha,\phi)\dvol(x)\leq \int_{\cM}d_g^p(x,b)f(x;\alpha,\phi)\dvol(x),\quad \forall b\in\cM,
    \#
    then $\tilde\alpha=\alpha$. Suppose on the contrary that $\tilde\alpha$ satisfies \eqref{equ:b1} but $\tilde\alpha\neq\alpha$. Then we have
    \$
    \int_{\cM} d_g^p(x,\tilde \alpha)f(x;\alpha,\phi)\dvol(x)\leq \int_{\cM}d_g^p(x,\alpha)f(x;\alpha,\phi)\dvol(x).
    \$
    Ignoring the normalizing constant $Z(\phi)$, we have
    \#\label{equ:b2}
    I(\tilde\alpha,\alpha)\leq I(\alpha,\alpha)<\infty,
    \#
    where 
    \$
    I(a,b)=\int_{\cM}d_g^p(x,a)e^{-\phi(d_g(x,b))}\dvol(x),\quad \forall a,b\in\cM. 
    \$
    By property \eqref{equ:homo}, we can show that for any isometry $F\in\Iso(\cM)$,
    \#\label{equ:b3}
    I(a,b)=I(F(a),F(b)),\quad\forall a,b\in\cM.
    \#
    Since $\cM$ is a Riemannian symmetric space and $\tilde\alpha\neq \alpha$, we can find an isometry $F\in\Iso(\cM)$ such that $F(\alpha)=\tilde\alpha$ and $F\circ F$ is the identity, as a consequence of Proposition \ref{prop:symmetric}. Using this $F$ in \eqref{equ:b2} and using \eqref{equ:b3}, we have
    \$
    I(\alpha,\tilde\alpha)=I(F(\tilde\alpha),F(\alpha))=I(\tilde\alpha,\alpha)\leq I(\alpha,\alpha)=I(\tilde\alpha,\alpha).
    \$
    Therefore, 
    \$
    I(\alpha,\tilde\alpha)+I(\tilde\alpha,\alpha)\leq I(\alpha,\alpha)+I(\tilde\alpha,\tilde\alpha). 
    \$
    Equivalently, we have
    \#\label{equ:b4}
    \int_{\cM}D(x;\alpha,\tilde\alpha)\dvol(x)\leq 0,
    \#
    where 
    \$
    D(x;\alpha,\tilde\alpha)=D_1(d_g(x,\alpha),d_g(x,\tilde\alpha)),
    \$
    and
    \$
    D_1(u,v)=u^pe^{-\phi(v)}+v^pe^{-\phi(u)}-u^pe^{-\phi(u)}-v^pe^{-\phi(v)}.
    \$
    By calculation, we find that 
    \#\label{equ:b5}
    D_1(u,v)=(u^p-v^p)\cdot (e^{-\phi(v)}-e^{-\phi(u)})\geq 0,
    \#
    where we use the fact that $\phi$ is an increasing function. Since $\phi$ is strictly increasing, we know that the equality in \eqref{equ:b5} holds if and only if $u=v$. Combining this with \eqref{equ:b4}, we obtain
    \$
    d_g(x,\alpha)=d_g(x,\tilde\alpha),\quad\forall x\in\cM,
    \$
    where we use the continuity of the distance function $d_g(\cdot,\cdot)$. It implies that $\tilde\alpha=\alpha$ and concludes the proof. 
\end{proof}





    


    

\section{Technical lemmas for Section \ref{sec:4}}

This section presents lemmas used in Section \ref{sec:4}. Section \ref{sec:C31} gives a metric entropy estimate for a geodesic ball in a Riemannian homogeneous space. Section \ref{sec:C32} establishes a Lipschitz property of a Riemannian radial distribution $f(x;\alpha,\phi)$ with respect to its parameter $\alpha$ under certain conditions. Section \ref{sec:C3} establishes the identifiability of the parameter $\alpha$ in the density function $f(x;\alpha,\phi)$ under certain conditions. Section \ref{sec:c4} proves the claim \eqref{equ:4.13}, which is used in the proof of Theorem \ref{thm:4.3}. Section \ref{sec:c5} establishes a geometric lemma that is used in Section \ref{sec:c4}.  

\subsection{Basic entropy estimates}\label{sec:C31}

\begin{lemma}\label{lma:c1}
    Let $(\cM,g^{\cM})$ be a Riemannian homogeneous space and $\cX=\cB_{\cM}(\alpha^*,D)$ a geodesic ball in $\cM$, given by \eqref{equ:ball}. Then for sufficiently small $\epsilon$, we have
    \$
    \cN(\epsilon,\cX,d_g)\lesssim \epsilon^{-m}.
    \$
    where $d_g$ is the geodesic distance and $\lesssim$ omits constants independent of $\epsilon$. 
\end{lemma}

\begin{proof}
    Construct an $\epsilon$-net $\cS$ of $\cX$ by the following greedy procedure. Let $x_1\in\cX$ be an arbitrary point. Suppose $x_1,\ldots,x_s$ have been chosen. If the set $\{x\in\cX\mid d_g(x,x_i)>\epsilon,i\leq s\}$ is not empty, let $x_{s+1}$ be an arbitrary point in this set. Else end the construction of the set $\cS$. Such set $\cS$ is easily shown to be an $\epsilon$-net of $\cX$. 

    It suffices to upper bound the cardinality of the $\epsilon$-net $\cS$. To this end, we observe that the distance between any two points $x_i$ and $x_j$ in $\cS$ is at least $\epsilon$. It implies that the geodesic balls $\cB_{\cM}(x_i,\epsilon/2)$ and $\cB_{\cM}(x_j,\epsilon/2)$ are disjoint. Therefore, the volume of $\cup_{x_i\in\cS}\cB_{\cM}(x_i,\epsilon/2)$ is equal to the sum of the volumes of these geodesic balls. Recall that $\cX=\cB_{\cM}(\alpha^*,D)$. We have 
    \$
    \cup_{x_i\in\cS}\cB_{\cM}(x_i,\epsilon/2)\subseteq\cB_{\cM}(\alpha^*,D+\epsilon/2). 
    \$
    Considering the volumes of these two regions, we obtain
    \#\label{equ:|S|}
    |\cS|\cdot \vol(\cB_{\cM}(x,\epsilon/2))\leq \vol(\cB_{\cM}(\alpha^*,D+\epsilon/2)),
    \#
    where we use the fact that $\vol(\cB_{\cM}(x,\epsilon))$ is independent of $x$ as $\cM$ is a Riemannian homogeneous space. It remains to lower bound the volume of the geodesic ball $\cB_{\cM}(x,\epsilon/2)$ and upper bound the volume of $\cB_{\cM}(\alpha^*,D+\epsilon/2)$. 

    Since $\cM$ is a Riemannian homogeneous space, its injectivity radius $\inj(\cM)$ is larger than zero and its sectional curvatures are bounded within $[\kappa_{\min},\kappa_{\max}]$. Here $\kappa_{\max}$ is assumed to be positive without loss of generality. Let $\epsilon^*$ be a constant such that 
    \$
    \epsilon^*/2\leq \min\{1,\inj(\cM)\},\quad\textnormal{and }\sn_{\kappa_{\max}}(\epsilon^*/2)\geq\epsilon^*/4,
    \$
    where $\sn_{\kappa_{\max}}(\cdot)$ is given by \eqref{equ:sn}. Then for all $\epsilon\leq\epsilon^*$, we have
    \#\label{equ:Ball-D}
    \vol(\cB_{\cM}(\alpha^*,D+\epsilon/2))\leq\vol(\cB_{\cM}(\alpha^*,D+1))<\infty.
    \#
    In addition, we have
    \$
    \vol(\cB_{\cM}(x,\epsilon/2))=\int_{\SSS^{m-1}}\int_{0}^{\epsilon/2}\lambda(r,\Theta)dr d\Theta,
    \$
    where we use the polar coordinate representation \eqref{equ:integral} of the integral and the fact that $\epsilon\leq 2\inj(\cM)$. By the volume comparison theorem, Theorem \ref{thm:2.3}, we have
    \$
    \vol(\cB_{\cM}(x,\epsilon/2))\geq \int_{\SSS^{m-1}}\int_0^{\epsilon/2}\sn_{\kappa_{\max}}^{m-1}(r)drd\Theta.
    \$
    Since $\epsilon\leq\epsilon^*$ and $\sn_{\kappa_{\max}}(r)/r$ is a decreasing function, we have $\sn_{\kappa_{\max}}(r)\geq r/2$ for $r\leq \epsilon/2$, and thus 
    \#\label{equ:Ball-epsilon}
    \vol(\cB_{\cM}(x,\epsilon/2))&\geq \frac{\vol(\SSS^{m-1})}{{2^{m-1}}}\cdot\int_0^{\epsilon/2}r^{m-1}dr\notag\\
    &=\frac{\vol(\SSS^{m-1})}{m2^{2m-1}}\cdot \epsilon^{m}.
    \#
    Combining \eqref{equ:|S|}, \eqref{equ:Ball-D}, and \eqref{equ:Ball-epsilon}, we have 
    \$
    |S|\leq \frac{m2^{2m-1}\cdot \vol(\cB_{\cM}(\alpha^*,D+1))}{\vol(\SSS^{m-1})}\cdot\epsilon^{-m},
    \$
    for all $\epsilon\leq\epsilon^*$. 
    This proves the lemma. 
\end{proof}

\subsection{Lipschitz property}\label{sec:C32}


\begin{lemma}\label{lma:c2}
    Suppose Condition \ref{cond:3.1} holds. Let $f(x;\alpha,\phi)$ be the density given by \eqref{equ:f}. If $e^{-\phi(r)}$ is a Lipschitz continuous function with a Lipschitz constant $L$, then we have
    \$
    d_{\infty}(f(x;\alpha_1,\phi),f(x;\alpha_2,\phi))\leq \frac{L}{Z(\phi)}\cdot d_g(\alpha_1,\alpha_2),
    \$
    where $Z(\phi)$ is the normalizing constant in \eqref{equ:f}. 
\end{lemma}

\begin{proof}
    Since $e^{-\phi(r)}$ is a Lipschitz continuous function with a Lipschitz constant $L$, we have
    \$
    |f(x;\alpha_1,\phi)-f(x;\alpha_2,\phi)|\leq \frac{L}{Z(\phi)}\cdot|d_g(x,\alpha_1)-d_g(x,\alpha_2)|,
    \$
    where $Z(\phi)$ is the normalizing constant in \eqref{equ:f}. By the triangular inequality, we have 
    \$
    |d_g(x,\alpha_1)-d_g(x,\alpha_2)|\leq d_g(\alpha_1,\alpha_2),
    \$
    which implies that 
    \$
    |f(x;\alpha_1,\phi)-f(x;\alpha_2,\phi)|\leq \frac{L}{Z(\phi)}\cdot d_g(\alpha_1,\alpha_2).
    \$
    By taking the supremum over $x$, we conclude the proof of this lemma.
\end{proof}

\subsection{Identifiability }\label{sec:C3}


\begin{lemma}\label{lma:c3}
    Assume $(\cM,g^{\cM})$  is a Riemannian homogeneous space with $r_\cM=\sup_{x,y}d_g(x,y).$ Suppose  $\phi$ satisfies Condition \ref{cond:4.1} and is strictly increasing on $[0,r_\cM]$. Let $f(x;\alpha,\phi)$ be the density in \eqref{equ:f} and $\cX\subseteq\cM$ a compact set. Then for any $\alpha_0,\{\alpha_n\}_{n=1}^{\infty}\subseteq\cX$ such that
    \#\label{equ:C4}
    d_1(f(x;\alpha_n,\phi),f(x;\alpha_0,\phi))\to 0,\quad \textnormal{as }n\to\infty,
    \#
    where $d_1$ is the $L^1$ distance,
    we have $d_g(\alpha_n,\alpha_0)\to0$ as $n\to\infty$. 
\end{lemma}

\begin{proof}
    We prove this lemma by contradiction. Suppose condition \eqref{equ:C4} holds but $\alpha_n$ do not converge to $\alpha_0$. Then there is a subsequence $\{\alpha_{n_i}\}_{i=1}^\infty$ of $\{\alpha_n\}$ such that $\alpha_{n_i}\to\alpha'_0$ for some $\alpha_0'\neq \alpha_0$, where we use the fact that $\{\alpha_n\}\subseteq\cX$ and $\cX$ is a compact set. By Lemma \ref{lma:c5}, we have
    \#\label{equ:C5}
    d_1(f(x;\alpha_{n_i},\phi),f(x;\alpha_0',\phi))\to 0,\quad \textnormal{as }i\to\infty.
    \#
    Combining with \eqref{equ:C4}, we have $d_1(f(x;\alpha_0,\phi),f(x;\alpha_0',\phi))=0$. It then follows from Lemma \ref{lma:c4} that $\alpha_0'=\alpha_0$, which leads to contradiction.
\end{proof}

Lemma \ref{lma:c5} establishes the continuity of $f(x;\alpha,\phi)$ with respect to its parameter $\alpha$ under certain conditions, where we use the $L^1$ distance to measure the distance between functions. 


\begin{lemma}\label{lma:c5}
    Assume $(\cM,g^{\cM})$  is a Riemannian homogeneous space and Condition \ref{cond:4.1} holds. Let $f(x;\alpha,\phi)$ be the density in \eqref{equ:f}. If $
    d_g(\alpha_n,\alpha)\to 0$ as $n\to\infty$,
    then 
    \$
    d_1(f(x;\alpha_n,\phi),f(x;\alpha,\phi))\to0,
    \$
    where $d_1$ is the $L^1$ distance.
\end{lemma}

\begin{proof}
    By definition of the density function $f(x;\alpha,\phi)$ and the $L^1$ distance $d_1$, we know
    \#
    d_1(f(x;\alpha_n,\phi),f(x;\alpha,\phi))=\frac{1}{Z(\phi)}\int_{\cM}\left|e^{-\phi(d_g(x,\alpha_n))}-e^{-\phi(d_g(x,\alpha))}\right|\dvol(x), \label{equ:dh2}
    \#
    where $Z(\phi)$ is the normalizing constant.
    Since $\alpha_n\to\alpha$, $\alpha_n$ is bounded within the geodesic ball $\cB_{\cM}(\alpha,D)=\{x\in\cM\mid d_g(x,\alpha)< D\}$ for a sufficiently large $D$. For $d_g(x,\alpha)>2D$, we have
    \$
    d_g(x,\alpha_n)\geq d_g(x,\alpha)-d_g(\alpha_n,\alpha)\geq d_g(x,\alpha)/2.
    \$
    Consequently, for all $n$, we have
    \$
    e^{-\phi(d_g(x,\alpha_n))}\leq H(x)\coloneqq\left\{
    \begin{array}{ll}
        e^{-\phi(d_g(x,\alpha)/2)}, &  \textnormal{if }d_g(x,\alpha)>2D,\\
       1,  & \textnormal{otherwise}.
    \end{array}
    \right.
    \$
    where we use the fact that $\phi$ is an increasing function. Using $H$, we can upper bound the integrand in \eqref{equ:dh2} as follows:
    \$
    \left|e^{-\phi(d_g(x,\alpha_n))}-e^{-\phi(d_g(x,\alpha))}\right|\leq H(x)+e^{-\phi(d_g(x,\alpha))}.
    \$
    When $\cM$ is compact, $H(x)+e^{-\phi(d_g(x,\alpha))}$ is clearly an integrable function as it is bounded. When $\cM$ is noncompact, by using the polar coordinate representation \eqref{equ:integral}, Theorem \ref{thm:2.3}, Condition \ref{cond:3.1}, and Condition \ref{cond:4.1}, we also can show that it is integrable. Specifically,
    \$
    &\int_{\cM}\left(H(x)+e^{-\phi(d_g(x,\alpha))}\right)\dvol(x)\\
    \leq\ & \vol(\cB_{\cM}(\alpha,2D))+\vol(\SSS^{m-1})\cdot\int_{2D}^{\infty}e^{-\phi(r/2)}\sn_{\kappa_{\min}}^{m-1}(r)dr+\int_{\cM}e^{-\phi(d_g(x,\alpha))}\dvol(x)\\
    <\ &\infty,    
    \$
    where $\kappa_{\min}$ is the lower bound on the sectional curvatures of $\cM$ used in Condition \ref{cond:4.1} and $\sn_{\kappa_{\min}}(\cdot)$ is given by \eqref{equ:sn}.
    These results imply that the integrand in \eqref{equ:dh2} has a dominating function $H(x)+e^{-\phi(d_g(x,\alpha))}$. 
    In addition, since $e^{-\phi(r)}$ is a Lipschitz continuous function with a Lipschitz constant $L$, we have 
    \$
    |e^{-\phi(d_g(x,\alpha_n))}-e^{-\phi(d_g(x,\alpha))}|\leq L|d_g(x,\alpha_n)-d_g(x,\alpha)|\leq L\cdot d_g(\alpha_n,\alpha),
    \$
    where the second inequality uses the triangle inequality. It implies that the integrand in \eqref{equ:dh2} converges to zero pointwise, and by the dominated convergence theorem, the integral in \eqref{equ:dh2} converges to zero, which concludes the proof. 
\end{proof}

Lemma \ref{lma:c4} shows that if the $L^1$ distance between $f(x;\alpha,\phi)$ and $f(x;\alpha',\phi)$ is zero, then $\alpha=\alpha'$ under mild conditions. 

\begin{lemma}\label{lma:c4}
    Assume $(\cM,g^{\cM})$  is a Riemannian homogeneous space with $r_{\cM}=\sup_{x,y}d_g(x,y)$. Suppose Condition \ref{cond:3.1}  holds and $\phi$ is a strictly increasing function over $[0,r_\cM]$. Let $f(x;\alpha,\phi)$ be the density in \eqref{equ:f}. Then the following condition
    \#\label{equ:dh=0}
    d_1(f(x;\alpha,\phi),f(x;\alpha',\phi))=0,
    \#
    where $d_1$ is the $L^1$ distance, implies that $\alpha=\alpha'$. 
\end{lemma}

\begin{proof}
    By definition of the density $f$ and the $L^1$ distance $d_1$, the condition \eqref{equ:dh=0} is reduced to
    \$
    \int_{\cM}\left|e^{-\phi(d_g(x,\alpha))}-e^{-\phi(d_g(x,\alpha'))}\right|\dvol(x)=0.
    \$
    As the integrand is nonnegative, the above condition implies that
    \$
    e^{-\phi(d_g(x,\alpha))}-e^{-\phi(d_g(x,\alpha'))}=0,\quad\forall x\in\cM,
    \$
    where we use the continuity of the function $\phi$. It follows that
    \$
    \phi(d_g(x,\alpha))=\phi(d_g(x,\alpha')),\quad\forall x\in\cM,
    \$
    and since $\phi$ is strictly increasing, we have
    \$
    d_g(x,\alpha)=d_g(x,\alpha'),\quad\forall x\in\cM.
    \$
    This shows that $\alpha=\alpha'$, which concludes the proof. 
\end{proof}

\subsection{Proof of the claim \eqref{equ:4.13}}\label{sec:c4}

Lemma \ref{lma:c6} proves the claim \eqref{equ:4.13} used in the proof of Theorem \ref{thm:4.3}. 

\begin{lemma}\label{lma:c6}
    Assume $(\cM,g^{\cM})$ is a Riemannian homogeneous space and $\phi$ satisfies Condition \ref{cond:4.5}. Let $f(x;\alpha,\phi)$ be the density in \eqref{equ:f}. Then we have
    \$
    \lim_{\epsilon\to 0}\inf\left\{{D(\alpha,\alpha_0)}/{d_g(\alpha,\alpha_0)}\mid d_g(\alpha,\alpha_0)<\epsilon\right\}>0,
    \$
    where $D(\alpha,\alpha_0)=d_1(f(x;\alpha,\phi),f(x;\alpha_0,\phi)).$
\end{lemma}

\begin{proof}
    Without loss of generality, we assume that $\alpha_0\in\cM$ is fixed and let $\alpha$ vary in the geodesic ball $\cB_{\cM}(\alpha_0,\epsilon)$. By using the definition of $f(x;\alpha,\phi)$ and the $L^1$ distance $d_1$, we can write $D(\alpha,\alpha_0)$ as follows:
    \$
    D(\alpha,\alpha_0)=\frac{1}{Z(\phi)}\int_{\cM}\left|e^{-\phi(d_g(x,\alpha))}-e^{-\phi(d_g(x,\alpha_0))}\right|\dvol(x).
    \$
    For any $\delta>0$, it holds that
    \$
    Z(\phi)\cdot D(\alpha,\alpha_0)\geq \int_{\cB_{\cM}(\alpha_0,\delta)}\left|e^{-\phi(d_g(x,\alpha))}-e^{-\phi(d_g(x,\alpha_0))}\right|\dvol(x),
    \$
    where $\cB_{\cM}(\alpha_0,\delta)$ is the geodesic ball $\{x\in\cM\mid d_g(x,\alpha_0)<\delta\}$. Let $\inj(\cM)$ denote the injectivity radius of $\cM$, which is positive since $\cM$ is a Riemannian homogeneous space. We can show that for sufficiently small $\delta$, 
    $\delta<\frac{\inj(\cM)}{2}$, $\phi(\delta)<\infty$, and
    \$
    |1-e^{w}|\geq \frac{1}{2}\cdot |w|,\quad\forall w\in[\phi(0)-\phi(2\delta),0].
    \$ 
    where we use the fact that $\phi$ is an increasing and continuous function. Choosing any one of such $\delta$, we can show that for all $\alpha,x\in\cB_{\cM}(\alpha_0,\delta)$, the following inequality holds:
    \$
    \left|e^{-\phi(d_g(x,\alpha))}-e^{-\phi(d_g(x,\alpha_0))}\right|&=e^{-\phi(d_g(x,\alpha_0))}\cdot |e^{-\phi(d_g(x,\alpha))+\phi(d_g(x,\alpha_0))}-1|\\
    &\geq e^{-\phi(\delta)}\cdot\frac{1}{2}\cdot |\phi(d_g(x,\alpha))-\phi(d_g(x,\alpha_0))|.
    \$
    In particular, for all $\alpha\in\cB_{\cM}(\alpha_0,\delta)$, we have
    \$
    2e^{\phi(\delta)}\cdot Z(\phi)\cdot D(\alpha,\alpha_0)\geq \int_{\cB_{\cM}(\alpha_0,\delta)}|\phi(d_g(x,\alpha))-\phi(d_g(x,\alpha_0))|\dvol(x).
    \$
    Consequently, to prove the lemma, it suffices to show that
    \#\label{equ:c.8}
    \lim_{\epsilon\to0}\inf\{T(\alpha,\alpha_0)\mid d_g(\alpha,\alpha_0)< \epsilon\}>0,
    \#
    where
    \$
    T(\alpha,\alpha_0)\coloneqq\int_{\cB_{\cM}(\alpha_0,\delta)}\frac{|\phi(d_g(x,\alpha))-\phi(d_g(x,\alpha_0))|}{d_g(\alpha,\alpha_0)}\dvol(x).
    \$
    We will prove this condition by contradiction. Suppose \eqref{equ:c.8} does not hold, then there is a sequence  $\{y_s\}_{s=1}^{\infty}$ converging to $\alpha_0$ such that $T(y_s,\alpha_0)$ tends to $0$. Without loss of generality, we can assume that $d_g(y_s,\alpha_0)<\epsilon_0$ for some $\epsilon_0\leq\delta$ and all $s$. In addition, let $\theta_0=\frac{\pi}{100}$, and we assume there is a unit tangent vector $V\in T_{\alpha_0}\cM$ such that the angle between $V$ and $\Log_{\alpha_0}y_s$ is smaller than $\theta_0$ for an infinite number of $s$, where $\Log_{\alpha_0}$ is the logarithmic map at $\alpha_0$. Without loss of generality, we can assume that the angle between such $V$ and $\Log_{\alpha}y_s$ is smaller than $\theta_0$ for all $s$\footnote{Otherwise, we may take a subsequence such that this condition holds.}. The rest of the proof is to provide a positive lower bound for $T({y_s,\alpha_0})$ when $y_s$ is sufficiently close to $\alpha_0$, which immediately leads to contradiction. 
    
    Since $\phi$ is strictly increasing over $[0,r_{\cM}]$ and continuously differentiable over $(0,r_{\cM})$, we can find an interval $[r_1,r_2]$ such that $0<r_1<r_2<\delta$ and $\phi'(r)\geq c_0$ for all $r\in[r_1,r_2]$ and some positive constant $c_0$. Let $0<\Delta<{(r_2-r_1)}/{2}$ and consider the following region
    \#
    \cR=\{x\in\cM\mid \ &d_g(x,\alpha_0)\in[r_1+\Delta,r_2-\Delta], \notag\\
    &\textnormal{the angle between }\Log_{\alpha_0}x\textnormal{ and }V\textnormal{ is larger than }\pi-\theta_0.\},\label{equ:cR}
    \#
    where $V\in T_{\alpha_0}\cM$ is the aforementioned unit tangent vector. 
    It is obvious that $\cR\subseteq\cB_{\cM}(\alpha_0,\delta)$ and $\cR$ is a region of positive volume. In addition, for any $x\in\cR$ and any $s$, the angle between $\Log_{\alpha_0}x$ and $\Log_{\alpha_0}y_s$ is larger than $\pi-2\theta_0$ by the triangle inequality. 

    \begin{figure}
        \centering
        \includegraphics[height=0.18\textheight]{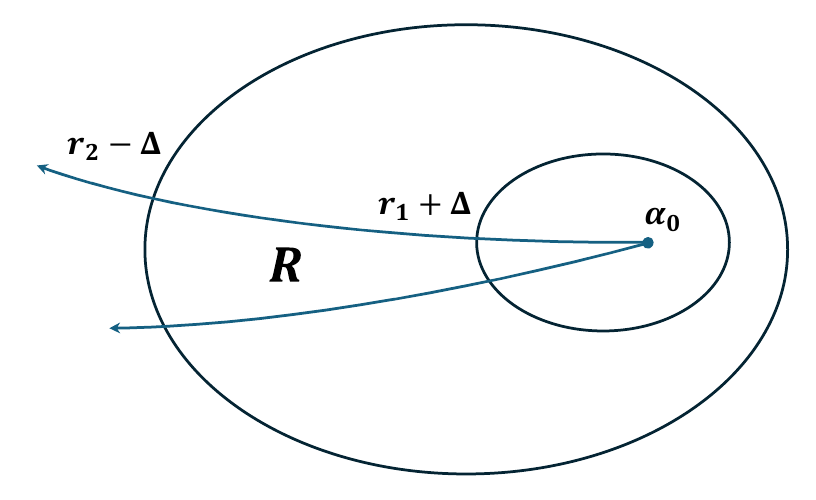}
        \caption{Visualization of the region $\cR$ defined in \eqref{equ:cR}.}
        \label{fig:cR}
    \end{figure}
    
    Now we assume without loss of generality that $\epsilon_0\leq\Delta$. It implies that for any $x\in\cR$ and any $y_s$, the condition $d_g(x,y_s)\in[r_1,r_2]$ holds, due to the triangle inequality. Therefore, by the mean value theorem, for any $x\in\cR$ and any $y_s$, we have
    \$
    \phi(d_g(x,y_s))-\phi(d_g(x,\alpha_0)))=\phi'(\xi)\cdot (d_g(x,y_s)-d_g(x,\alpha_0)),
    \$
    where $\xi$ is some number between $d_g(x,y_s)$ and $d_g(x,\alpha_0)$ and thus $\xi\in[r_1,r_2]$. Since $\phi'(r)\geq c_0>0$ for all $r\in[r_1,r_2]$, it follows that 
    \$
    |\phi(d_g(x,y_s))-\phi(d_g(x,\alpha_0)))|\geq c_0\cdot |d_g(x,y_s)-d_g(x,\alpha_0)|,
    \$
    and thus
    \#\label{equ:c10}
    T(y_s,\alpha_0)\geq c_0\cdot \int_{\cR}\frac{|d_g(x,y_s)-d_g(x,\alpha_0)|}{d_g(y_s,\alpha_0)}\dvol(x). 
    \#
    The rest of the proof aims to lower bound the right hand side of \eqref{equ:4.10} by examining the geodesic triangle with vertices $x,$ $\alpha_0,$ and $y_s$. 
    
    Our idea is to employ Lemma \ref{lma:c8}. Since $\cM$ is a Riemannian homogeneous space, its sectional curvatures are upper bounded by some $\kappa>0$. Then we can assume without loss of generality that the constant $\delta$ in the previous arguments is smaller than the constant $r_0$ specified in Lemma \ref{lma:c8}. Then by Lemma \ref{lma:c8}, for any $x\in\cR$ and $y_s$ specified above, we have  
    \#\label{equ:c11}
    d_g(x,y_s)\geq d_{\kappa}(a,b),
    \#
    where $a=d_g(x,\alpha_0),b=d_g(y_s,\alpha_0),$ and
    \$
    d_{\kappa}(a,b)=\kappa^{-1/2}\cdot \arccos\left(\cos(\sqrt{\kappa}a)\cos(\sqrt{\kappa}b)-\sin(\sqrt{\kappa}a)\sin(\sqrt{\kappa}b)\cos(2\theta_0)\right).
    \$
    Here we use the fact that the angle between $\Log_{\alpha_0}x$ and $\Log_{\alpha_0}y_s$ is larger than $\pi-2\theta_0$. 
    
    Notice that for all $x\in\cR$, we have $a=d_g(x,\alpha_0)\geq r_1>0$. Then there is a sufficiently small $b_0>0$, which is independent of $x$, such that the function $\cos(\sqrt{\kappa}d_{\kappa}(a,b))<1$ for all $b\in[0,b_0]$. Consequently, the function $d_{\kappa}(a,b)$ is differentiable with respect to $b$ in $[0,b_0]$, and for $b\in[0,b_0]$, we have 
    \$
    \frac{\partial}{\partial {b}}d_{\kappa}(a,b)=\frac{1}{\sqrt{1-\cos^2(\sqrt{\kappa}d_{\kappa}(a,b))}}\cdot \left(\cos(\sqrt{\kappa}a)\sin(\sqrt{\kappa}b)+\sin(\sqrt{\kappa}a)\cos(\sqrt{\kappa}b)\cos(2\theta_0)\right).
    \$
    By selecting a sufficiently small $b_0>0$, we can ensure that $\frac{\partial}{\partial b}d_{\kappa}(a,b)\geq C_0>0$ for all $b\in[0,b_0]$ and $a=d_g(x,\alpha_0)$ with $x\in\cR$. Here $C_0$ is a constant independent of $b$ and $x$. Therefore, by using  mean value theorem, we can show that for all $a=d_g(x,\alpha_0)$ with $x\in\cR$, 
    \$
    d_\kappa(a,b)-a\geq C_0b,\quad\forall b\in[0,b_0],
    \$
    where we use the fact that $d_{\kappa}(a,0)=a$. Combining this with \eqref{equ:c11}, we can conclude that for all $x\in\cR$, 
    \$
    d_g(x,y_s)-d_g(x,\alpha_0)\geq C_0\cdot d_g(y_s,\alpha_0),\quad\textnormal{if } d_g(y_s,\alpha_0)\leq b_0.
    \$
    Combining this with \eqref{equ:c10}, we have that
    \$
    T(y_s,\alpha_0)\geq c_0\cdot C_0\cdot \vol(\cR)>0,\quad \textnormal{if } d_g(y_s,\alpha_0)\leq b_0.
    \$
    This contradicts the assumption that $T(y_s,\alpha_0)$ tends to zero, as $s$ tends to infinity.
\end{proof}




    

\subsection{Geometric lemma}\label{sec:c5}

\begin{figure}
    \centering
    \includegraphics[width = 0.95\textwidth]{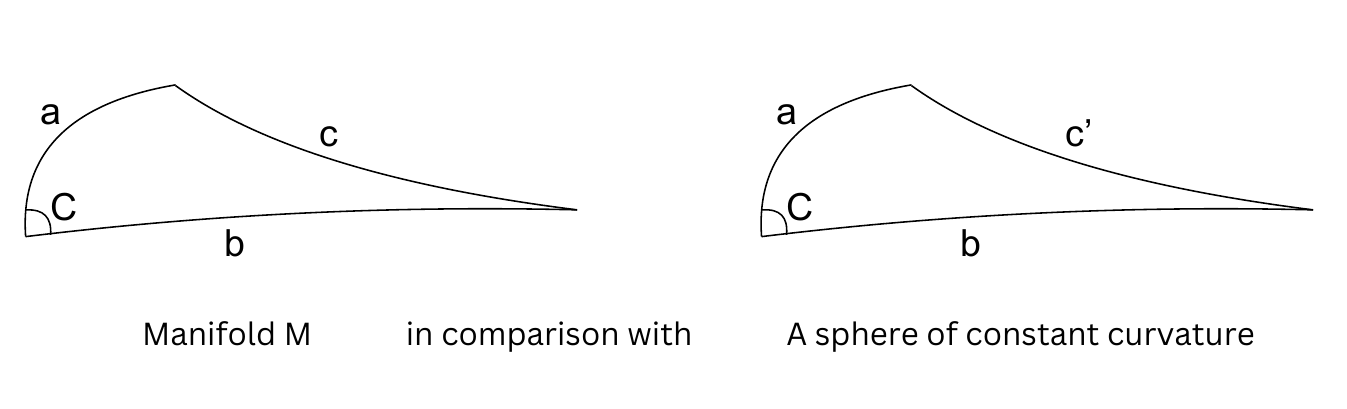}
    \caption{Illustration of Lemma \ref{lma:c8}. Left: a geodesic triangle in $\cM$ with edges of lengths $a,b,$ $c$, and the angle opposite side $c$ is $C$. Right: a geodesic triangle in a sphere of curvature $\kappa$ with edges of lengths $a,b,c'$, and the angle opposite $c'$ is still $C$. Under the conditions of Lemma \ref{lma:c8}, $c\geq c'$.}
    \label{fig:1}
\end{figure}
Lemma \ref{lma:c8} delves into the geodesic triangle on a Riemannian manifold with an upper  curvature bound. It is built on the Rauch comparison theorem, Section 11, Chapter 1 in \cite{cheeger1975comparison}, as well as the spherical law of cosines.  

\begin{lemma}\label{lma:c8}
    Consider  a complete Riemannian manifold $(\cM,g^{\cM})$ with a positive injectivity radius $\inj(\cM)$. Assume its sectional curvatures are upper bounded by $\kappa>0$. Consider a geodesic triangle in $\cM$ with edges of lengths $a,b$, and $c$, and suppose the angle opposite side $c$ is $C$. If $a,b\leq r_0$ for a sufficiently small constant $r_0$, then we have $c\geq c'$, where 
    \$
    c'=\kappa^{-1/2}\cdot\arccos\left( \cos(\sqrt{\kappa}a)\cos(\sqrt{\kappa}b)+\sin(\sqrt{\kappa}a)\sin(\sqrt{\kappa}b)\cos(C)\right).
    \$
\end{lemma}

\begin{proof}
    Let $m$ be the dimension of $\cM$. 
    Since the curvatures of $\cM$ are upper bounded by $\kappa$ and $\cM$ has a positive injectivity radius, by Corollary 1.35 in \cite{cheeger1975comparison}, we can show that when $a,b$ are sufficiently small, the length $c$ is larger than $c'$, where $c'$ is given by the following rule:
    \begin{center}
        $a,b,c'$ are the  edge lengths of a geodesic triangle in an $m$-dimensional sphere with sectional curvature $\kappa$, and the angle opposite side $c'$ is still equal to $C$.
    \end{center} 
    By the spherical law of cosines, we have
    \$
    \cos(\sqrt{\kappa}c')=\cos(\sqrt{\kappa}a)\cos(\sqrt{\kappa}b)+\sin(\sqrt{\kappa}a)\sin(\sqrt{\kappa}b)\cos(C),
    \$
    which proves the lemma.
\end{proof}

\section{Technical lemmas for Section \ref{sec:5}}

This section collects technical results used in Section \ref{sec:5}. Lemmas related to Section \ref{sec:5.1} will be given in Section \ref{sec:d1}, and lemmas related to Section \ref{sec:5.2} will be provided in Section \ref{sec:d2}.

\subsection{Simply connected noncompact Riemannian symmetric spaces}\label{sec:d1}

In this section, we shall focus on the case of a simply connected noncompact Riemannian symmetric space. We will provide several geometrical lemmas in Section \ref{sec:d1.1}, and conduct statistical analysis in Section \ref{sec:d1.2}.

\subsubsection{Geometric lemmas}\label{sec:d1.1}

In this section, we present several geometric lemmas on a simply connected noncompact Riemannian symmetric space $(\cM,g^{\cM})$. In Lemma \ref{lma:d11}, we consider a geodesic $\alpha(t)$ in $\cM$ with a constant compact component, and demonstrate the smoothness of the function $\varrho(t)=d_g^2(x,\alpha(t))$ 
 at $t\in\RR$ for any fixed $x\in\cM$. 

\begin{lemma}\label{lma:d11}
    Let $(\cM,g^{\cM})$ be a simply connected noncompact Riemannian symmetric space with the decomposition $\cM=\cM_H\times\cM_C$, where $\cM_H$ and $\cM_C$ are the Hadamard and compact components of $\cM$, respectively. Let 
    $\alpha(t)=(\alpha^H(t),\alpha^C)$ be a geodesic in $\cM$ with a constant compact component. Then given any $x\in\cM$, the function $\varrho(t)=d_g^2(x,\alpha(t))$ is smooth at any $t\in\RR$. 
\end{lemma}

\begin{proof}
    We denote each point $x\in\cM$ by $(x^H,x^C)$, where $x^H$ and $x^C$ are the Hadamard and compact components of $x$, respectively. For all pairs $x,y\in\cM$, we have
    \#\label{equ:d1.1}
    d_g^2(x,y)=d_H^2(x^H,y^H)+d_C^2(x^C,y^C),
    \#
    where $d_H$ and $d_C$ denote the geodesic distances on $\cM_H$ and $\cM_C$, respectively. Consider a  geodesic $\alpha(t)=(\alpha^H(t),\alpha^C)$ in $\cM$ with a constant compact component, where $\alpha^H(t)$ is a geodesic in $\cM_H$ and $\alpha^C$ is a constant point in $\cM_C$. Then for any $x=(x^H,x^C)\in\cM$, we have
    \$
    \varrho(t)&=d_g^2(x,\alpha(t))=d_H^2(x^H,\alpha^H(t))+d_C^2(x^C,\alpha^C),
    \$
    where we use the equality \eqref{equ:d1.1}.
    Since $\cM^H$ is a Hadamard manifold, the squared distance function $ d_H^2(x^H,y^H)$ is smooth at all  $y^H\in\cM_H$. Thus, as a consequence of composition,  the function $\varrho(t)$ is smooth at any $t\in\RR$. 
\end{proof}

In Lemma \ref{lma:d12}, we consider a unit-speed geodesic $\alpha(t)$ in $\cM$ with a constant compact component. We will investigate the upper bounds on the absolute values of the first two order derivatives of the function $\varrho(t)=d_g^2(x,\alpha(t))$ for any fixed $x\in\cM$.

\begin{lemma}\label{lma:d12}
    Suppose $\cM=\cM_H\times\cM_C$ is a simply connected noncompact Riemannian symmetric space, where $\cM_H$ and $\cM_C$ are the Hadamard and compact components of $\cM$, respectively. Let $\alpha(t)=(\alpha^H(t),\alpha^C)$ be a unit-speed geodesic in $\cM$ with a constant compact component. Consider the function $\varrho(t)=d_g^2(x,\alpha(t))$ for some $x\in\cM$. Then the following upper bounds hold:
    \#
    \left|\varrho'(t)\right|&\leq 2\cdot d_g(x,\alpha(t)),\label{equ:d1.3} \\
    \left|\varrho''(t)\right|&\leq 4\cdot d_g(x,\alpha(t))\cdot \frac{\sn_{\kappa_{\min}}'(d_g(x,\alpha(t)))}{\sn_{\kappa_{\min}}(d_g(x,\alpha(t)))}.\label{equ:d1.4}
    \#
    where $\kappa_{\min}<0$ is a lower bound on the sectional curvatures of $\cM$ and $\sn_{\kappa_{\min}}(r)$ is given by \eqref{equ:sn}.
\end{lemma}

\begin{proof}
    By Lemma \ref{lma:d11}, we know  $\varrho(t)$ is smooth at any $t\in\RR$. To derive the inequalities \eqref{equ:d1.3} and \eqref{equ:d1.4}, we will consider the following two cases separately.
    \begin{itemize}
        \item {\bf Case 1. $x=(\alpha^H(t_0),x^C)$ for some $t_0\in\RR$ and $x^C\in\cM_C$.} In this case, the function $\varrho(t)$ can be expressed as 
    \$
    \varrho(t)=d_H^2(\alpha^H(t_0),\alpha^H(t))+d_C^2(x^C,\alpha^C),
    \$
    where $d_H$ and $d_C$ denote the geodesic distances in $\cM_H$ and $\cM_C$, respectively. Since $\cM_H$ is a Hadamard manifold, there exists a unique geodesic connecting $\alpha^H(t_0)$ and $\alpha^H(t)$. Moreover, since $\alpha$ is a unit-speed geodesic in $\cM$ with a constant compact component, $\alpha^H(t)$ is a unit-speed geodesic in $\cM_H$ and $d_H(\alpha^H(t_0),\alpha^H(t))=|t-t_0|$. It implies that
    \$
    \varrho(t)=(t-t_0)^2+d_C^2(x^C,\alpha^C).
    \$
     By taking the derivatives, we can show that $|\varrho'(t)|=2|t-t_0|\leq 2\cdot d_g(x,\alpha(t))$, and $\varrho''(t)=2$ for any $t\in\RR$. This implies \eqref{equ:d1.4} since
     \$
     |\varrho''(t)|=2\leq 2\cdot d_g(x,\alpha(t))\cdot \frac{\sn_{\kappa_{\min}}'(d_g(x,\alpha(t)))}{\sn_{\kappa_{\min}}(d_g(x,\alpha(t)))},
     \$
     where $\kappa_{\min}<0$ is a lower bound on the sectional curvatures of $\cM$ and $\sn_{\kappa_{\min}}(\cdot)$ is given by \eqref{equ:sn}. 

        \item  {\bf Case 2. $x=(x^H,x^C)\in\cM_H\times\cM_C$ where $x^H\notin\Img(\alpha^H)\coloneqq\{\alpha^H(t)\mid t\in\RR\}$.} In this case, the function $\varrho(t)$ can be expressed as 
        \$
        \varrho(t)=d_H^2(x^H,\alpha^H(t))+d_C^2(x^C,\alpha^C),
        \$
        where $d_H$ and $d_C$ denote the geodesic distances in $\cM_H$ and $\cM_C$, respectively. Since $\cM_H$ is a Hadamard manifold and $x_H\notin\Img(\alpha^H)$, the function $d_H(x^H,\alpha^H(t))$ is positive and smooth at any $t\in\RR$. Additionally, by calculation, the first two order derivatives of $\varrho(t)$ are given by
        \$
        \varrho'(t)&=2\cdot d_H(x^H,\alpha^H(t))\cdot\frac{d}{dt}d_H(x^H,\alpha^H(t)),\\
        \varrho''(t)&=2\cdot\left|\frac{d}{dt} d_H(x^H,\alpha^H(t))\right|^2+2\cdot d_H(x^H,\alpha^H(t))\cdot\frac{d^2}{dt^2}d_H(x^H,\alpha^H(t)).
        \$
        Taking the absolute values of $\varrho'(t)$ and $\varrho''(t)$ and using the triangle inequality, we obtain that
        \#
        |\varrho'(t)|&\leq2\cdot d_H(x^H,\alpha^H(t))\cdot \left|\frac{d}{dt} d_H(x^H,\alpha^H(t))\right|,\label{equ:varrho-prime-1}\\ 
        \label{equ:varrho-1}
        |\varrho''(t)|&\leq 2\cdot\left|\frac{d}{dt} d_H(x^H,\alpha^H(t))\right|^2+2\cdot  d_H(x^H,\alpha^H(t)) \cdot\left|\frac{d^2}{dt^2}d_H(x^H,\alpha^H(t))\right|.
        \#
        Since $\alpha(t)$ is a unit-speed geodesic in $\cM$ with a constant compact component, its Hadamard component $\alpha^H(t)$ is also a unit-speed geodesic in $\cM_H$. It implies that
        \$
        \left|\frac{d}{dt} d_H(x^H,\alpha^H(t))\right|\leq 1.
        \$
        Substituting this into \eqref{equ:varrho-prime-1} and \eqref{equ:varrho-1}, we obtain
        \#
        |\varrho'(t)|&\leq 2\cdot d_H(x^H,\alpha^H(t))\overset{({\rm i})}{\leq} 2d_g(x,\alpha(t)),\notag\\ 
        |\varrho''(t)|&\leq 2+2\cdot d_H(x^H,\alpha^H(t))  \cdot \left|\frac{d^2}{dt^2}d_H(x^H,\alpha^H(t))\right|.\label{equ:varrho-2}
        \#
        where (i) uses the inequality $d_H(x^H,\alpha^H(t))\leq d_g(x,\alpha(t))$. 
        To further upper bound the right hand side of \eqref{equ:varrho-2}, we  calculate 
    \$
    \frac{d^2}{dt^2}d_H(x^H,\alpha^H(t))=\Hess_{\alpha^H(t)} \rho(\dot{\alpha}^{H}(t),\dot{\alpha}^{H}(t)),
    \$
    where 
    \$
    \Hess_{\alpha^H}\rho:T_{\alpha^H}\cM_H\times T_{\alpha^H}\cM_H\to\RR
    \$ 
    is the Hessian of the function $\rho(\alpha^H)=d_H(x^H,\alpha^H)$, and $\dot{\alpha}^H(t)\in T_{\alpha^H(t)}\cM_H$ is the differential of $\alpha^H(t)$. 
    By employing the Hessian comparison theorem, Theorem 27, Chapter 6 in \cite{petersen2006riemannian}, we have 
    \#
    &\left|\Hess_{\alpha^H(t)} \rho(\dot{\alpha}^{H}(t),\dot{\alpha}^{H}(t))\right|\notag\\
    \leq\ & \max\left\{\frac{1}{d_H(x^H,\alpha^H(t))},\frac{\sn_{\kappa_{\min}}'(d_H(x^H,\alpha^H(t)))}{\sn_{\kappa_{\min}}(d_H(x^H,\alpha^H(t)))}\right\}\cdot |\dot{\alpha}^H(t)|^2\notag\\
    {\leq} \ &\frac{\sn_{\kappa_{\min}}'(d_H(x^H,\alpha^H(t)))}{\sn_{\kappa_{\min}}(d_H(x^H,\alpha^H(t)))},\label{equ:d115}
    \#
    where $\kappa_{\min}$ is a lower bound on the sectional curvatures of $\cM$, $\sn_{\kappa_{\min}}(\cdot)$ is defined by \eqref{equ:sn}, and the second inequality follows from $|\dot{\alpha}^H(t)|=1$ and 
    \#\label{equ:d118}
    \frac{1}{r}\leq \frac{\sn_{\kappa_{\min}}'(r)}{\sn_{\kappa_{\min}}(r)}.
    \# 
    Consequently, by substituting \eqref{equ:d115} into \eqref{equ:varrho-2}, we obtain
    \$
    |\varrho''(t)|&\leq 2+2\cdot d_H(x^H,\alpha^H(t)) \cdot \frac{\sn_{\kappa_{\min}}'(d_H(x^H,\alpha^H(t)))}{\sn_{\kappa_{\min}}(d_H(x^H,\alpha^H(t)))}\\
    &\leq 4\cdot d_H(x^H,\alpha^H(t)) \cdot \frac{\sn_{\kappa_{\min}}'(d_H(x^H,\alpha^H(t)))}{\sn_{\kappa_{\min}}(d_H(x^H,\alpha^H(t)))}\\
    &\leq 4\cdot d_g(x,\alpha(t))\cdot \frac{\sn_{\kappa_{\min}}'(d_g(x,\alpha(t)))}{\sn_{\kappa_{\min}}(d_g(x,\alpha(t)))},
    \$
    where the second inequality uses \eqref{equ:d118} and the third inequality uses the inequality 
    \$
    d_g(x,\alpha(t))\geq d_H(x^H,\alpha^H(t))
    \$
    and the fact that
    \$
    r\cdot \sn_{\kappa_{\min}}'(r)/\sn_{\kappa_{\min}}(r)
    \$
    is an increasing function. 
    \end{itemize}
    The proof is complete by combining the above two cases. 
\end{proof}

In Lemma \ref{lma:d-varrho_phi}, we will consider the case when $\dim(\cM)>1$. Let $\alpha(t)$ be a unit-speed geodesic in $\cM$ with a constant compact component. We will examine the upper bound on the absolute values of the first two order derivatives of the function $\varrho_0(t)=d_g(x,\alpha(t))$ for any fixed $x\in\cM-\Img(\alpha)$, where
\#\label{equ:img-alpha}
\Img(\alpha)=\{\alpha(t)\mid t\in\RR\}
\#
is of measure zero in $\cM$.

\begin{lemma}\label{lma:d-varrho_phi}
    Suppose $\cM$ is a simply connected noncompact Riemannian symmetric space with $\dim(\cM)>1$. Let $\alpha(t)$ be a unit-speed geodesic in $\cM$ with a constant compact component. Let $\varrho_0(t)=d_g(x,\alpha(t))$ be the function  for some $x\in\cM-\Img(\alpha)$, where $\Img(\alpha)$ is given by \eqref{equ:img-alpha}. Then $\varrho_0(t)$ is smooth at any $t\in\RR$. In addition, the following upper bounds hold:
    \#
    |\varrho_0'(t)|&\leq 1,\\
    |\varrho_0''(t)|&\leq 3\cdot \frac{\sn_{\kappa_{\min}}'(d_g(x,\alpha(t)))}{\sn_{\kappa_{\min}}(d_g(x,\alpha(t)))},
    \#
    where $\kappa_{\min}<0$ is a lower bound on the sectional curvatures of $\cM$ and $\sn_{\kappa_{\min}}(\cdot)$ is given by \eqref{equ:sn}. 
\end{lemma}

\begin{proof}
    Suppose $q(r)=\sqrt{r}$ is the square root function. Then the relationship  $\varrho_0(t)=q\circ \varrho(t)$ holds, where $\varrho(t)=d_g^2(x,\alpha(t))$. Since $\alpha(t)$ is a geodesic in $\cM$ with a constant compact component, it follows from Lemma \ref{lma:d11} that $\varrho(t)$ is smooth over $t\in\RR$. Additionally, since $x\in\cM-\Img(\alpha)$, where $\Img(\alpha)$ is given by \eqref{equ:img-alpha}, the inequality $\varrho(t)>0$ holds for any $t\in\RR$. As a consequence of composition, the function $\varrho_0(t)$ is smooth at any $t\in\RR$, since the square root function $q(\cdot)$ is smooth over $(0,\infty)$. Furthermore, by using the chain rule, we can compute its first two order derivatives as follows
    \$
    \varrho_0'(t)&=\frac{\varrho'(t)}{2 \sqrt{\varrho(t)}},\\
    \varrho_0''(t)&=\frac{\varrho''(t)}{2\sqrt{\varrho(t)}}-\frac{(\varrho'(t))^2}{4\varrho(t)\sqrt{\varrho(t)}}.
    \$
    By Lemma \ref{lma:d12}, we can establish the following upper bounds on the absolute values of the first two order derivatives of $\varrho_0(t)$:
    \$
    |\varrho_0'(t)|&\leq 1.\\
    |\varrho_0''(t)|&\leq 2\cdot \frac{\sn_{\kappa_{\min}}'(d_g(x,\alpha(t)))}{\sn_{\kappa_{\min}}(d_g(x,\alpha(t)))}+\frac{1}{d_g(x,\alpha(t))}\\
    &\overset{({\rm i})}{\leq} 3\cdot \frac{\sn_{\kappa_{\min}}'(d_g(x,\alpha(t)))}{\sn_{\kappa_{\min}}(d_g(x,\alpha(t)))}, 
    \$
    where $\kappa_{\min}<0$ is a lower bound on the sectional curvatures of $\cM$, $\sn_{\kappa_{\min}}(\cdot)$ is given by \eqref{equ:sn}, and the inequality (i) uses the inequality $1\leq r\cdot \sn_{\kappa_{\min}}'(r)/\sn_{\kappa_{\min}}(r)$ for any $r>0$. This completes the proof of the lemma. 
\end{proof}

\subsubsection{KL divergence}\label{sec:d1.2}

This section will present lemmas related to the KL divergence $D_{\KL}(P_{\alpha}\|P_{\alpha'})$, where $P_{\alpha}$ denotes the distribution $f(x;\alpha,\phi)$. In Lemma \ref{lma:d1}, we will examine the finiteness of this KL divergence under certain conditions. 

\begin{lemma}\label{lma:d1}
    Let $(\cM,g^{\cM})$ be a simply connected noncompact Riemannian symmetric space. Let $\phi$ be a function satisfying Condition \ref{cond:3.1} and  property (2) in Condition \ref{cond:5.2}. Let
    $\alpha(t)$ be a unit-speed geodesic in $\cM$. Then the KL divergence $D_{\KL}(P_{\alpha(t)}\|P_{(\alpha(s))})$ is finite for all $t,s\in\RR$,   where $P_{\alpha}$ denotes the distribution $f(x;\alpha,\phi)$. 
\end{lemma}

\begin{proof}
    By definition of the KL divergence and $P_{\alpha}$, we have
    \$
    D_{\KL}(P_{\alpha(t)}\|P_{\alpha(s)})=\frac{1}{Z(\phi)}\cdot\int_{\cM}(\phi(d_g(x,\alpha(s)))-\phi(d_g(x,\alpha(t))))e^{-\phi(d_g(x,\alpha(t)))}\dvol(x).
    \$
    Since $\phi$ is an increasing function, we have
    \$
    \int_{\cM}\phi(d_g(x,\alpha(t)))e^{-\phi(d_g(x,\alpha(t)))}\dvol(x)\leq \int_{\cM}\phi(d_g(x,\alpha(t))+|t-s|)e^{-\phi(d_g(x,\alpha(t)))}\dvol(x),\\
    \int_{\cM}\phi(d_g(x,\alpha(s)))e^{-\phi(d_g(x,\alpha(t)))}\dvol(x)\overset{(i)}{\leq} \int_{\cM}\phi(d_g(x,\alpha(t))+|t-s|)e^{-\phi(d_g(x,\alpha(t)))}\dvol(x),
    \$
    where (i) uses the triangle inequality $d_g(x,\alpha(s))\leq d_g(x,\alpha(t))+d_g(\alpha(s),\alpha(t))\leq d_g(x,\alpha(t))+|t-s|$. 
    Using the polar coordinate expression \eqref{equ:integral} of the integral and Theorem \ref{thm:2.3}, we can show that
    \#
    &\int_{\cM}\phi\left(d_g(x,\alpha(t))+|t-s|\right)\cdot e^{-\phi(d_g(x,\alpha(t)))}\dvol(x)\notag\\
    \leq & \left(\int_{0}^\infty \phi(r+|t-s|)e^{-\phi(r)}\sn_{\kappa_{\min}}^{m-1}(r)dr\right)\cdot \vol(\SSS^{m-1}),\notag\\
    \leq &\left( \phi(2|t-s|) \cdot\sn_{\kappa_{\min}}^{m-1}(|t-s|)\cdot{|t-s|}+\int_{|t-s|}^\infty \phi(2r)e^{-\phi(r)}\sn_{\kappa_{\min}}^{m-1}(r)dr\right)\cdot \vol(\SSS^{m-1}),\label{equ:d1}
    \#
    where $m$ is the dimension of $\cM$ and $\kappa_{\min}<0$ is the lower sectional curvature bound of $\cM$ used in Condition \ref{cond:5.2}. By property (2) in Condition \ref{cond:5.2}, we can show that the integral in  \eqref{equ:d1} is finite. As a result, the KL divergence $D_{\KL}(P_{\alpha(t)}\|P_{\alpha(s)})$ is finite.  
\end{proof}

In Lemma \ref{lma:d2}, we will consider a unit-speed geodesic $\alpha(t)$ with a constant compact component, and investigate the first-order asymptotic behavior of the KL divergence $D_{\KL}(P_{\alpha(0)}\|P_{\alpha(t)})$ as $t\to 0$. Our analysis is based on the differentiability and symmetry of a function $I(t)$, which is defined by
\$
I(t)=\int_{\cM}\phi(d_g(x,\alpha(t)))e^{-\phi(d_g(x,\alpha(0)))}\dvol(x).
\$
This $I(t)$ relates to the KL divergence in the sense that 
\$
D_{\KL}(P_{\alpha(0)}\|P_{\alpha(t)})=\frac{1}{Z(\phi)}(I(t)-I(0)),
\$
where $Z(\phi)$ is the normalizing constant related to the density $f(x;\alpha,\phi)$.

\begin{lemma}\label{lma:d2}
    Let $(\cM,g^{\cM})$ be a simply connected noncompact Riemannian symmetric space  and $\phi$ a function satisfying Condition  \ref{cond:5.2}. Let $\alpha(t)$ be a unit-speed geodesic in $\cM$ with $\alpha(0)=\alpha$ and a constant compact component. Define
    \$
    I(t)=\int_{\cM}\phi(d_g(x,\alpha(t)))e^{-\phi(d_g(x,\alpha))}\dvol(x).
    \$
    Then $I(t)$ is differentiable at any $t\in\RR$ and 
    \#\label{equ:d2}
    I'(t)=\int_{\cM}\frac{d}{dt}\phi(d_g(x,\alpha(t)))e^{-\phi(d_g(x,\alpha))}\dvol(x).
    \#
    Moreover, $I'(0)=0$. 
\end{lemma}

\begin{remark}
    Since $\phi$ is differentiable on $(0,\infty)$ and $\alpha(\cdot)$ is a geodesic with a constant compact component, the function $\phi(d_g(x,\alpha(t)))$ is differentiable at $t$ for any $x\neq \alpha(t)$. Hence, the integrand in \eqref{equ:d2} is well-defined in $\cM$ except for a null set.  
\end{remark}


\begin{proof}
    By the proof of Lemma \ref{lma:d1}, we know $I(t)$ is integrable for all $t\in\RR$. To further demonstrate its differentiability, we fix any $t\in\RR$ and vary $\Delta\in[-1,1]$. Then consider the following first-order difference
    \#\label{equ:d12}
    \frac{I(t+\Delta)-I(t)}{\Delta}=\int_{\cM}\frac{\phi(d_g(x,\alpha(t+\Delta)))-\phi(d_g(x,\alpha(t)))}{\Delta}e^{-\phi(d_g(x,\alpha))}\dvol(x).
    \#
    Since $\phi$ is differentiable on $(0,\infty)$ and $\alpha(\cdot)$ is a geodesic in $\cM$ with a constant compact component,  it holds for any $x\neq\alpha(t)$ that
    \$
    \lim_{\Delta\to 0}\frac{\phi(d_g(x,\alpha(t+\Delta)))-\phi(d_g(x,\alpha(t)))}{\Delta}=\frac{d}{dt}\phi(d_g(x,\alpha(t))),
    \$
    where the right hand side of the above equation is well-defined for any $x\neq \alpha(t)$. Consequently, the integrand in \eqref{equ:d12} has the following convergence
    \$
    \lim_{\Delta\to 0}\frac{\phi(d_g(x,\alpha(t+\Delta)))-\phi(d_g(x,\alpha(t)))}{\Delta}e^{-\phi(d_g(x,\alpha))}=\frac{d}{dt}\phi(d_g(x,\alpha(t)))e^{-\phi(d_g(x,\alpha))}
    \$
    for all $x\neq \alpha(t)$.  
    Now we proceed to find a dominating function for the integrand in \eqref{equ:d12}. Since $\alpha(\cdot)$ is a unit-speed geodesic in $\cM$, it holds that 
    \$
    d_g(x,\alpha(t+\Delta))&\leq d_g(x,\alpha)+t+|\Delta|\leq d_g(x,\alpha)+t+1,\\
    d_g(x,\alpha(t))&\leq d_g(x,\alpha)+t,
    \$
    where we use the triangle inequality, the constraint $|\Delta|\leq 1$, and the condition that $\alpha(0)=\alpha$. Since $\phi$ is an increasing and convex function and is differentiable on $(0,\infty)$, $\phi$ is a Lipschitz function over $[0,l]$ with a Lipschitz constant $\phi'(l)$ for any $l>0$. Then it follows that
    \$
    |\phi(d_g(x,\alpha(t+\Delta)))-\phi(d_g(x,\alpha(t)))|&\leq \phi'(d_g(x,\alpha)+t+1)\cdot |d_g(x,\alpha(t+\Delta))-d_g(x,\alpha(t))|\\
    &\leq  \phi'(d_g(x,\alpha)+t+1)\cdot |\Delta|,
    \$
    where the second inequality uses the triangle inequality.
    Using this estimate, we can show that the integrand in \eqref{equ:d12} is dominated by the following function
    \#\label{equ:dominating}
    \left|\frac{\phi(d_g(x,\alpha(t+\Delta)))-\phi(d_g(x,\alpha(t)))}{\Delta}e^{-\phi(d_g(x,\alpha))}\right|\leq \phi'(d_g(x,\alpha)+t+1)e^{-\phi(d_g(x,\alpha))}.
    \#
    Using the polar coordinate expression \eqref{equ:integral} of the integral and Theorem \ref{thm:2.3}, we can show that the dominating function in \eqref{equ:dominating} is integrable. Specifically,
    \$
    &\int_{\cM}\phi'(d_g(x,\alpha)+t+1)\cdot e^{-\phi(d_g(x,\alpha))}\dvol(x)\\
    \leq 
    & \left(\int_0^{\infty}\phi'(r+t+1)e^{-\phi(r)}\sn_{\kappa_{\min}}^{m-1}(r)dr\right)\cdot \vol(\SSS^{m-1})\\
    \overset{({\rm i})}{\leq}
    & \left(\int_0^{t+1}\phi'(2(t+1))e^{-\phi(r)}\sn_{\kappa_{\min}}^{m-1}(r)dr +\int_{t+1}^\infty \phi'(2r)e^{-\phi(r)}\sn_{\kappa_{\min}}^{m-1}(r)dr\right)\cdot \vol(\SSS^{m-1})\\
    \overset{({\rm ii})}{<}  
    & \infty,
    \$
    where (i) uses the condition that $\phi$ is convex and (ii) uses property (4) in Condition \ref{cond:5.2}. Thus, by applying the dominated convergence theorem, we can show that $I(t)$ is differentiable at $t$ and 
    \$
    I'(t)=\lim_{\Delta\to0}\frac{I(t+\Delta)-I(t)}{\Delta}=\int_{\cM}\frac{d}{dt}\phi(d_g(x,\alpha(t)))e^{-\phi(d_g(x,\alpha))}\dvol(x).
    \$
    Finally, by Proposition \ref{prop:5.1}, we can show that the symmetry $I(t)=I(-t)$ holds for all $t\in\RR$, which implies that $I'(0)=0$. 
\end{proof}

In Lemma \ref{lma:d3}, we again consider a unit-speed geodesic $\alpha(t)$ in $\cM$ with a constant compact component. The objective is to study the second-order  behavior of the function $I(t)$ near $t=0$.  

\begin{lemma}\label{lma:d3}
    Let $(\cM,g^{\cM})$ be a simply connected noncompact Riemannian symmetric space  and $\phi$ a function satisfying Condition   \ref{cond:5.2}. Let $\alpha(t)$ be a unit-speed geodesic in $\cM$ with $\alpha(0)=\alpha$ and a constant compact component. Define
    \$
    I(t)=\int_{\cM}\phi(d_g(x,\alpha(t)))e^{-\phi(d_g(x,\alpha))}\dvol(x).
    \$
    Then for sufficient small $t$, we have $I(t)-I(0)\leq Ct^2$ for some constant $C$ independent of $t$.
\end{lemma}

\begin{proof}
    By the proof of Lemma \ref{lma:d1}, we can show that $I(t)$ is integrable for all $t\in\RR$. By Proposition \ref{prop:5.1},  we know that the symmetry $I(t)=I(-t)$ holds, and thus it suffices to consider the case $t\geq 0$. Since we only examine the behavior of $I(t)$ near $0$, we  assume without loss of generality that $t\leq 1$. 
    By Lemma \ref{lma:d2}, we know that $I'(0)=0$, where $I'(t)$ is given by \eqref{equ:d2}. Therefore, we can rewrite $I(t)-I(0)$ as follows:
    \#
    &I(t)-I(0) = I(t)-I(0)-tI'(0)\notag\\
    =\ &\int_{\cM}\left(\phi(d_g(x,\alpha(t)))-\phi(d_g(x,\alpha))-t\cdot \left.\frac{d}{dt}\right|_{t=0}\phi(d_g(x,\alpha(t)))\right)e^{-\phi(d_g(x,\alpha))}\dvol(x).\notag
    \#
    Then by taking the absolute value, we shall obtain the following upper bound
    \#
    &|I(t)-I(0)|\notag\\
    \leq\ & \int_{\cM}\left|\phi(d_g(x,\alpha(t)))-\phi(d_g(x,\alpha))-t\cdot \left.\frac{d}{dt}\right|_{t=0}\phi(d_g(x,\alpha(t)))\right|\cdot e^{-\phi(d_g(x,\alpha))}\dvol(x)\label{equ:integral-upper bound}
    \#
    To give an upper bound on the above term, we consider the following two cases separately.
    \begin{itemize}
        \item {\bf Case 1. The dimension of $\cM$ is 1.} In this case, $\cM=\RR$, and we can assume without loss of generality that $\alpha(t)=\alpha+t$. Under this assumption, the term in \eqref{equ:integral-upper bound} can be rewritten as follows
        \#
        &\int_{-\infty}^{\infty}\left|\phi(|x-\alpha-t|)-\phi(|x-\alpha|)-t\cdot\left.\frac{d}{dt}\right|_{t=0}\phi(|x-\alpha-t|)\right|\cdot e^{-\phi(|x-\alpha|)}dx\notag\\
        =\ &\int_{-\infty}^{\infty}\left|\phi(|x-t|)-\phi(|x|)-t\cdot\left.\frac{d}{dt}\right|_{t=0}\phi(|x-t|)\right|\cdot e^{-\phi(|x|)}dx,\label{equ:integral-L} \\
        =\ &\int_{-\infty}^{\infty}L(x)dx,\label{equ:integral-L2}
        \#
        where the first equality uses the change of variables ($x-\alpha\to x$) and $L(x)$ is the integrand in \eqref{equ:integral-L}. Notably, the integrand $L(x)$ and the integral $\int L(x)dx$ are independent of the choice of $\alpha$. To evaluate the integral \eqref{equ:integral-L2}, we decompose it into the following three terms:
        \$
        \int_{-\infty}^{\infty}L(x)dx=\int_{-\infty}^0L(x)dx+\int_{0}^tL(x)dx+\int_{t}^{\infty}L(x)dx.
        \$
        We shall upper bound these three terms separately. 
        \begin{itemize}[leftmargin=*]
            \item First, for any $x\in(-\infty,0)$, 
        \$
        L(x)&=\left|\phi(t-x)-\phi(-x)-t\cdot\phi'(-x)\right|\cdot e^{-\phi(-x)},\\
        &=\left(\phi(t-x)-\phi(-x)-t\cdot\phi'(-x)\right)\cdot e^{-\phi(-x)},
        \$
        where the second equality uses the fact that $\phi$ is a convex function. Since $\phi$ is second-order continuously differentiable on $(0,\infty)$, by using the Taylor theorem, we have
        \$
        L(x)= \int_{0}^{t}\phi''(\xi-x)(t-\xi)d\xi\cdot e^{-\phi(-x)},
        \$
        where $\phi''(\xi-t)$ is nonnegative due to the convexity of $\phi$. By applying the Tonelli theorem, we can show that
        \#
        \int_{-\infty}^0L(x)dx&=\int_{-\infty}^0\int_0^{t}\phi''(\xi-x)(t-\xi)d\xi\cdot e^{-\phi(-x)}dx\notag\\
        &=\int_{0}^t(t-\xi)d\xi\cdot\int_{-\infty}^0\phi''(\xi-x)e^{-\phi(-x)}dx.\label{equ:L_integral_1}
        \#
        Since $\phi$ is an increasing function and $0\leq \xi\leq t\leq 1$, it holds that
        \$
        \phi(-x)\geq \phi(\max\{\xi-x-1,0\}). 
        \$
        Therefore,
        \#
        \int_{-\infty}^0\phi''(\xi-x)e^{-\phi(-x)}dx&\leq \int_{-\infty}^{0}\phi''(\xi-x)e^{-\phi(\max\{\xi-x-1,0\})}dx\notag\\
        &\overset{({\rm i})}{\leq} \int_0^{\infty}\phi''(x)e^{-\phi(\max\{x-1,0\})}dx\label{equ:upper-bound-1}\\
        &\overset{({\rm ii})}{<}\infty,\notag
        \#
        where (i) uses the change of variables $(\xi-x\to x)$ and the non-negativity of the integrand, and (ii) uses Condition \ref{cond:5.2}. Notably, the integral in \eqref{equ:upper-bound-1} is a finite constant independent of the choice of $\xi\in[0,t]$ and $t\in[0,1]$. By substituting this upper bound into \eqref{equ:L_integral_1}, we obtain that 
        \#
        \int_{-\infty}^0L(x)dx&=\int_0^t(t-\xi)d\xi\cdot \int_0^{\infty}\phi''(x)e^{-\phi(\max\{x-1,0\})}dx \leq C_1t^2,\notag
        \#
        where $C_1>0$ is a finite constant independent of $t$.  
        \item Second, for any $x\in(0,t)$, 
        \$
        L(x)&=|\phi(t-x)-\phi(x)+t\cdot\phi'(x)|\cdot e^{-\phi(x)}\\
        &\leq \left(\phi'(t)\cdot |t-2x|+t\cdot \phi'(t)\right)\cdot e^{-\phi(0)}\\
        &\leq \phi'(1)\cdot \left(|t-2x|+t\right)\cdot e^{-\phi(0)},
        \$
        where the first inequality uses the fact that $\phi$ is an increasing and convex function, and the second inequality uses the condition that $t\leq 1$ and $\phi'(t)\leq \phi'(1)$. Substituting this into the integral $\int_0^tL(x)dx$, we obtain that
        \#
        \int_0^tL(x)dx&\leq \phi'(1)\cdot e^{-\phi(0)}\cdot \int_0^t(|t-2x|+t)dx\notag\\
        &\leq C_2 t^2,\notag
        \#
        where $C_2>0$ is a finite constant independent of $t$.
        \item Third, for any $x\in(t,\infty)$, 
        \$
        L(x)&=|\phi(x-t)-\phi(x)+t\cdot \phi'(x)|\cdot e^{-\phi(x)}\\
        &=(\phi(x-t)-\phi(x)+t\cdot\phi'(x))\cdot e^{-\phi(x)}, 
        \$
        where the second equality uses the fact that $\phi$ is a convex function. Since $\phi$ is second-order continuously differentiable on $(0,\infty)$, by using the Taylor theorem, we have
        \$
        L(x)=\int_{0}^t\phi''(x-\xi)(t-\xi)d\xi\cdot e^{-\phi(x)},
        \$
        where $\phi''(x-\xi)$ is nonnegative due to the convexity of $\phi$. By applying the Tonelli theorem, we can show that 
        \$
        \int_t^{\infty}L(x)dx&=\int_t^{\infty}\int_0^t\phi''(x-\xi)(t-\xi)d\xi\cdot e^{-\phi(x)}dx\\
        &=\int_0^{t}(t-\xi)d\xi\cdot \int_t^{\infty}\phi''(x-\xi)\cdot e^{-\phi(x)}dx\\
        &\leq \int_0^t(t-\xi)d\xi\cdot \int_t^{\infty}\phi''(x-\xi)\cdot e^{-\phi(x-\xi)}dx,
        \$
        where the inequality uses the fact that $\phi$ is an increasing function. By using the change of variables $(x-\xi\to x)$, we can show that 
        \$
        \int_t^{\infty}L(x)dx&\leq \int_0^t(t-\xi)d\xi\cdot \int_{t-\xi}^{\infty}\phi''(x)\cdot e^{-\phi(x)}dx\\
        &\leq \int_0^t(t-\xi)d\xi\cdot \int_0^{\infty}\phi''(x)\cdot e^{-\phi(x)}dx\\
        &\leq C_3t^2,
        \$
        where $C_3>0$ is a finite constant independent of $t$ and the third inequality uses Condition \ref{cond:5.2}. 
        \end{itemize}
        By combining the above three cases, we find that
        \$
        \int_{-\infty}^{\infty}L(x)dx\leq (C_1+C_2+C_3)t^2,
        \$
        where $C_1,C_2,C_3>0$ are finite constants independent of $t$. This implies that $|I(t)-I(0)|\leq Ct^2$ for some finite constant $C>0$ independent of $t$. 
        \item {\bf Case 2. The dimension of $\cM$ is larger than 1.} In this case, the set 
        \$
        \Img({\alpha})\coloneqq \{\alpha(s)\mid s\in\RR\}
        \$ 
        is of measure zero in $\cM$. Therefore, the integral in \eqref{equ:integral-upper bound} can be evaluated on $\cM-\Img(\alpha)$. In particular, by Lemma \ref{lma:d-varrho_phi}, for any $x\in\cM-\Img(\alpha)$, the function $\varrho_0(s)=d_g(x,\alpha(s))$ is positive and smooth at any $s\in\RR$. Since $\phi$ is second-order continuously differentiable over $(0,\infty)$, the function $\varrho_{\phi}(s)=\phi(\varrho_0(s))$ is also second-order continuously differentiable at any $s\in\RR$. By the chain rule, we can compute the first two order derivatives of $\varrho_{\phi}(s)$ as follows:
        \#
        \varrho_{\phi}'(s)&=\phi'(\varrho_0(s))\cdot \varrho'_0(s),\notag\\
        \varrho_{\phi}''(s)&=\phi''(\varrho_0(s))\cdot \left( \varrho_0'(s)\right)^2 +\phi'(\varrho_0(s))\cdot  \varrho_0''(s),\label{equ:varrho_phi_''}
        \#
        where we use the fact that $\varrho(s)$ is smooth over $s\in\RR$ when $x\in\cM-\Img(\alpha)$. By the Taylor theorem, we have
        \$
        \varrho_{\phi}(t)-\varrho_{\phi}(0)-t\varrho_{\phi}'(0)=\int_0^t\varrho_{\phi}''(\xi)(t-\xi)d\xi.
        \$
        Taking the absolute value, we can obtain the following bound 
        \$
        |\varrho_{\phi}(t)-\varrho_{\phi}(0)-t\varrho_{\phi}'(0)|\leq \int_0^t|\varrho_{\phi}''(\xi)|(t-\xi)d\xi. 
        \$
        Substituting this into \eqref{equ:integral-upper bound}, we obtain 
        \$
        |I(t)-I(0)|\leq \int_{\cM}\int_0^t|\varrho_{\phi}''(\xi)|(t-\xi)d\xi\cdot e^{-\varrho_{\phi}(0)}\dvol(x).
        \$
        Since the integrand in the right hand side of the above inequality is nonnegative and measurable, we can employ the Tonelli's theorem to obtain that
        \#\label{equ:d.22}
        |I(t)-I(0)|\leq \int_{0}^t(t-\xi)d\xi\int_{\cM}|\varrho_{\phi}''(\xi)|e^{-\varrho_{\phi}(0)}\dvol(x).
        \#
        To evaluate the above the above upper bound, we first upper bound $|\varrho_{\phi}''(\xi)|$ using Lemma \ref{lma:d-varrho_phi} and the expression \eqref{equ:varrho_phi_''}. Specifically, for any $x\in\cM-\Img(\alpha)$, the following upper bound holds 
        \$
        |\varrho_{\phi}''(\xi)|&\leq \phi''(\varrho_0(\xi))\cdot (\varrho_0'(\xi))^2+\phi'(\varrho_0(\xi))\cdot |\varrho_0''(\xi)|\\
        &\leq \phi''(\varrho_0(\xi))+3\cdot \phi'(\varrho_0(\xi))\cdot \frac{\sn_{\kappa_{\min}}'(\varrho_{0}(\xi))}{\sn_{\kappa_{\min}}(\varrho_0(\xi))}
        \$
        where the first inequality uses the triangle inequality and the fact that $\phi$ is an increasing and convex function, and the second inequality uses Lemma \ref{lma:d-varrho_phi}. Here $\kappa_{\min}<0$ is a lower bound on the sectional curvatures of $\cM$ used in Condition \ref{cond:5.2} and $\sn_{\kappa_{\min}}(\cdot)$ is given by \eqref{equ:sn}. Using this upper bound and the fact that $\Img(\alpha)$ is of measure zero, we can show that 
        \$
        &\int_{\cM}|\varrho_{\phi}''(\xi)|e^{-\varrho_{\phi}(0)}\dvol(x)\\
        \leq\ &\int_{\cM} \left(\phi''(\varrho_0(\xi))+3\cdot \phi'(\varrho_0(\xi))\cdot \frac{\sn_{\kappa_{\min}}'(\varrho_{0}(\xi))}{\sn_{\kappa_{\min}}(\varrho_0(\xi))}\right)\cdot e^{-\phi(\varrho_0(0))}\dvol(x)
        \$
        Since $\alpha$ is a unit-speed geodesic in $\cM$ and $\xi\in[0,1]$, we can show by the triangle inequality that
        \$
        \varrho_0(0)\geq \max\{\varrho_0(\xi)-1,0\}.
        \$
        Therefore, since $\phi$ is an increasing function, we have
        \#
        &\int_{\cM}|\varrho_{\phi}''(\xi)|e^{-\varrho_{\phi}(0)}\dvol(x)\notag\\
        \leq\ &\int_{\cM} \left(\phi''(\varrho_0(\xi))+3\cdot \phi'(\varrho_0(\xi))\cdot \frac{\sn_{\kappa_{\min}}'(\varrho_{0}(\xi))}{\sn_{\kappa_{\min}}(\varrho_0(\xi))}\right)\cdot e^{-\phi(\max\{\varrho_0(\xi)-1,0\})}\dvol(x)\notag\\
        \eqqcolon\ &J(\xi).\label{equ:J-xi}
        \#
        The integrand in $J(\xi)$ is a function of $\varrho_0(\xi)=d_g(x,\alpha(\xi))$, so we can use the homogeneity of $\cM$ to demonstrate that 
        \#\label{equ:d.23}
        J(\xi)=J(0),\quad\forall \xi\in[0,1].
        \#
        This proof is omitted, since it is similar to those for Proposition \ref{prop:3.2} and \ref{prop:5.1}, where we apply a suitable isometry of $\cM$ to the integral and then employ property \eqref{equ:homo}. The property \eqref{equ:d.23} states that the integral $J(\xi)$ is independent of $\xi$, which means that one  only needs to focus on the quantity $J(0)$. In fact, by using the polar coordinate expression \eqref{equ:integral} of the integral and Theorem \ref{thm:2.3}, we can show that the quantity $J(0)$ is finite:
        \$
        J(0)&\leq \vol(\SSS^{m-1})\cdot \int_{0}^{\infty}\left(\phi''(r)+3\phi'(r)\frac{\sn_{\kappa_{\min}}'(r)}{\sn_{\kappa_{\min}}(r)}\right)e^{-\phi(\max\{r-1,0\})}\sn_{\kappa_{\min}}^{m-1}(r)dr\\
        &<\infty,
        \$
        where $m$ is the dimension of $\cM$ and the second inequality uses Condition \ref{cond:5.2}. By combining this with \eqref{equ:d.22}, \eqref{equ:J-xi}, and \eqref{equ:d.23}, we can show that
        \$
        |I(t)-I(0)|\leq J(0)\cdot \int_0^t(t-\xi)d\xi\leq Ct^2,\quad\forall t\in[0,1],
        \$
        where $C>0$ is a finite constant independent of $t$.  
    \end{itemize}
    The proof of the lemma is complete by combining the above two cases.
    \end{proof}
    



    

\subsection{Simply connected compact Riemannian symmetric spaces}\label{sec:d2}

In this section, we shall delve into the case of a simply connected compact Riemannian symmetric space. We will first provide several geometric lemmas  in Section \ref{sec:d21}, and then present detailed statistical analysis in Section \ref{sec:d22}.

\subsubsection{Geometric lemmas}\label{sec:d21}

In this section, we will present several geometric lemmas on a simply connected compact Riemannian symmetric space $(\cM,g^{\cM})$. In Lemma \ref{lma:d.21}, we consider a geodesic $\alpha(t)$ in $\cM$ and investigate the differentiability of the function $\varrho_0(t)=d_g(x,\alpha(t))$ for any fixed $x\in\cM-\Img(\alpha)\cup\Ucut(\alpha)$, where
\#
\Img(\alpha)&=\{\alpha(t)\mid t\in\RR\},\label{equ:Img-alpha-compact}\\
\Ucut(\alpha)&=\cup_{t\in\RR}\Cut(\alpha(t)),\label{equ:Ucut-alpha}
\#
and $\Cut(\alpha(t))$ denotes the cut locus of $\alpha(t)$ in $\cM$. Since $\cM$ is a simply connected compact Riemannian symmetric space, the dimension $\dim(\cM)>1$ and thus $\Img(\alpha)$ is of measure zero in $\cM$. Additionally, by Lemma \ref{lma:a2}, $\Ucut(\alpha)$ is also of measure zero in $\cM$. Therefore, Lemma \ref{lma:d.21} proves the smoothness of $\varrho_0(t)$ in $\RR$ for almost all $x$ in $\cM$. 

\begin{lemma}\label{lma:d.21}
    Let $(\cM,g^{\cM})$ be a simply connected compact Riemannian symmetric space, and $\alpha(t)$ a geodesic in $\cM$. Then for any $x\in\cM-\Img(\alpha)\cup\Ucut(\alpha)$, the function $\varrho_0(t)=d_g(x,\alpha(t))$ is smooth at all $t\in\RR$, where $\Img(\alpha)$ and $\Ucut(\alpha)$ are defined by \eqref{equ:Img-alpha-compact} and \eqref{equ:Ucut-alpha} respectively.
\end{lemma}

\begin{proof}
    Since $x\in\cM-\Img(\alpha)\cup\Ucut(\alpha)$, it holds that $\alpha(t)\neq x$ and $\alpha(t)\notin\Cut(x)$ for all $t\in\RR$, where $\Cut(x)$ is the cut locus of $x$ in $\cM$. Because the distance function $d_g(x,y)$ is smooth at any $y\in\cM-\{x\}\cup\Cut(x)$, it follows from the chain rule that the function $\varrho_0(t)=d_g(x,\alpha(t))$ is also smooth at all $t\in\RR$. 
\end{proof}

Lemma \ref{lma:d.22} considers a unit-speed geodesic $\alpha(t)$ and a function $\phi$ satisfying Condition \ref{cond:5.6}. We aim to examine the first two order derivatives of the function $\varrho_{\phi}(t)=\phi(d_g(x,\alpha(t)))$ for any  $x\in\cM-\Img(\alpha)\cup\Ucut(\alpha)$, where $\Img(\alpha)$ and $\Ucut(\alpha)$ are given by \eqref{equ:Img-alpha-compact} and \eqref{equ:Ucut-alpha}, respectively. 

\begin{lemma}\label{lma:d.22}
    Let $(\cM,g^{\cM})$ is a simply connected compact Riemannian symmetric space, $\alpha(t)$ be a unit-speed geodesic in $\cM$, and $\phi$ be a function satisfying Condition  \ref{cond:5.6}. Then for any $x\in\cM-\Img(\alpha)\cup\Ucut(\alpha)$, the function $\varrho_\phi(t)=\phi(d_g(x,\alpha(t)))$ is second-order continuously differentiable at any $t\in\RR$, where $\Img(\alpha)$ and $\Ucut(\alpha)$ are given by \eqref{equ:Img-alpha-compact} and \eqref{equ:Ucut-alpha} respectively. In addition,  the absolute values of its first two order derivatives have the following upper bound
    \$
    |\varrho_\phi'(t)|&\leq \phi'(d_g(x,\alpha(t))),\\
    |\varrho_\phi''(t)|&\leq |\phi''(d_g(x,\alpha(t)))|+\frac{\phi'(d_g(x,\alpha(t)))}{d_g(x,\alpha(t))}\cdot(1+\frac{|\varrho''(t)|}{2}), 
    \$
    where $\varrho(t)=d_g^2(x,\alpha(t))$. 
\end{lemma}

\begin{proof}
    Lemma \ref{lma:d.21} has shown that the function $\varrho_0(t)=d_g(x,\alpha(t))$ is smooth at any $t\in\RR$ for any fixed $x\in\cM-\Img(\alpha)\cup\Ucut(\alpha)$. In addition, one can verify that $\varrho_0(t)\in(0,r_{\cM})$ for any $t\in\RR$ and any $x\in\cM-\Img(\alpha)\cup\Ucut(\alpha)$, where $r_{\cM}=\sup_{x,y}d_g(x,y)$ is the maximum radius of $\cM$. Hence, the composite function $\varrho_\phi(t)=\phi(\varrho_0(t))$ is second-order continuously differentiable at any $t\in\RR$ for any $x\in\cM-\Img(\alpha)\cup\Ucut(\alpha)$, since $\phi$ is second-order continuously differentiable on $(0,r_{\cM})$. Additionally, by the chain rule, we can compute the first two order derivatives of $\varrho_\phi(t)$ as follows
    \$
    \varrho_\phi'(t)&=\phi'(\varrho_0(t))\cdot \varrho_0'(t),\\
    \varrho_\phi''(t)&=\phi''(\varrho_0(t))\cdot(\varrho_0'(t))^2+\phi'(\varrho_0(t))\cdot \varrho_0''(t). 
    \$
    By taking the absolute values and using the triangle inequality, we obtain that
    \#
    |\varrho_\phi'(t)|&\leq \phi'(\varrho_0(t))\cdot|\varrho_0'(t)|\notag\\
    &\leq \phi'(\varrho_0(t)),\notag\\
    |\varrho_\phi''(t)|&\leq |\phi''(\varrho_0(t))|\cdot|\varrho_0'(t)|^2+\phi'(\varrho_0(t))\cdot |\varrho_0''(t)|\notag\\
    &\leq |\phi''(\varrho_0(t))|+\phi'(\varrho_0(t))\cdot |\varrho_0''(t)|,\label{equ:|varrho_phi|-27}
    \#
    where we use the fact that $|\varrho_0'(t)|\leq 1$ and $\phi'(\cdot)$ is nonnegative. 
    Let $\varrho(t)=d_g^2(x,\alpha(t))$. Then 
    \$
    \varrho_0(t)&=\sqrt{\varrho(t)},\\
    \varrho_0'(t)&=\frac{1}{2\sqrt{\varrho(t)}}\cdot \varrho'(t),
    \$
    and 
    \$
    \varrho_0''(t)=\frac{1}{2\sqrt{\varrho(t)}}\cdot \varrho''(t)-\frac{1}{4\varrho(t)\sqrt{\varrho(t)}}\cdot (\varrho'(t))^2.
    \$
    By taking the absolute value and using the triangle inequality, we can show that
    \$
    |\varrho_0''(t)|\leq \frac{|\varrho''(t)|}{2\sqrt{\varrho(t)}}+\frac{1}{\sqrt{\varrho(t)}},
    \$
    where we use the fact that $|\varrho'(t)|\leq 2\sqrt{\varrho(t)}$. Substituting this into \eqref{equ:|varrho_phi|-27}, we obtain that
    \$
    |\varrho_\phi''(t)|\leq |\phi''(\varrho_0(t))|+\phi'(\varrho_0(t))\cdot \frac{2+|\varrho''(t)|}{2\varrho_0(t)},
    \$
    which concludes the proof. 
\end{proof}

\subsubsection{KL divergence}\label{sec:d22}

This section will present lemmas related to the KL divergence $D_{\KL}(P_{\alpha}\|P_{\alpha'})$, where $P_{\alpha}$ denotes the distribution $f(x;\alpha,\phi)$. First, in Lemma \ref{lma:d.2.2.1}, we will establish the finiteness of this KL divergence under mild conditions. 

\begin{lemma}\label{lma:d.2.2.1}
    Let $(\cM,g^{\cM})$ be a simply connected compact Riemannian symmetric space, and $\phi$ a function satisfying Condition \ref{cond:3.1}. Then the KL divergence $D_{\KL}(P_{\alpha}\|P_{\alpha'})$ is finite for all $\alpha,\alpha'\in\cM$, where $P_{\alpha}$ denotes the distribution $f(x;\alpha,\phi)$. 
\end{lemma}

\begin{proof}
    By definition of the KL divergence and $P_{\alpha}$, we have
    \$
    D_{\KL}(P_{\alpha}\|P_{\alpha'})=\frac{1}{Z(\phi)}\cdot\int_{\cM}(\phi(d_g(x,\alpha'))-\phi(d_g(x,\alpha )))e^{-\phi(d_g(x,\alpha ))}\dvol(x).
    \$
    Since $\cM$ is a compact manifold and the integrand in the above integral is finite, the KL divergence $D_{\KL}(P_{\alpha}\|P_{\alpha'})$ is finite for all pairs $\alpha,\alpha'\in\cM$. 
\end{proof}

In Lemma \ref{lma:d.2.2.2}, we will consider a unit-speed geodesic $\alpha(t)$ in $\cM$. The objective is to investigate the first-order asymptotic behavior of the KL divergence $D_{\KL}(P_{\alpha(0)}\|P_{\alpha(t)})$ as $t\to 0$, where $P_{\alpha}$ denotes the distribution $f(x;\alpha,\phi)$. Similar to the noncompact case in Section \ref{sec:d1.2}, the analysis here is again based on the differentiability and symmetry of the following function
\$
I(t)=\int_{\cM}\phi(d_g(x,\alpha(t)))e^{-\phi(d_g(x,\alpha(0)))}\dvol(x).
\$
This $I(t)$ relates to the KL divergence in the sense that 
\$
D_{\KL}(P_{\alpha(0)}\|P_{\alpha(t)})=\frac{1}{Z(\phi)}(I(t)-I(0)),
\$
where $Z(\phi)$ is the normalizing constant related to the density $f(x;\alpha,\phi)$.

\begin{lemma}\label{lma:d.2.2.2}
    Let $(\cM,g^{\cM})$ be a simply connected compact Riemannian symmetric space, and $\phi$ a function satisfying Condition \ref{cond:5.6}. Let $\alpha(t)$ is a unit-speed geodesic in $\cM$ with $\alpha(0)=\alpha$. Define
    \$
    I(t)=\int_{\cM}\phi(d_g(x,\alpha(t)))e^{-\phi(d_g(x,\alpha))}\dvol(x).
    \$
    Then $I(t)$ is differentiable at any $t\in\RR$ and
    \#\label{equ:compact-I'}
    I'(t)=\int_{\cM}\frac{d}{dt}\phi(d_g(x,\alpha(t)))\cdot e^{-\phi(d_g(x,\alpha))}\dvol(x).
    \#
    Moreover, $I'(0)=0$.
\end{lemma}

\begin{remark}
    Lemma \ref{lma:d.21} shows that the function $\varrho_0(t)=d_g(x,\alpha(t))$ is smooth at all $t\in\RR$ for any $x\in\cM-\Img(\alpha)\cup\Ucut(\alpha)$, where $\Img(\alpha)$ and $\Ucut(\alpha)$ are defined by \eqref{equ:Img-alpha-compact} and \eqref{equ:Ucut-alpha} respectively, and the set $\Img(\alpha)\cup\Ucut(\alpha)$ is of measure zero in $\cM$. Thus, for almost all $x\in\cM$, the function $\phi(\varrho_0(t))$ is differentiable at all $t\in\RR$, since $\phi$ is differentiable over $(0,r_{\cM})$, where $r_{\cM}=\sup_{x,y}d_g(x,y)$ is the maximum radius of $\cM$. This implies that the derivative $\frac{d}{dt}\phi(\varrho_0(t))$ and thus the integrand in \eqref{equ:compact-I'} are well-defined in $\cM$ except for a null set.
\end{remark}

\begin{proof}
    By the proof of Lemma \ref{lma:d.2.2.1}, we can show that $I(t)$ is finite for all $t\in\RR$. To demonstrate its differentiability, we fix any $t\in\RR$ and vary $\Delta\in[-1,1]$. Then we consider the following first-order difference
    \#\label{equ:d.25}
    \frac{I(t+\Delta)-I(t)}{\Delta}=\int_{\cM}\frac{\phi(d_g(x,\alpha(t+\Delta)))-\phi(d_g(x,\alpha(t)))}{\Delta}e^{-\phi(d_g(x,\alpha))}\dvol(x)
    \#
    By Lemma \ref{lma:d.21}, the function $\varrho_0(t)=d_g(x,\alpha(t))$ is smooth at all $t\in\RR$ for any $x\in\cM-\Img(\alpha)\cup\Ucut(\alpha)$, where  $\Img(\alpha)$ and $\Ucut(\alpha)$ are defined by \eqref{equ:Img-alpha-compact} and \eqref{equ:Ucut-alpha} respectively. Additionally, since $\phi$ is differentiable on $(0,r_{\cM})$, where $r_{\cM}=\sup_{x,y}d_g(x,y)$ is the maximum radius of $\cM$, the function $\phi(\varrho_0(t))$ is differentiable at all $t\in\RR$ for any $x\in\cM-\Img(\alpha)\cup\Ucut(\alpha)$. Therefore, for any $x\in\cM-\Img(\alpha)\cup\Ucut(\alpha)$, the following convergence holds
    \$
    \lim_{\Delta\to 0}\frac{\phi(d_g(x,\alpha(t+\Delta)))-\phi(d_g(x,\alpha(t)))}{\Delta}=\frac{d}{dt}\phi(d_g(x,\alpha(t))).
    \$
    Moreover, for any $x\in\cM-\Img(\alpha)\cup\Ucut(\alpha)$, the integrand in \eqref{equ:d.25} has the following convergence
    \$
    \lim_{\Delta\to 0}\frac{\phi(d_g(x,\alpha(t+\Delta)))-\phi(d_g(x,\alpha(t)))}{\Delta}e^{-\phi(d_g(x,\alpha))}=\frac{d}{dt}\phi(d_g(x,\alpha(t)))e^{-\phi(d_g(x,\alpha))}.
    \$
    This proves the almost surely convergence of the integrand in \eqref{equ:d.25} to the integrand in \eqref{equ:compact-I'} as $\Delta\to 0$, where we use the fact that $\Img(\alpha)\cup\Ucut(\alpha)$ is of measure zero in $\cM$. Now we proceed to find a dominating function on the integrand in \eqref{equ:d.25}. Specifically, since $\phi$ is a Lipschitz function over $[0,r_{\cM}]$ with some Lipschitz constant $L$, we have
    \$
    |\phi(d_g(x,\alpha(t+\Delta)))-\phi(d_g(x,\alpha(t)))|&\leq L\cdot |d_g(x,\alpha(t+\Delta))-d_g(x,\alpha(t))|\\
    &\leq L\cdot |\Delta|,
    \$
    where the second inequality uses the triangle inequality and the condition that $\alpha(\cdot)$ is a unit-speed geodesic. Using this upper bound and the fact that $\phi$ is nonnegative, we can show that
    \$
    \left|\frac{\phi(d_g(x,\alpha(t+\Delta)))-\phi(d_g(x,\alpha(t)))}{\Delta}\right|e^{-\phi(d_g(x,\alpha))}\leq L.
    \$
    Since the manifold $\cM$ is compact, the constant function $L$ is integrable over $\cM$. By the dominated convergence theorem, we conclude that $I(t)$ is differentiable at $t$ and 
    \$
    I'(t)=\lim_{\Delta\to0}\frac{I(t+\Delta)-I(t)}{\Delta}=\int_{\cM}\frac{d}{dt}\phi(d_g(x,\alpha(t)))e^{-\phi(d_g(x,\alpha))}\dvol(x).
    \$
    Finally, by Proposition \ref{prop:5.1}, we can show that the symmetry $I(t)=I(-t)$ holds for all $t\in\RR$, which implies that $I'(0)=0$. This concludes the proof.
\end{proof}

In Lemma \ref{lma:d.2.2.3}, we again consider a unit-speed geodesic $\alpha(t)$ in $\cM$. The objective is to study the second-order behavior of the function $I(t)$ near $t=0$. 

\begin{lemma}\label{lma:d.2.2.3}
    Let $(\cM,g^{\cM})$ be a simply connected compact Riemannian symmetric space, and $\phi$ a function satisfying Condition \ref{cond:5.6}. Let $\alpha(t)$ be a unit-speed geodesic in $\cM$ with $\alpha(0)=\alpha$. Define 
    \$
    I(t)=\int_{\cM}\phi(d_g(x,\alpha(t)))e^{-\phi(d_g(x,\alpha))}\dvol(x).
    \$
    Then for sufficiently small $t$, we have $I(t)-I(0)\leq Ct^2$ for some constant $C$ independent of $t$. 
\end{lemma}

\begin{proof}
    By the proof of Lemma \ref{lma:d.2.2.1}, we can show that $I(t)$ is finite for all $t\in\RR$. By Proposition \ref{prop:5.1}, we know that the symmetry $I(t)=I(-t)$ holds, and thus it suffices to consider the case $t\geq 0$. Since we only examine the behavior of $I(t)$ near $0$, we assume without loss of generality that $t\leq 1$. By Lemma \ref{lma:d.2.2.2}, we know that $I'(0)=0$, where $I'(t)$ is given by \eqref{equ:compact-I'}. Therefore, we can rewrite $I(t)-I(0)$ as follows:
    \#
    &I(t)-I(0) = I(t)-I(0)-tI'(0)\notag\\
    =\ &\int_{\cM}\left(\phi(d_g(x,\alpha(t)))-\phi(d_g(x,\alpha))-t\cdot \left.\frac{d}{dt}\right|_{t=0}\phi(d_g(x,\alpha(t)))\right)e^{-\phi(d_g(x,\alpha))}\dvol(x).\notag
    \#
    By taking the absolute value, we can obtain the following upper bound
    \#
    &|I(t)-I(0)|\notag\\
    \leq\ & \int_{\cM}\left|\phi(d_g(x,\alpha(t)))-\phi(d_g(x,\alpha))-t\cdot \left.\frac{d}{dt}\right|_{t=0}\phi(d_g(x,\alpha(t)))\right|\cdot e^{-\phi(d_g(x,\alpha))}\dvol(x)\notag\\
    \leq\ & \int_{\cM}\left|\phi(d_g(x,\alpha(t)))-\phi(d_g(x,\alpha))-t\cdot \left.\frac{d}{dt}\right|_{t=0}\phi(d_g(x,\alpha(t)))\right|\dvol(x),\label{equ:compact-integral-upper bound}
    \#
    where the second inequality uses the fact that $\phi$ is nonnegative. Our current objective is to establish an upper bound on the integral in \eqref{equ:compact-integral-upper bound}. Since $\phi$ is second-order continuously differentiable over $(0,r_{\cM})$, where $r_{\cM}=\sup_{x,y}d_g(x,y)$ is the maximum radius of $\cM$, the function $\phi(d_g(x,\alpha(t)))$ is second-order continously differentiable at $t\in\RR$ for any fixed $x\in\cM-\Img(\alpha)\cup\Ucut(\alpha)$. Here we use the chain rule and Lemma \ref{lma:d.21}, and $\Img(\alpha)$ and $\Ucut(\alpha)$ are given by \eqref{equ:Img-alpha-compact} and \eqref{equ:Ucut-alpha} respectively. Therefore, by using the Taylor theorem, we have 
    \$
    &\phi(d_g(x,\alpha(t)))-\phi(d_g(x,\alpha))-t\cdot\left.\frac{d}{dt}\right|_{t=0}\phi(d_g(x,\alpha(t)))\\
    =\ &\int_0^t\left(\left.\frac{d^2}{dt^2}\right|_{t=\xi}\phi(d_g(x,\alpha(t)))\right)(t-\xi)d\xi
    \$
    for any $x\in\cM-\Img(\alpha)\cup\Ucut(\alpha)$. By taking the absolute value, we obtain the following upper bound
    \$
    &\left|\phi(d_g(x,\alpha(t)))-\phi(d_g(x,\alpha))-t\cdot\left.\frac{d}{dt}\right|_{t=0}\phi(d_g(x,\alpha(t)))\right|\\
    \leq\ &\int_0^t\left|\left.\frac{d^2}{dt^2}\right|_{t=\xi}\phi(d_g(x,\alpha(t)))\right|(t-\xi)d\xi
    \$
    for any $x\in\cM-\Img(\alpha)\cup\Ucut(\alpha)$. Substituting this into \eqref{equ:compact-integral-upper bound} and using the fact that $\Img(\alpha)\cup\Ucut(\alpha)$ is of measure zero in $\cM$, we obtain the following inequality
    \$
    |I(t)-I(0)|\leq \int_\cM\int_0^t\left|\left.\frac{d^2}{dt^2}\right|_{t=\xi}\phi(d_g(x,\alpha(t)))\right|(t-\xi)d\xi\dvol(x).
    \$
    By using the Tonelli's theorem, we obtain that
    \#\label{equ:d.30}
    |I(t)-I(0)|\leq \int_0^t(t-\xi)d\xi\cdot \int_\cM\left|\left.\frac{d^2}{dt^2}\right|_{t=\xi}\phi(d_g(x,\alpha(t)))\right|\dvol(x).
    \#
    It suffices to  establish a suitable upper bound on the above integral. By Lemma \ref{lma:d.22},  it holds for any $x\in\cM-\Img(\alpha)\cup\Ucut(\alpha)$ that
    \$
    \left|\left.\frac{d^2}{dt^2}\right|_{t=\xi}\phi(d_g(x,\alpha(t)))\right|\leq |\phi''(d_g(x,\alpha(\xi)))|+\frac{\phi'(d_g(x,\alpha(\xi)))}{d_g(x,\alpha(\xi))}\cdot (1+\frac{|\varrho''(\xi)|}{2}),
    \$
    where $\varrho(t)=d_g^2(x,\alpha(t))$. Therefore,
    \#
    &\int_{\cM}\left|\left.\frac{d^2}{dt^2}\right|_{t=\xi}\phi(d_g(x,\alpha(t)))\right|\dvol(x)\notag\\
    \leq \ &\int_{\cM}\left(|\phi''(d_g(x,\alpha(\xi)))|+\frac{\phi'(d_g(x,\alpha(\xi)))}{d_g(x,\alpha(\xi))}\cdot (1+\frac{|\varrho''(\xi)|}{2})\right)\dvol(x)\eqqcolon J(\xi).\label{equ:def-J-xi-32}
    \#
    Given any $\xi\in\RR$, there exists an isometry $F\in\Iso(\cM)$ such that $F^{-1}(\alpha(t))=\alpha(t-\xi)$, since $\cM$ is a Riemannian symmetric space and $\alpha(\cdot)$ is a geodesic in $\cM$. Then 
    \#
    &J(\xi)\notag\\
    = &\int_{\cM}\left(|\phi''(d_g(x,\alpha(\xi)))|+\frac{\phi'(d_g(x,\alpha(\xi)))}{d_g(x,\alpha(\xi))}\cdot (1+\left|\left.\frac{d^2}{dt^2}\right|_{t=\xi}\frac{d_g^2(x,\alpha(t))}{2}\right|)\right)\dvol(x)\notag\\
    \overset{({\rm a})}{=} &\int_{\cM}\left(|\phi''(d_g(F (x),\alpha(\xi)))|+\frac{\phi'(d_g(F (x),\alpha(\xi)))}{d_g(F(x),\alpha(\xi))}\cdot (1+\left|\left.\frac{d^2}{dt^2}\right|_{t=\xi}\frac{d_g^2(F(x),\alpha(t))}{2}\right|)\right)\dvol(x)\notag\\
    \overset{({\rm b})}{=}  &\int_{\cM}\left(|\phi''(d_g(x,\alpha(0)))|+\frac{\phi'(d_g(x,\alpha(0)))}{d_g(x,\alpha(0))}\cdot (1+\left|\left.\frac{d^2}{dt^2}\right|_{t=\xi}\frac{d_g^2(x,\alpha(t-\xi))}{2}\right|)\right)\dvol(x)\notag\\
    \overset{({\rm c})}{=}  &\int_{\cM}\left(|\phi''(d_g(x,\alpha(0)))|+\frac{\phi'(d_g(x,\alpha(0)))}{d_g(x,\alpha(0))}\cdot (1+\left|\left.\frac{d^2}{dt^2}\right|_{t=0}\frac{d_g^2(x,\alpha(t))}{2}\right|)\right)\dvol(x)\notag\\
    \overset{({\rm d})}{=}  &J(0),\label{equ:J-xi-32}
    \#
    where (a) uses the property \eqref{equ:homo} of the integral, (b) uses the fact that 
    \$
    d_g(F(x),\alpha(t))=d_g(x,F^{-1}(\alpha(t)))=d_g(x,\alpha(t-\xi)),\quad\forall t\in\RR,
    \$
    (c) uses the change of variables and the chain rule, and (d) uses the definition of $J(0)$. The relationship \eqref{equ:J-xi-32} demonstrates that  $J(\xi)$ is a quantity independent of $\xi$, and thus one only needs to focus on the quantity $J(0)$. By definition, $J(0)=J_1+J_2$, where
    \$
    J_1&=\int_{\cM}|\phi''(d_g(x,\alpha))|\dvol(x),\\
    J_2&=\int_{\cM}\frac{\phi'(d_g(x,\alpha))}{d_g(x,\alpha)}\cdot \left(1+\left|\left.\frac{d^2}{dt^2}\right|_{t=0}\frac{d_g^2(x,\alpha(t))}{2}\right|\right)\dvol(x).
    \$
    In the following, we will analyze $J_1$ and $J_2$ separately. 
    \begin{itemize}
        \item {\bf Analyzing $J_1$.} By using the polar coordinate expression \eqref{equ:integral} of the integral and the volume comparison theorem, Theorem \ref{thm:2.3}, we obtain that
        \$
        J_1&\leq \vol(\SSS^{m-1})\cdot \int_0^{r_{\cM}}|\phi''(r)|r^{m-1}dr\\
        &<\infty,
        \$
        where the first inequality uses the observation that the sectional curvatures of $\cM$ are nonnegative and the second inequality uses Condition \ref{cond:5.6}.  
        \item {\bf Establishing that $J_2$ is finite using Lie theory.} The second-order derivative 
        \$
        \left|\left.\frac{d^2}{dt^2}\right|_{t=0}\frac{d_g^2(x,\alpha(t))}{2}\right|
        \$ 
        may diverge when $x$ approaches the cut locus $\Cut(\alpha)$ of $\alpha$ in $\cM$, posing a significant challenge to the analysis of $J_2$. To address this challenge, we need to use the Lie theory introduced in Section \ref{sec:apd-a2}. Specifically, we will treat $\alpha$ as the reference point $o$ used in Section \ref{sec:apd-a2}, and we will adopt the same set of notations in Section \ref{sec:apd-a2}. For each $x\in\cM$, we shall use the parametrization $x=\varphi(s,a)$ as in \eqref{equ:a6}, where $s\in S$, $a\in \bar C_+$, and $S$ and $\bar C_+$ are bounded sets defined in Section \ref{sec:a22}. Then we can establish the finiteness of $J_2$ by combining the following steps. 
        \begin{enumerate}[leftmargin=*]
            \item[a)] By the integral formula \eqref{equ:a7}, we can rewrite the integral of any function $\ell(\cdot)$ on $\cM$ in the following way: 
        \$
        \int_{\cM}\ell(x)\dvol(x)=\int_{C_+}\int_S\ell (s,a)\prod_{\lambda\in\Delta_+}(\sin \lambda(a))^{m_\lambda}dad\omega(s),
        \$
        where $\ell(s,a)=\ell\circ\varphi(s,a)$, $\Delta_+$ is the set of certain positive roots introduced in Section \ref{sec:a21}, and $m_\lambda\geq 1$ is the multiplicity of the root $\lambda\in\Delta_+$. This integral formula is essential in proving the integrability of $J_2$.  

            \item[b)] Since $\alpha(\cdot)$ is a geodesic in $\cM$, the second-order derivative 
            \$
            \left.\frac{d^2}{dt^2}\right|_{t=0}\frac{d_g^2(x,\alpha(t))}{2}=H_\alpha(x)(\dot\alpha(0),\dot\alpha(0)),
            \$
            where $H_\alpha(x)$ is the Hessian of the function $h_x(y)=\frac{d_g^2(x,y)}{2}$ at $\alpha$, and $\dot\alpha(0)$ is the differential of $\alpha(\cdot)$ at $t=0$. Then by using the Hessian formula \eqref{equ:a9}  and the fact that $\alpha(\cdot)$ is a unit-speed geodesic, we can show that
            \$
            \left|\left.\frac{d^2}{dt^2}\right|_{t=0}\frac{d_g^2(x,\alpha(t))}{2}\right|\leq 1+\sum_{\lambda\in\Delta_+}\lambda(a)\cot \lambda(a),
            \$
            where $x=\varphi(s,a)$, $\Delta_+$ is the set of certain positive roots introduced in Section \ref{sec:a21}, and $\cot$ is the cotangent function. When $x$ approaches the cut locus $\Cut(\alpha)$ of $\alpha$,  the quantity $\lambda(a)$ for some $\lambda\in\Delta_+$ will converge to $\pi$ and thus $\lambda(a)\cot\lambda(a)$ will diverge. This poses the aforementioned challenge to the analysis of $J_2$. Fortunately, we can address this issue by using the integral formula in part a), which we will discuss in part d).   

            \item[c)] By using the formula of the Riemannian metric \eqref{equ:a4}, we can obtain that 
            \$
            d_g^2(x,\alpha)=\sum_{\lambda\in\Delta_+}m_\lambda(\lambda(a))^2,
            \$
            where $x=\varphi(s,a)$. It implies that $\lambda(a)\leq d_g(x,\alpha)$ for any $\lambda\in\Delta_+$, where $x=\varphi(s,a)$.  

            \item[d)] Finally, by combining the above properties a) and b) and the fact that $\phi'(\cdot)$ is bounded by some constant $L$ within $(0,r_{\cM})$, we obtain that
            \#
            J_2\leq \int_{C_+}\int_S\frac{L}{d_g(x,\alpha)}\left(2+\sum_{\lambda\in\Delta_+}\lambda(a)\cot\lambda(a)\right)\prod_{\lambda\in\Delta_+}(\sin\lambda(a))^{m_\lambda}dad\omega(s).\label{equ:J_2-int}
            \#
            By the above property c), we find that the integrand in \eqref{equ:J_2-int} is bounded. Therefore, since $C_+$ and $S$ are bounded regions, we conclude that $J_2$ is integrable. 
        \end{enumerate}
        
    \end{itemize}
    By combining the above two cases, we obtain that $J(0)=J_1+J_2<\infty$. Then by combining this property with \eqref{equ:d.30}, \eqref{equ:def-J-xi-32}, \eqref{equ:J-xi-32}, we obtain that
    \$
    |I(t)-I(0)|\leq J(0)\cdot\int_0^t(t-\xi)d\xi\leq Ct^2,
    \$
    where $C>0$ is a finite constant independent of $t$. This proves the lemma. 
\end{proof}

\section{Technical proofs for Section \ref{sec:6-temperature}}

This section provides proofs for the results in Section \ref{sec:6-temperature}.

\subsection{Proof of Proposition \ref{prop:6.1}}

\begin{proof}
    Given $n$ independent samples $\{x_i\}_{i=1}^n$ drawn from the distribution $f(x;\alpha,\beta,\phi)$, the MLEs for $\alpha$ and $\beta$ are given by 
    \$
    (\hat\alpha^{\MLE},\hat\beta^{\MLE})=\argmax_{\alpha,\beta}\frac{1}{Z(\beta,\phi)^n}e^{-\beta \sum_{i=1}^n\phi(d_g(x_i,\alpha))},
    \$
    where the maximum of $\alpha$ is taken over $\cX$ since the true parameter is in $\cX$.
    For any $\beta>0$, the MLE of $\alpha$ is given by
    \#\label{equ:e1}
    \hat\alpha^{\MLE}=\argmin_{\alpha\in\cX}\sum_{i=1}^n\phi(d_g(x_i,\alpha)),
    \#
    which is independent of $\beta$. Therefore, \eqref{equ:e1} gives the MLE of $\alpha$ when $\beta$ is unknown. Moreover, the MLE of $\beta$ is given by
    \$
    \hat\beta=\argmin_{\beta}\beta\sum_{i=1}^n\phi(d_g(x_i,\hat\alpha^{\MLE}))+n\log Z(\beta,\phi),
    \$
    where $\hat \alpha$ is given by \eqref{equ:e1}. 
\end{proof}

\subsection{Proof of Theorem \ref{thm:6.2}}

    

\begin{proof}
    We will prove this theorem in three steps. First, we construct an $\epsilon$-net $\cS_{\alpha}=\{\alpha_i\}_{i=1}^{N_{\alpha}}$ of the set $\cX=\cB_{\cM}(\alpha^*,D)$ and an $\epsilon$-net $\cS_{\beta}=\{\beta_i\}_{i=1}^{N_{\beta}}$ of the bounded interval $[\beta_{\min},\beta_{\max}]$. By Lemma \ref{lma:c1}, the $\epsilon$-net $\cS_{\alpha}$ can be chosen such that $N_{\alpha}\lesssim\epsilon^{-m}$, and by conventional knowledge, the $\epsilon$-net $\cS_{\beta}$ can be chosen such that $N_{\beta}\lesssim\epsilon^{-1}$, where $\lesssim$ omits constants independent of $\epsilon$.

    Next, we use $\cS_{\alpha}$ and $\cS_{\beta}$ to construct the following set
    \#\label{equ:S_F-sec-6}
    \cS_{\cF}=\{f(x;\alpha_i,\beta_j,\phi)\mid\alpha_i\in\cS_{\alpha},\beta_j\in\cS_{\beta}\}.
    \#
    The cardinality of the set $\cS_{\cF}$ is controlled by 
    \$
    |\cS_{\cF}|\leq|\cS_{\alpha}|\cdot|\cS_{\beta}|\lesssim \epsilon^{-(m+1)},
    \$
    where $\lesssim$ omits constants independent of $\epsilon$. For any $\alpha\in\cX$ and $\beta\in[\beta_{\min},\beta_{\max}]$, there are $\alpha_i\in\cS_{\alpha}$ and $\beta_{j}\in\cS_{\beta}$ such that $d_g(\alpha_i,\alpha)\leq\epsilon$ and $|\beta-\beta_j|\leq \epsilon$. 
    By Lemma \ref{lma:e1} and \ref{lma:e2}, we have
    \$
    d_{\infty}(f(x;\alpha,\beta,\phi),f(x;\alpha_i,\beta,\phi))&\leq C_1d_g(\alpha,\alpha_i)\leq C_1\epsilon,\\
    d_{\infty}(f(x;\alpha_i,\beta,\phi),f(x;\alpha_i,\beta_j,\phi))&\leq C_2|\beta-\beta_j|\leq C_2\epsilon, 
    \$
    where $C_1,C_2$ are finite constants independent of $\alpha,\alpha_i,\beta,\beta_i$, and $\epsilon$. Then by the triangle inequality,  we have
    \$
    &d_{\infty}(f(x;\alpha,\beta,\phi),f(x;\alpha_i,\beta_j,\phi))\\
    \leq\ &d_{\infty}(f(x;\alpha,\beta,\phi),f(x;\alpha_i,\beta,\phi))+d_{\infty}(f(x;\alpha_i,\beta,\phi),f(x;\alpha_i,\beta_j,\phi))\\
    \leq\ &C\epsilon.
    \$
    where $C$ is a constant independent of $\alpha,\alpha_i,\beta,\beta_i$, and $\epsilon$. This demonstrates that $\cS_{\cF}$ is a $C\epsilon$-net of $\cF$ and thus
    \$
    \cN(C\epsilon,\cF,d_{\infty})\leq|\cS_{\cF}|\lesssim \epsilon^{-(m+1)}.
    \$
    where $\lesssim$ omits constants independent of $\epsilon$. By rescaling $\epsilon$, we obtain the following entropy estimate:
    \$
    \log\cN(\epsilon,\cF,d_{\infty})\lesssim \log(\frac{1}{\epsilon}), 
    \$
    where $\lesssim$ omits constants independent of $\epsilon$. This proves the first entropy inequality in the theorem. 

    To proceed, we analyze the bracketing entropy $\log\cN_B(\epsilon,\cF,d_1)$ and $\log\cN_B(\epsilon,\cF,d_h)$.
    Observe that for any $f(x;\alpha,\beta,\phi)\in\cF$, it holds that
    \#\label{equ:upper-bound-1-e}
    0\leq f(x;\alpha,\beta,\phi)\leq \frac{1}{Z(\beta_{\max},\phi)},\quad\forall x\in\cM,
    \#
    where we use the facts that $\phi$ is nonnegative and
    \$
    Z(\beta,\phi)\geq Z(\beta_{\max},\phi),\quad\forall \beta\in[\beta_{\min},\beta_{\max}].
    \$
    Additionally, if $d_g(x,\alpha^*)\geq 2D$, then for all $\alpha\in\cB_{\cM}(\alpha^*,D)$, it holds that
    \$
    d_g(x,\alpha)\geq d_g(x,\alpha^*)-d_g(\alpha,\alpha^*)\geq \frac{d_g(x,\alpha^*)}{2}.
    \$
    This implies that for all $x$ with $d_g(x,\alpha^*)\geq 2D$ and $\beta\in[\beta_{\min},\beta_{\max}]$, 
    \#\label{equ:upper-bound-2}
    f(x;\alpha,\beta,\phi)\leq \frac{1}{Z(\beta_{\max},\phi)}e^{-\beta_{\min}\phi(d_g(x,\alpha^*)/2)},
    \#
    where we use the fact that $\phi$ is increasing and nonnegative. Combining \eqref{equ:upper-bound-1-e} with \eqref{equ:upper-bound-2}, we conclude that
    \$
    H(x)=\left\{
    \begin{array}{ll}
        \frac{1}{Z(\beta_{\max},\phi)}e^{-\beta_{\min}\phi(d_g(x,\alpha^*)/2)}, & \textnormal{if }d_g(x,\alpha^*)>2D, \\
        \frac{1}{Z(\beta_{\max},\phi)}, & \textnormal{otherwise},
    \end{array}\right.
    \$
    is an envelop for $\cF$, that is, $f(x;\alpha,\beta,\phi)\leq H(x)$ for all $x\in\cM$, $\alpha\in\cX$, and $\beta\in[\beta_{\min},\beta_{\max}]$. Then we construct brackets $[l_i,u_i]$ as follows:
    \$
    l_i=\max\{f_i-\eta,0\},\quad u_i=\min\{f_i+\eta,H\},
    \$
    where $\{f_i\}_{i=1}^N$ is an $\eta$-net of $\cF$ under $d_\infty$. It is obvious that $\cF\subseteq\cup_{i=1}^N[l_i,u_i]$ and 
    \$
    u_i-l_i\leq \min\{2\eta,H\}. 
    \$
    Therefore, for any $B>0$, 
    \#\label{equ:upperbound-3}
    d_1(u_i,l_i)\coloneqq \int_{\cM}(u_i-l_i)\dvol\leq 2\eta\cdot\vol(\cB_{\cM}(\alpha^*,B))+\int_{d_g(x,\alpha^*)>B}H(x)\dvol(x).
    \#
    To provide an upper bound on this quantity, we consider the following two cases separately.
    \begin{itemize}
        \item {\bf Case 1: $\cM$ is compact.} In this case, by taking $B>\sup_{x\in\cM}d_g(x,\alpha^*)$, we obtain that
        \$
        d_1(u_i,l_i)\leq 2\eta\cdot\vol(\cM).
        \$
        This implies that  for sufficiently small $\eta$,
        \$
        \cN_{B}(2\eta\cdot\vol(\cM),\cF,d_1)\leq \cN(\eta,\cF,d_{\infty})\lesssim \eta^{-(m+1)},
        \$
        where we use the entropy estimate of $\cF$ under $d_{\infty}$. By taking $\epsilon=2\eta\cdot\vol(\cM)$, we obtain that for sufficiently small $\epsilon$,
        \$
        \log\cN_{B}(\epsilon,\cF,d_1)\lesssim\log(\frac{1}{\epsilon}),
        \$
        where $\lesssim$ omits constants independent of $\epsilon$. Since  $\cN_B(\sqrt{\epsilon},\cF,d_h)\leq\cN_B(\epsilon,\cF,d_1)$, we have 
        \$
        \log\cN_B(\epsilon,\cF,d_h)\lesssim \log(\frac{1}{\epsilon})
        \$ 
        for sufficiently small $\epsilon$, where $\lesssim$ omits constants independent of $\epsilon$. 
        
        \item {\bf Case 2: $\cM$ is noncompact.} In this case, we take $B>2D$. The upper bound in \eqref{equ:upperbound-3} can be rewritten as
        \#\label{equ:e6}
         d_1(u_i,l_i) 
        \leq 2\eta\cdot \vol(\cB_{\cM}(\alpha^*,B))+\frac{1}{Z(\beta_{\max},\phi)}\int_{d_g(x,\alpha^*)>B}e^{-\beta_{\min}\phi(d_g(x,\alpha^*)/2)}\dvol(x).
        \#
        Using the polar coordinate expression \eqref{equ:integral} of the integral, the volume comparison theorem, Theorem \ref{thm:2.3}, and the same argument in \eqref{equ:4.9}, we obtain that
        \#\label{equ:e7}
        \vol(\cB_{\cM}(\alpha^*,B))&\leq C_1e^{\sqrt{-\kappa_{\min}}(m-1)B},
        \#
        where $C_1$ is a constant independent of $B$, $\kappa_{\min}<0$ is a lower bound on the sectional curvatures of $\cM$, and $m$ is the dimension of $\cM$. Similarly, by using the polar coordinate expression \eqref{equ:integral} and Theorem \ref{thm:2.3}, we have 
        \$
        \int_{d_g(x,\alpha^*)>B}e^{-\beta_{\min}\phi(d_g(x,\alpha^*)/2)}\dvol(x)  
        \leq \int_B^\infty e^{-\beta_{\min}\phi(r/2)}\sn_{\kappa_{\min}}^{m-1}(r)dr\cdot\vol(\SSS^{m-1}),
        \$
        where  $\sn_{\kappa_{\min}}(\cdot)$ is given by \eqref{equ:sn}. Combining this with \eqref{equ:e6} and \eqref{equ:e7}, we obtain that
        \$
        d_1(u_i,l_i)\leq 2C_1\eta e^{\sqrt{-\kappa_{\min}}(m-1)B}+\frac{\vol(\SSS^{m-1})}{Z(\beta_{\max},\phi)}\int_B^{\infty}e^{-\beta_{\min}\phi(r/2)}\sn_{\kappa_{\min}}^{m-1}(r)dr.
        \$
        Taking $B=B_{\eta}=\frac{\log(1/\eta)}{2(m-1)\sqrt{-\kappa_{\min}}}$, and using Condition \ref{cond:6.3}, we conclude that for all sufficiently small $\eta$, 
        \$
        d_1(u_i,l_i)\leq 2C_1\eta^{1/2}+\frac{c_1\vol(\SSS^{m-1})}{Z(\beta_{\max},\phi)}\eta^{c_2}\leq C\eta^c,
        \$
        where $c_1,c_2,C,c>0$ are constants independent of $\eta$. This implies that for sufficiently small $\eta$, 
        \$
        \cN_{B}(C\eta^c,\cF,d_1)\leq \cN(\eta,\cF,d_\infty)\lesssim\eta^{-(m+1)},
        \$
        where $\lesssim$ omits constants independent of $\eta$. By taking $\epsilon=C\eta^c$, we obtain that for sufficiently small $\epsilon$,
            \$
        \log\cN_B(\epsilon,\cF,d_1)\lesssim \log(\frac{1}{\epsilon}),
        \$
        where $\lesssim$ omits constants independent of $\epsilon$.
        Again, since $\cN_B(\sqrt{\epsilon},\cF,d_h)\leq\cN_B(\epsilon,\cF,d_1)$, we have 
        \$
        \log\cN_B(\epsilon,\cF,d_h)\lesssim \log(\frac{1}{\epsilon})
        \$ 
        for sufficiently small $\epsilon$, where $\lesssim$ omits constants independent of $\epsilon$. 
    \end{itemize}
    The proof of the theorem is complete by combining the above two cases. 
\end{proof}

\subsection{Proof of Corollary \ref{corollary:6.4}}

\begin{proof}
    This follows from empirical process theory and the bracketing entropy estimates in Theorem \ref{thm:6.2}. Specifically, by the bracketing entropy estimates in Theorem \ref{thm:6.2}, the bracketing entropy integral satisfies that
    \$
    J_B(\delta,\cF,d_h)\coloneqq\int_0^\delta \sqrt{\log\cN_{B}(u,\cF,d_h)}du\lesssim\int_0^\delta \log^{1/2}(\frac{1}{u})du,
    \$
    for sufficiently small $\delta$, where $\lesssim$ omits constants independent of $\delta$. By Theorem 2,  \citet{wong1995probability}, we obtain that with probability at least $1-ce^{-c\log^2n}$, 
    \$
    d_h(f(x;\alpha^{\tr},\beta^{\tr},\phi),f(x;\hat\alpha^{\MLE},\hat\beta^{\MLE},\phi))\lesssim\frac{\log n}{\sqrt{n}},
    \$
    where $\lesssim$ omits constants independent of $n$ and $c$ is a universal constant. This proves \eqref{equ:6.3}.  Since the $L^1$ distance is upper bounded by twice the hellinger distance, the inequality \eqref{equ:6.4}  follows from the inequality \eqref{equ:6.3}. This proves the corollary.  
\end{proof}


    

    

\subsection{Proof of Theorem \ref{thm:6.4}}

\begin{proof}
    By Lemma \ref{lma:e6}, we have
    \#\label{equ:6-lemma}
    \lim_{\epsilon\to0}\inf\left\{\left.\frac{D(\alpha,\beta;\alpha_0,\beta_0)}{d_g(\alpha,\alpha_0)+|\beta-\beta_0|}\ \right|\  d_g(\alpha,\alpha_0)<\epsilon,|\beta-\beta_0|<\epsilon\right\}>0,
    \#
    where 
    \$
    D(\alpha,\beta;\alpha_0,\beta_0)=d_1(f(x;\alpha,\beta,\phi),f(x;\alpha_0,\beta_0,\phi)).
    \$
    Therefore, there exists a positive constant $\epsilon_0$ such that for all $\alpha,\beta$ with $d_g(\alpha,\alpha_0)<\epsilon_0$ and $|\beta-\beta_0|<\epsilon_0$, the following inequality holds:
    \#\label{equ:6-property}
    d_g(\alpha,\alpha_0)+|\beta-\beta_0|\leq C_1\cdot D(\alpha,\beta;\alpha_0,\beta_0),
    \#
    where $C_1>0$ is a constant independent of $\alpha$ and $\beta$. Then to prove the theorem, we show that
    \#\label{equ:6-assumption}
    \inf_{\alpha,\beta\in\cA}D(\alpha,\alpha_0)>0,
    \#
    where $\cA$ is a compact set given by
    \$
    \cA=\left\{(\alpha,\beta)\in\cX\times [\beta_{\min},\beta_{\max}]\mid d_g(\alpha,\alpha_0)\geq \epsilon_0\textnormal{ or }|\beta-\beta_0|\geq \epsilon_0\right\}.
    \$
    Suppose on the contrary that 
    \$
    \inf_{\alpha,\beta\in\cA}D(\alpha,\alpha_0)=0.
    \$
    Then there exists a sequence $\{\alpha_n,\beta_n\}_{n=1}^{\infty}\subseteq\cA$ such that
    \$
    D(\alpha_n,\beta_n;\alpha_0,\beta_0)\to 0.
    \$
    Since $\cA$ is a compact set, we can apply Lemma \ref{lma:e1} to obtain that $d_g(\alpha_n,\alpha_0)\to0$ and $|\beta_n-\beta_0|\to0$. This contradicts with the choices of $\alpha_n$ and $\beta_n$, and thus we prove the result \eqref{equ:6-assumption}. Combining this with \eqref{equ:6-property}, we immediately obtain the theorem.  
\end{proof}

    

    

    
    

\subsection{Technical lemmas}

This section collects technical lemmas for this section. Section \ref{sec:e51} derives the Lipschitz properties of the distribution $f(x;\alpha,\beta,\phi)$ with respect to its parameters. Section \ref{sec:e52} presents the parameter identifiability of $\alpha,\beta$ in the distribution $f(x;\alpha,\beta,\phi)$. Section \ref{sec:e53} proves a claim used in the proof of Theorem \ref{thm:6.4}. Section \ref{sec:e54} present technical lemmas used in Section \ref{sec:e53}.

\subsubsection{Lipschitz properties}\label{sec:e51}

Lemma \ref{lma:e1} establishes the Lipschitz property of the distribution $f(x;\alpha,\beta,\phi)$ with respect to $\alpha$. 

\begin{lemma}\label{lma:e1}
    Assume $(\cM,g^{\cM})$ is a Riemannian homogeneous space, and $\phi$ is a function satisfying Conditions  \ref{cond:6.3} for some  $\beta_{\min}>0$. Let $f(x;\alpha,\beta,\phi)$ be the distribution defined in \eqref{equ:function-6.2} and $\beta_{\min}\leq\beta_{\max}<\infty$. Then for any $\beta\in[\beta_{\min},\beta_{\max}]$, the following inequality holds for any pair $\alpha_1,\alpha_2\in\cM$:
    \$
    d_{\infty}(f(x;\alpha_1,\beta,\phi),f(x;\alpha_2,\beta,\phi))\leq C\cdot d_g(\alpha_1,\alpha_2),
    \$
    where $C>0$ is a finite constant independent of $\alpha_1,\alpha_2$ and $\beta$.
\end{lemma}

\begin{proof}
    Condition \ref{cond:6.3} states that the function $|\phi'(r)e^{-\beta_{\min}\phi(r)}|$ is bounded on $[0,r_{\cM}]$ by a constant $L$, where $r_{\cM}=\sup_{x,y\in\cM}d_g(x,y)$ is the maximum radius of $\cM$. Then for any $\beta\in[\beta_{\min},\beta_{\max}]$, the first-order derivative of $e^{-\beta\phi(r)}$ satisfies the following bound:
    \$
    \left|\frac{d}{dr}e^{-\beta\phi(r)}\right|=|\beta\phi'(r)e^{-\beta\phi(r)}|&\leq |\beta_{\max}\phi'(r)e^{-\beta_{\min}\phi(r)}|\\
    &\leq \beta_{\max}L.
    \$
    Thus, the function $e^{-\beta\phi(r)}$ is a Lipschitz continuous function with a Lipschitz constant $L_0=\beta_{\max}L$ for any $\beta\in[\beta_{\min},\beta_{\max}]$. As a result, we have
    \$
    |f(x;\alpha_1,\beta,\phi)-f(x;\alpha_2,\beta,\phi)|&\leq \frac{L_0}{Z(\beta,\phi)}\cdot|d_g(x,\alpha_1)-d_g(x,\alpha_2)|\\
    &\leq \frac{L_0}{Z(\beta_{\max},\phi)}\cdot|d_g(x,\alpha_1)-d_g(x,\alpha_2)|,
    \$
    where $Z(\beta,\phi)$ is the normalizing constant in \eqref{equ:function-6.2} and the second inequality follows from the following observation 
    \$
    Z(\beta,\phi)\geq Z(\beta_{\max},\phi),\quad\forall \beta\leq\beta_{\max}.
    \$ 
    By the triangular inequality, we have 
    \$
    |d_g(x,\alpha_1)-d_g(x,\alpha_2)|\leq d_g(\alpha_1,\alpha_2).
    \$
    This implies that for any $\beta\in[\beta_{\min},\beta_{\max}]$,
    \$
    |f(x;\alpha_1,\beta,\phi)-f(x;\alpha_2,\beta,\phi)|\leq \frac{L_0}{Z(\beta_{\max},\phi)}\cdot d_g(\alpha_1,\alpha_2).
    \$
    By taking the supremum over $x$, we conclude the proof of this lemma.
\end{proof}

Lemma \ref{lma:e2} establishes the Lipschitz property of the distribution $f(x;\alpha,\beta,\phi)$ with respect to $\beta$. 

\begin{lemma}\label{lma:e2}
    Assume $(\cM,g^{\cM})$ is a Riemannian homogeneous space, and $\phi$ is a function satisfying Conditions  \ref{cond:6.3} for some  $\beta_{\min}>0$. Let $f(x;\alpha,\beta,\phi)$ be the distribution defined in \eqref{equ:function-6.2} and $\beta_{\min}\leq\beta_{\max}<\infty$. Then for any $\alpha\in\cM$, the following inequality holds for any pair $\beta_1,\beta_2\in[\beta_{\min},\beta_{\max}]$:
    \$
    d_{\infty}(f(x;\alpha,\beta_1,\phi),f(x;\alpha,\beta_2,\phi))\leq C\cdot |\beta_1-\beta_2|,
    \$
    where $C>0$ is a finite constant independent of $\alpha,\beta_1,\beta_2$.
\end{lemma}

\begin{proof}
    By the definition of $f(x;\alpha,\beta,\phi)$, we can show that for all $\beta_1,\beta_2\in[\beta_{\min},\beta_{\max}]$,
    \#
    &|f(x;\alpha,\beta_1,\phi)-f(x;\alpha,\beta_2,\phi)|\notag\\
    =\ &\left|\frac{1}{Z(\beta_1,\phi)}e^{-\beta_1\phi(d_g(x,\alpha))}-\frac{1}{Z(\beta_2,\phi)}e^{-\beta_2\phi(d_g(x,\alpha))}\right|\notag\\
    \leq \ &\left|\frac{1}{Z(\beta_1,\phi)}-\frac{1}{Z(\beta_2,\phi)}\right|\cdot e^{-\beta_2\phi(d_g(x,\alpha))}+\frac{1}{Z(\beta_1,\phi)}\cdot \left|e^{-\beta_1\phi(d_g(x,\alpha))}- e^{-\beta_2\phi(d_g(x,\alpha))}\right|\notag\\
    \leq \ &\frac{1}{Z(\beta_{\max},\phi)^2}\cdot|Z(\beta_2,\phi)-Z(\beta_1,\phi)|+\frac{1}{Z(\beta_{\max},\phi)}\cdot \left|e^{-\beta_1\phi(d_g(x,\alpha))}- e^{-\beta_2\phi(d_g(x,\alpha))}\right|,\label{equ:lip-f-beta}
    \#
    where the first inequality uses the triangle inequality and the second inequality uses the facts that $\phi$ is a nonnegative function and
    \$
    Z(\beta,\phi)\geq Z(\beta_{\max},\phi),\quad\forall \beta\leq\beta_{\max}.
    \$
    Observe that the function $\ell(\beta)=e^{-\beta r}$ for any $r\geq 0$ is a Lipschitz function over $[\beta_{\min},\beta_{\max}]$ with a Lipschitz constant 
    \#\label{equ:L0-e3}
    L_0=\sup_{r\geq 0}re^{-\beta_{\min}r}<\infty,
    \#
    which is independent of $r$. Then the following inequality holds
    \#\label{equ:lipschitz-e4}
    \left|e^{-\beta_1\phi(d_g(x,\alpha))}- e^{-\beta_2\phi(d_g(x,\alpha))}\right|\leq L_0|\beta_1-\beta_2|,
    \#
    where $L_0$ is defined by \eqref{equ:L0-e3}. To proceed, we analyze the following term
    \#\label{equ:e12}
    |Z(\beta_1,\phi)-Z(\beta_2,\phi)|\leq \int_{\cM}|e^{-\beta_1 \phi(d_g(x,\alpha))}-e^{-\beta_2\phi(d_g(x,\alpha))}|\dvol(x).
    \#
    Using the Lipschitz property of the function $\ell(\beta)=e^{-\beta r}$ over $[\beta_{\min},\beta_{\max}]$, whose Lipschitz constant is $re^{-\beta_{\min}r}$, we obtain that
    \#
     &\int_{\cM}|e^{-\beta_1 \phi(d_g(x,\alpha))}-e^{-\beta_2\phi(d_g(x,\alpha))}|\dvol(x)\notag\\
     \leq \ & \int_{\cM}\phi(d_g(x,\alpha))\cdot e^{-\beta_{\min} \phi(d_g(x,\alpha))}\cdot\dvol(x)\cdot |\beta_1-\beta_2|.\label{equ:e13}
    \#
    Then we consider two cases. 
    \begin{itemize} 
        \item {\bf Case 1: $\cM$ is compact.} In this case, the integral $\int_{\cM}\phi(d_g(x,\alpha))\cdot e^{-\beta_{\min} \phi(d_g(x,\alpha))}\dvol(x)$ is finite and independent of $\alpha$. Thus, by combining this with \eqref{equ:e12} and \eqref{equ:e13}, we obtain that 
        \$
        |Z(\beta_1,\phi)-Z(\beta_2,\phi)|\leq L_1|\beta_1-\beta_2|,
        \$
        where $L_1$ is a finite constant independent of $\alpha,\beta_1,\beta_2$.  
        
        \item {\bf Case 2: $\cM$ is noncompact.} Using the polar coordinate expression \eqref{equ:integral} of the integral and the volume comparison theorem, Theorem \ref{thm:2.3}, we obtain that
    \$
    &\int_{\cM}\phi(d_g(x,\alpha))e^{-\beta_{\min}\phi(d_g(x,\alpha))}\dvol(x)\\
    \leq  &\int_{0}^{\infty}\phi(r)e^{-\beta_{\min}\phi(r)}\sn_{\kappa_{\min}}^{m-1}(r)dr\cdot\vol(\SSS^{m-1})\\
    <\ &\infty, 
    \$
    where $\kappa_{\min}$ is a lower bound on the sectional curvatures of $\cM$, $\sn_{\kappa_{\min}}(\cdot)$ is given by \eqref{equ:sn}, and the second inequality uses Condition \ref{cond:6.3}. It follows that
    \$
        |Z(\beta_1,\phi)-Z(\beta_2,\phi)|\leq L_1|\beta_1-\beta_2|,
    \$
    where $L_1$ is a finite constant independent of $\alpha,\beta_1,\beta_2$. 
    \end{itemize}
    Combining these two cases with \eqref{equ:lip-f-beta} and \eqref{equ:lipschitz-e4}, we obtain that
    \$
    |f(x;\alpha,\beta_1,\phi)-f(x;\alpha,\beta_2,\phi)|\leq C\cdot|\beta_1-\beta_2|,
    \$
    where $C>0$ is a finite constant independent of $\alpha,\beta_1,\beta_2$. Then we conclude the proof by taking the supremum of $x$. 
\end{proof}

\subsubsection{Identifiability}\label{sec:e52}

\begin{lemma}
    Suppose $(\cM,g^{\cM})$ is a Riemannian homogeneous space with $r_{\cM}=\sup_{x,y}d_g(x,y)$. Assume $\phi$ satisfies Condition \ref{cond:6.3} for  $\beta_{\min}>0$ and is strictly increasing on $[0,r_{\cM}]$. Let $f(x;\alpha,\beta,\phi)$ be the density in \eqref{equ:6.1-f} and $\cX\subseteq\cM$ a compact set. Then if $\alpha_0,\{\alpha_n\}_{n=1}^\infty\subseteq\cX$ and $\beta_0,\{\beta_n\}_{n=1}^\infty\subseteq[\beta_{\min},\beta_{\max}]$ for some $\beta_{\min}\leq \beta_{\max}<\infty$ satisfy that
    \#\label{equ:e.13}
    d_1(f(x;\alpha_n,\beta_n,\phi),f(x;\alpha_0,\beta_0,\phi))\to 0,
    \#
    where $d_1$ is the $L^1$ distance, then $d_g(\alpha_n,\alpha_0)\to0$ and $|\beta_n-\beta_0|\to0$ as $n\to\infty$.
\end{lemma}

\begin{proof}
    We prove this lemma by contradiction. Suppose Condition \ref{equ:e.13} holds but $(\alpha_n,\beta_n)$ does not converge to $(\alpha_0,\beta_0)$. Then there exists a subsequence $\{(\alpha_{n_i},\beta_{n_i})\}_{i=1}^\infty$ that converges to $(\alpha_0',\beta_0')\neq (\alpha_0,\beta_0)$, where we use the fact that $\{\alpha_n\}\subseteq\cX$, $\{\beta_n\}\subseteq[\beta_{\min},\beta_{\max}]$, and $\cX$ is a compact set. By the continuity of the function $f(x;\alpha,\beta,\phi)$ with respect to its parameters, i.e., Lemma \ref{lma:e4}, we have that
    \$
    d_1(f(x;\alpha_{n_i},\beta_{n_i},\phi),f(x;\alpha'_0,\beta'_0,\phi))\to 0.
    \$
    Combining this with \eqref{equ:e.13}, we obtain that $d_1(f(x;\alpha_0,\beta_0,\phi),f(x;\alpha_0',\beta_0',\phi))=0$. It then follows by Lemma \ref{lma:e5} that $\alpha_0=\alpha'_0$ and $\beta_0=\beta_0'$, which leads to contradiction. 
\end{proof}

Lemma \ref{lma:e4} proves the continuity of $f(x;\alpha,\beta,\phi)$ with respect to its parameters $\alpha$ and $\beta$ under certain conditions, where we use the $L^1$ distance to measure the distance between distributions. 

\begin{lemma}\label{lma:e4}
    Let $(\cM,g^{\cM})$ be a Riemannian homogeneous space and $\phi$ a function satisfying Condition \ref{cond:6.3} for  $\beta_{\min}>0$. Let $f(x;\alpha,\beta,\phi)$ be the density in \eqref{equ:6.1-f}. If $d_g(\alpha_n,\alpha)\to0$ and $|\beta_n-\beta|\to 0$ as $n\to\infty$, and $\{\beta_n\},\beta\geq \beta_{\min}$, then the following convergence holds:
    \$
    d_1(f(x;\alpha_n,\beta_n,\phi),f(x;\alpha,\beta,\phi))\to 0,
    \$
    where $d_1$ is the $L^1$ distance. 
\end{lemma}

\begin{proof}
    Using Lemma \ref{lma:c5}, we have
    \$
    d_1(f(x;\alpha_n,\beta,\phi),f(x;\alpha,\beta,\phi))\to 0.
    \$
    Then to prove the lemma, it remains to show that
    \#\label{equ:ed1_convergence-1}
    d_1(f(x;\alpha_n,\beta_n,\phi),f(x;\alpha_n,\beta,\phi))\to 0.
    \#
    By definition of $f$, we have
    \$
    &d_1(f(x;\alpha_n,\beta_n,\phi),f(x;\alpha_n,\beta,\phi))\\
    =&\int_{\cM}\left|\frac{1}{Z(\beta_{n},\phi)}e^{-\beta_n\phi(d_g(x,\alpha_n))}-\frac{1}{Z(\beta,\phi)}e^{-\beta\phi(d_g(x,\alpha_n))}\right|\dvol(x).
    \$
    Since $\cM$ is a homogeneous space, there exists an isometry $F$ of $\cM$ such that $F(\alpha)=\alpha_n$. By using this isometry $F$ and the property \eqref{equ:homo} of the integral, we can show that 
    \$
    d_1(f(x;\alpha_n,\beta_n,\phi),f(x;\alpha_n,\beta,\phi))=d_1(f(x;\alpha,\beta_n,\phi),f(x;\alpha,\beta,\phi)).
    \$
    Thus, \eqref{equ:ed1_convergence-1} is reduced to
    \#\label{equ:ed1_convergence-2}
    d_1(f(x;\alpha,\beta_n,\phi),f(x;\alpha,\beta,\phi))\to 0. 
    \#
    To prove this, we observe that $\{\beta_n\}\subseteq[\beta_{\min}, \beta_{\max}]$ for some $\beta_{\max}>\beta$, and 
    \$
    \left|\frac{1}{Z(\beta_{n},\phi)}e^{-\beta_n\phi(d_g(x,\alpha))}-\frac{1}{Z(\beta,\phi)}e^{-\beta\phi(d_g(x,\alpha))}\right|\leq H(x),
    \$
    where 
    \$
    H(x)=\frac{2}{Z(\beta_{\max},\phi)}e^{-\beta_{\min}\phi(d_g(x,\alpha))}.
    \$
    By Condition \ref{cond:6.3}, we know that $H(x)$ is integrable. Thus, by applying the dominated convergence theorem, we have
    \#\label{equ:e-16}
    \lim_{n\to\infty}\int_{\cM}\left|\frac{1}{Z(\beta_{n},\phi)}e^{-\beta_n\phi(d_g(x,\alpha))}-\frac{1}{Z(\beta,\phi)}e^{-\beta\phi(d_g(x,\alpha))}\right|\dvol(x)= 0,
    \#
    where we use the observation that the integrand in \eqref{equ:e-16} converges to zero pointwise as $n$ tends to infinity. It implies that the convergence \eqref{equ:ed1_convergence-2} holds, which concludes the proof.
\end{proof}

Lemma \ref{lma:e5} shows that if the $L^1$ distance between $f(x;\alpha,\beta,\phi)$ and $f(x;\alpha',\beta',\phi)$ is zero, then $\alpha=\alpha'$ and $\beta=\beta'$ under mild conditions.

\begin{lemma}\label{lma:e5}
        Suppose $(\cM,g^{\cM})$ is a Riemannian homogeneous space with $r_{\cM}=\sup_{x,y}d_g(x,y)$. Assume $\phi$ satisfies Condition \ref{cond:6.1.phi} for some $\beta_{\min}>0$ and is strictly increasing on $[0,r_{\cM}]$. Then for any $\beta,\beta'\geq\beta_{\min}$ and $\alpha,\alpha'\in\cM$, the following equality 
        \#\label{equ:equal-e14}
        d_1(f(x;\alpha,\beta,\phi),f(x;\alpha',\beta',\phi))=0,
        \#
        implies that $\alpha=\alpha'$ and $\beta=\beta'$.
\end{lemma}

\begin{proof}
    By the continuity of the functions $f(x;\alpha,\beta,\phi)$ and $f(x;\alpha',\beta',\phi)$, we know  that Condition \eqref{equ:equal-e14} implies that
    \$
    f(x;\alpha,\beta,\phi)=f(x;\alpha',\beta',\phi),\quad\forall x\in\cM,
    \$
    By the definition of $f$, we obtain that
    \#\label{equ:function-equal-e15}
    \frac{1}{Z(\beta,\phi)}e^{-\beta\phi(d_g(x,\alpha))}=\frac{1}{Z(\beta',\phi)}e^{-\beta'\phi(d_g(x,\alpha'))},\quad \forall x\in\cM,
    \#
    where $Z(\beta,\phi)$ is the normalizing constant of $f(x;\alpha,\beta,\phi)$. 
    By taking $x=\alpha$ and $x=\alpha'$, we have
    \$
    \frac{1}{Z(\beta,\phi)}e^{-\beta\phi(0)}
    &=\frac{1}{Z(\beta',\phi)}e^{-\beta'\phi(d_g(\alpha,\alpha'))}\\
    &\overset{({\rm i})}{\leq} \frac{1}{Z(\beta',\phi)}e^{-\beta'\phi(0)}\\
    &=\frac{1}{Z(\beta,\phi)}e^{-\beta\phi(d_g(\alpha',\alpha))}\\
    &\overset{({\rm ii})}{\leq} \frac{1}{Z(\beta,\phi)}e^{-\beta\phi(0)},
    \$
    where the inequalities (i) and (ii) use the fact that $\phi$ is an increasing function. The above relationships imply that (i) and (ii) are in fact equalities. Since $\phi$ is a strictly increasing function, it follows that $d_g(\alpha,\alpha')=0$ and thus $\alpha=\alpha'$. Substituting this into \eqref{equ:function-equal-e15}, we obtain that
    \$
    \frac{1}{Z(\beta,\phi)}e^{-\beta\phi(d_g(x,\alpha))}=\frac{1}{Z(\beta',\phi)}e^{-\beta'\phi(d_g(x,\alpha))},\quad \forall x\in\cM,
    \$
    which implies  $\beta=\beta'$. 
\end{proof}

\subsubsection{Proof of the claim \eqref{equ:6-lemma}}\label{sec:e53}

Lemma \ref{lma:e6} proves the claim \eqref{equ:6-lemma} used in the proof of Theorem \ref{thm:6.4}. 

\begin{lemma}\label{lma:e6}
    Assume $(\cM,g^{\cM})$ is a Riemannian homogeneous space  and $\phi$ satisfies Condition \ref{cond:6.7} for  $\beta_{\min}>0$. Let $f(x;\alpha,\beta,\phi)$ be the density in \eqref{equ:6.1-f} and $(\alpha_0,\beta_0)\in\cX\times[\beta_{\min},\beta_{\max}]$ a fixed pair of parameters. Then
    \$
    \lim_{\epsilon\to0}\inf\left\{\left.\frac{D(\alpha,\beta;\alpha_0,\beta_0)}{d_g(\alpha,\alpha_0)+|\beta-\beta_0|}\ \right|\  d_g(\alpha,\alpha_0)<\epsilon,|\beta-\beta_0|<\epsilon,(\alpha,\beta)\in\cX\times[\beta_{\min},\beta_{\max}]\right\}>0,
    \$
    where $D(\alpha,\beta;\alpha_0,\beta_0)=d_1(f(x;\alpha,\beta,\phi),f(x;\alpha_0,\beta_0,\phi))$.
\end{lemma}

\begin{proof}
    Without loss of generality, we assume that $d_g(\alpha,\alpha_0)<\epsilon$ and $|\beta-\beta_0|<\epsilon$ for a sufficiently small $\epsilon$. They by Lemma \ref{lma:e7}, we have
    \#\label{equ:ineq-1-e-6}
    D(\alpha,\beta;\alpha_0,\beta_0)\geq C_1|\beta-\beta_0|,
    \#
    where $C_1>0$ is a constant independent of $\epsilon,\beta,\beta_0,\alpha,$ and $\alpha_0$. Furthermore, by Lemma \ref{lma:c6}, for all 
    $d_g(\alpha,\alpha_0)<\epsilon$ for a sufficiently small $\epsilon$, we have that
    \#\label{equ:ineq-2-e-6}
    D(\alpha,\beta_0;\alpha_0,\beta_0)\geq C_2d_g(\alpha,\alpha_0),
    \#
    where $C_2>0$ is a constant independent of $\epsilon$,  $\alpha_0,\alpha$ and $\beta$. In addition, by Lemma \ref{lma:e-9}, we have for all $\beta,\beta_0\in[\beta_{\min},\beta_{\max}]$ and  $\alpha\in\cM$ that
    \#\label{equ:ineq-3-e-6}
    D(\alpha,\beta;\alpha,\beta_0)\leq C_3|\beta-\beta_0|,
    \#
    where $C_3>0$ is a finite constant independent of the choice of $\beta,\beta_0,$ and $\alpha$. We shall prove Lemma \ref{lma:e6} by combining the above three conclusions. Specifically, we  consider the following two cases separately.
    \begin{itemize}
        \item {\bf Case 1: $|\beta-\beta_0|\leq \frac{C_2}{2C_3}\cdot d_g(\alpha,\alpha_0)$.} In this case, we consider $d_g(\alpha,\alpha_0)<\epsilon$ for a sufficiently small $\epsilon$. By using the triangle inequality, we have
        \$
        D(\alpha,\beta;\alpha_0,\beta_0)&\geq D(\alpha,\beta_0;\alpha_0,\beta_0)-D(\alpha,\beta;\alpha,\beta_0)\\
        &\overset{({\rm i})}{\geq} C_2d_g(\alpha,\alpha_0)-C_3|\beta-\beta_0|\\
        &\overset{({\rm ii})}{\geq} \frac{C_2}{2}d_g(\alpha,\alpha_0)\\
        &\overset{({\rm iii})}{\geq}  \frac{C_2C_3}{2C_3+C_2}(d_g(\alpha,\alpha_0)+|\beta-\beta_0|)
        \$
        where (i) uses \eqref{equ:ineq-2-e-6} and \eqref{equ:ineq-3-e-6} and (ii) and (iii) use the assumption  $|\beta-\beta_0|\leq \frac{C_2}{2C_3}\cdot d_g(\alpha,\alpha_0)$. This inequality proves the first case. 

        \item {\bf Case 2: $|\beta-\beta_0|> \frac{C_2}{2C_3}\cdot d_g(\alpha,\alpha_0)$.} In this case, we consider $|\beta-\beta_0|<\epsilon$ for a sufficiently small $\epsilon$. We have
        \$
        D(\alpha,\beta;\alpha_0,\beta_0)&\geq C_1|\beta-\beta_0|\\
        &\geq \frac{2C_1C_3}{C_2+2C_3}(|\beta-\beta_0|+d_g(\alpha,\alpha_0)),
        \$
        where the first inequality uses \eqref{equ:ineq-1-e-6} and the second inequality uses the condition $|\beta-\beta_0|> \frac{C_2}{2C_3}\cdot d_g(\alpha,\alpha_0)$. 
    \end{itemize}
    By combining the above two cases, we prove the lemma. 
\end{proof}

\subsubsection{Technical lemmas used in Section \ref{sec:e53}}\label{sec:e54}


In the proof of Lemma \ref{lma:e6}, we use the following lemmas. First, in Lemma \ref{lma:e-9}, we provide an upper bound on the quantity $D(\alpha,\beta;\alpha,\beta_0)/|\beta-\beta_0|$, where 
\$
D(\alpha,\beta;\alpha,\beta_0)=d_1(f(x;\alpha,\beta,\phi),f(x;\alpha,\beta_0,\phi)),
\$
and $f(x;\alpha,\beta,\phi)$ is the density in \eqref{equ:6.1-f}. 

\begin{lemma}\label{lma:e-9}
    Suppose $(\cM,g^{\cM})$ is a Riemannian homogeneous space with $r_{\cM}=\sup_{x,y}d_g(x,y)$. Assume $\phi$ satisfies Condition \ref{cond:6.3} for $\beta_{\min}>0$ and is strictly increasing on $[0,r_{\cM}]$. Let $f(x;\alpha,\beta,\phi)$ be the density in \eqref{equ:6.1-f} and $\beta_{\max}>\beta_{\min}$ be a finite constant. Define  
    \$
    D(\alpha,\beta;\alpha_0,\beta_0)=d_1(f(x;\alpha,\beta,\phi),f(x;\alpha_0,\beta_0,\phi)).
    \$ 
    Then for any $\beta,\beta_0\in[\beta_{\min},\beta_{\max}]$, we have
    \$
    D(\alpha,\beta;\alpha,\beta_0)\leq C|\beta-\beta_0|,
    \$
    where $C$ is a finite constant independent of the choice of $\beta$, $\beta_0$, and $\alpha$. 
\end{lemma}

\begin{proof}
    Without loss of generality, we assume that $\beta<\beta_0$. By definition of $f(x;\alpha,\beta,\phi)$, we have that
    \#
    D(\alpha,\beta;\alpha,\beta_0)&=\int_{\cM}\left|\frac{1}{Z(\beta,\phi)}e^{-\beta\phi(d_g(x,\alpha))}-\frac{1}{Z(\beta_0,\phi)}e^{-\beta_0\phi(d_g(x,\alpha))}\right|\dvol(x)\notag\\
    &\leq \int_{\cM}\frac{1}{Z(\beta,\phi)}\cdot \left|e^{-\beta\phi(d_g(x,\alpha))}-e^{-\beta_0\phi(d_g(x,\alpha))}\right|\dvol(x)\notag\\
    &+\int_{\cM}\left|\frac{1}{Z(\beta,\phi)}-\frac{1}{Z(\beta_0,\phi)}\right|\cdot e^{-\beta_0\phi(d_g(x,\alpha))}\dvol(x),\label{equ:decom-e-9}
    \#
    where the inequality uses the triangle inequality. Since $\beta_{\min}\leq \beta<\beta_0\leq\beta_{\max}$, we have
    \$
    \int_{\cM}\frac{1}{Z(\beta,\phi)}\cdot \left|e^{-\beta\phi(d_g(x,\alpha))}-e^{-\beta_0\phi(d_g(x,\alpha))}\right|\dvol(x) \leq \frac{1}{Z(\beta_{\max},\phi)}\cdot |Z(\beta,\phi)-Z(\beta_0,\phi)|,
    \$
    where we use the inequality  $Z(\beta,\phi)\geq Z(\beta_{\max},\phi)$ and the definition of $Z(\beta,\phi)$. Additionally,
    by Lemma \ref{lma:e8}, we have
    \#\label{equ:Z-e35-e-9}
    |Z(\beta,\phi)-Z(\beta_0,\phi)|\leq I_1|\beta-\beta_0|,
    \#
    where $0<I_1<\infty$ is a  constant defined by
    \$
    I_1=\int_{\cM}\phi(d_g(x,\alpha))e^{-\beta_{\min}\phi(d_g(x,\alpha))}\dvol(x).
    \$
    Therefore, 
    \#\label{equ:quantity-1-e-9}
    \int_{\cM}\frac{1}{Z(\beta,\phi)}\cdot \left|e^{-\beta\phi(d_g(x,\alpha))}-e^{-\beta_0\phi(d_g(x,\alpha))}\right|\dvol(x) \leq\frac{I_1}{Z(\beta_{\max},\phi)}\cdot|\beta-\beta_0|. 
    \#
    In addition, we have
    \#
    &\int_{\cM}\left|\frac{1}{Z(\beta,\phi)}-\frac{1}{Z(\beta_0,\phi)}\right|\cdot e^{-\beta_0\phi(d_g(x,\alpha))}\dvol(x)\notag\\
    \overset{({\rm i})}{\leq}\ & \frac{I_1}{Z(\beta_{\max},\phi)^2}\cdot|\beta-\beta_0|\cdot\int_{\cM}e^{-\beta_{\min}\phi(d_g(x,\alpha))}\dvol(x)\notag\\
    \overset{({\rm ii})}{\leq}\ &C_1|\beta-\beta_0|,\label{equ:quantity-2-e-9}
    \#
    where (i) uses \eqref{equ:Z-e35-e-9} and $\beta_{\min}\leq \beta\leq \beta_0\leq \beta_{\max}$, (ii) uses Condition \ref{cond:6.3}, and $C_1>0$ is a finite constant independent of $\beta,\beta_0,$ and $\alpha$.
    By substituting \eqref{equ:quantity-1-e-9} and \eqref{equ:quantity-2-e-9} into \eqref{equ:decom-e-9}, we obtain that
    \$
    D(\alpha,\beta;\alpha,\beta_0)\leq C|\beta-\beta_0|,
    \$
    where $C>0$ is a finite constant independent of $\beta$, $\beta_0$, and $\alpha$. 
\end{proof}

Lemma \ref{lma:e7} establishes a lower bound on the quantity $D(\alpha,\beta;\alpha_0,\beta_0)/|\beta-\beta_0|$, where
\$
D(\alpha,\beta;\alpha_0,\beta_0)=d_1(f(x;\alpha,\beta,\phi),f(x;\alpha_0,\beta_0,\phi)),
\$
and $f(x;\alpha,\beta,\phi)$ is the density in \eqref{equ:6.1-f}. 

\begin{lemma}\label{lma:e7}
    Suppose $(\cM,g^{\cM})$ is a Riemannian homogeneous space with $r_{\cM}=\sup_{x,y}d_g(x,y)$. Assume $\phi$ satisfies Condition \ref{cond:6.3} for $\beta_{\min}>0$ and is strictly increasing on $[0,r_{\cM}]$. Let $f(x;\alpha,\beta,\phi)$ be the density in \eqref{equ:6.1-f} and $\beta_{\max}>\beta_{\min}$ be a finite constant. Define 
    \$
D(\alpha,\beta;\alpha_0,\beta_0)=d_1(f(x;\alpha,\beta,\phi),f(x;\alpha_0,\beta_0,\phi)).
\$
Then for all $\alpha,\alpha_0$ with $d_g(\alpha,\alpha_0)<\epsilon$ for a sufficiently small $\epsilon$ and $\beta,\beta_0\in[\beta_{\min},\beta_{\max}]$, we have
\$
D(\alpha,\beta;\alpha_0,\beta_0)\geq C|\beta-\beta_0|,
\$
where $C>0$ is a constant independent of $\epsilon$, $\beta$, $\beta_0$, $\alpha$, and $\alpha_0$. 
\end{lemma}

\begin{proof}
    Without loss of generality, we assume that $\beta>\beta_0$. By using the definition of $f(x;\alpha,\beta,\phi)$, we have that
    \#\label{equ:D-e22}
    D(\alpha,\beta;\alpha_0,\beta_0)=\int_{\cM}\left|\frac{1}{Z(\beta,\phi)}e^{-\beta\phi(d_g(x,\alpha))}-\frac{1}{Z(\beta_0,\phi)}e^{-\beta_0\phi(d_g(x,\alpha_0))}\right|\dvol(x),
    \#
    where $Z(\beta,\phi)$ is the normalizing constant of $f(x;\alpha,\beta,\phi)$. Thus, we can assume that $\phi(0)=0$ without loss of generality\footnote{Otherwise, we can analyze $\phi-\phi(0)$.}. Since $\phi$ is nonnegative, for any $\beta\in[\beta_{\min},\beta_{\max}]$, we have
        \#\label{equ:Z-beta-bound}
        0<Z(\beta_{\max},\phi)\leq Z(\beta,\phi)\leq Z(\beta_{\min},\phi)<\infty. 
        \#
        Substituting this into \eqref{equ:D-e22}, we obtain that
        \#\label{equ:D-lower-bound}
        D(\alpha,\beta;\alpha_0,\beta_0) 
        \geq \frac{1}{Z(\beta_{\min},\phi)^2}\int_{\cM}\left|W(x)\right|\dvol(x),
        \#
        where
        \$
        W(x)\coloneqq Z(\beta_0,\phi)\cdot e^{-\beta\phi(d_g(x,\alpha))}-Z(\beta,\phi)\cdot e^{-\beta_0\phi(d_g(x,\alpha_0))}.
        \$
        Since $\beta>\beta_0$, it follows from Lemma \ref{lma:e8} that
        \#\label{equ:Z-e27}
            Z(\beta_0,\phi)-I_1(\beta-\beta_0)\leq Z(\beta,\phi)\leq Z(\beta_0,\phi)-I_0(\beta-\beta_0). 
            \#
        where $I_0$ and $I_1$ are constants defined as follows:
        \#
        I_0&=\int_{\cM}\phi(d_g(x,\alpha))e^{-\beta_{\max}\phi(d_g(x,\alpha))}\dvol(x),\label{equ:I0-e23}\\
        I_1&=\int_{\cM}\phi(d_g(x,\alpha))e^{-\beta_{\min}\phi(d_g(x,\alpha))}\dvol(x).\label{equ:I1-e24}
        \#
        Note that $0<I_0\leq I_1<\infty$ are constants independent of $\alpha,\beta$, and $\beta_0$. Let 
            \$
            \cA_{\alpha}=\{x\in\cM\mid d_g(x,\alpha)\leq d_g(x,\alpha_0)\}. 
            \$
        Then for any $x\in\cA_\alpha$, we have
            \$
            W(x)&\geq  Z(\beta_0,\phi)e^{-\beta\phi(d_g(x,\alpha_0))}-Z(\beta,\phi)e^{-\beta_0\phi(d_g(x,\alpha_0))}\\
            &\geq Z(\beta_0,\phi)e^{-\beta\phi(d_g(x,\alpha_0))}-(Z(\beta_0,\phi)-I_0(\beta-\beta_0))\cdot e^{-\beta_0\phi(d_g(x,\alpha_0))}\\
            &= Z(\beta_0,\phi)e^{-\beta_0\phi(d_g(x,\alpha_0))}\cdot \left(e^{-(\beta-\beta_0)\phi(d_g(x,\alpha_0))}-1\right)+I_0(\beta-\beta_0)e^{-\beta_0\phi(d_g(x,\alpha_0))}\\
            &\eqqcolon W_1(x),
            \$
            where we use the fact that $\phi$ is an increasing function and \eqref{equ:Z-e27}. Since $e^{-x}-1\geq-x$ for all $x\geq 0$, we have
            \$
            W_1(x)&\geq (\beta-\beta_0)e^{-\beta_0\phi(d_g(x,\alpha_0))}\cdot\left(I_0-Z(\beta_0,\phi)\phi(d_g(x,\alpha_0))\right)\\
            &\geq (\beta-\beta_0)e^{-\beta_0\phi(d_g(x,\alpha_0))}\cdot\left(I_0-Z(\beta_{\min},\phi)\phi(d_g(x,\alpha_0))\right)\\
            &\eqqcolon W_2(x).
            \$
            Since $\phi(0)=0$ and $\phi$ is a continuous function, there exists a constant $\epsilon_0>0$ such that
            \$
            Z(\beta_{\min},\phi)\phi(r)\leq \frac{I_0}{2},\quad\forall 0\leq r\leq \epsilon_0.
            \$
            Fix any one of such $\epsilon_0$ and let 
            \$
            \cA_{\epsilon_0}=\{x\in\cM\mid d_g(x,\alpha_0)\leq \epsilon_0\}.
            \$ 
            Then for all $x\in\cA_{\alpha}\cap\cA_{\epsilon_0}$, it holds that 
            \#
            W(x)\geq W_2(x)&\geq \frac{I_0}{2}(\beta-\beta_0)e^{-\beta_{\max}\phi(\epsilon_0)}>0. \label{equ:e28}
            \#
            By taking a sufficiently small $\epsilon$, we can ensure that
            \#\label{equ:e29}
            \vol(\cA_{\alpha}\cap\cA_{\epsilon_0})\geq V_0>0,\quad\forall \alpha\in\cB_{\cM}(\alpha_0,\epsilon),
            \#
            where $V_0$ is a constant independent of $\epsilon$, $\alpha$, and $\alpha_0$.
            Thus, by combining \eqref{equ:e28},  \eqref{equ:e29}, and \eqref{equ:D-lower-bound}, we can show that for any $\alpha\in\cB_{\cM}(\alpha_0,\epsilon)$ with a sufficiently small $\epsilon$, 
            \$
            D(\alpha,\beta;\alpha_0,\beta_0)&\geq \frac{1}{Z(\beta_{\min},\phi)^2}\int_{\cA_\alpha\cap\cA_{\epsilon_0}}\frac{I_0}{2}(\beta-\beta_0)e^{-\beta_{\max}\phi(\epsilon_0)}\dvol(x)\\
            &\geq \frac{1}{Z(\beta_{\min},\phi)^2}\frac{I_0}{2}e^{-\beta_{\max}\phi(\epsilon_0)}\cdot V_0\cdot(\beta-\beta_0)\\
            &\eqqcolon C(\beta-\beta_0),
            \$
            where $C>0$ is a constant independent of the choice of $\epsilon$, $\beta$, $\beta_0$, $\alpha$, and $\alpha_0$. 
\end{proof}

Lemma \ref{lma:e8} examines the quantity $|Z(\beta,\phi)-Z(\beta_0,\phi)|/|\beta-\beta_0|$, where $Z(\beta,\phi)$ is the normalizing constant of the distribution $f(x;\alpha,\beta,\phi)$. This lemma is useful in the proof of Lemma \ref{lma:e-9} and Lemma \ref{lma:e7}. 

\begin{lemma}\label{lma:e8}
    Suppose $(\cM,g^{\cM})$ is a Riemannian homogeneous space with $r_{\cM}=\sup_{x,y}d_g(x,y)$. Assume $\phi$ satisfies Condition \ref{cond:6.3} for $\beta_{\min}>0$ and is strictly increasing on $[0,r_{\cM}]$. 
    Let $Z(\beta,\phi)$ be the normalizing constant of the function $f(x;\alpha,\beta,\phi)$ and $\beta_{\max}\geq \beta_{\min}$ a finite constant. 
    Then for any different $\beta,\beta_0\in[\beta_{\min},\beta_{\max}]$, we have 
    \$
    0<I_0\leq \left|\frac{Z(\beta,\phi)-Z(\beta_0,\phi)}{\beta-\beta_0}\right|\leq I_1<\infty,
    \$
    where $I_0$ and $I_1$ are constants defined as follows:
        \#
        I_0&=\int_{\cM}\phi(d_g(x,\alpha))e^{-\beta_{\max}\phi(d_g(x,\alpha))}\dvol(x), \label{equ:I0-lemma-e8}\\
        I_1&=\int_{\cM}\phi(d_g(x,\alpha))e^{-\beta_{\min}\phi(d_g(x,\alpha))}\dvol(x). \label{equ:I1-lemma-e8}
        \#
\end{lemma}

\begin{proof}
    Without loss of generality, we assume that $\beta<\beta_0$. By using the definition of $Z(\beta,\phi)$, we have 
    \#\label{equ:diff-Z-int}
    Z(\beta,\phi)-Z(\beta_0,\phi)=\int_{\cM}\left(e^{-\beta\phi(d_g(x,\alpha))}-e^{-\beta_0\phi(d_g(x,\alpha))}\right)\dvol(x).
    \#
    Since $\beta<\beta_0$ and $\phi$ is nonnegative, the integrand in \eqref{equ:diff-Z-int} is nonnegative. Also, we notice that for any $r\geq 0$ and $\beta,\beta_0\in[\beta_{\min},\beta_{\max}]$, the following bounds hold:
    \$
    e^{-\beta_{\max}r}r\leq\frac{|e^{-\beta r}-e^{-\beta_0 r}|}{|\beta-\beta_0|}\leq e^{-\beta_{\min}r}r.
    \$
    Substituting this into \eqref{equ:diff-Z-int}, we obtain that
    \$
    I_0\leq \left|\frac{Z(\beta,\phi)-Z(\beta_0,\phi)}{\beta-\beta_0}\right|\leq I_1,
    \$
    where $I_0$ and $I_1$ are defined by \eqref{equ:I0-lemma-e8} and \eqref{equ:I1-lemma-e8} respectively. Since $\phi$ satisfies Condition \ref{cond:6.3} for $\beta_{\min}>0$ and $\phi$ is strictly increasing on $[0,r_{\cM}]$, one can easily show that $0<I_0\leq I_1<\infty$ using the volume comparison theorem, Theorem \ref{thm:2.3}. This concludes the proof of this lemma. 
\end{proof}






\section{Technical proofs for Section \ref{sec:6}}

This section collects proofs for the results in Section \ref{sec:6}. Specifically, Section \ref{sec:f71} provides the proof of Theorem \ref{thm:7.1}. Section \ref{sec:f72} presents the proof of Corollary \ref{corollary:7.2}. Section \ref{sec:f73} provides the proof of Theorem \ref{thm:7.3}, and Section \ref{sec:f74} gives the proof of Theorem \ref{thm:7.4}. The remaining technical lemmas are postponed to Section \ref{sec:f75}. 

\subsection{Proof of Theorem \ref{thm:7.1}}\label{sec:f71}

\begin{proof}
    Let $\epsilon_0\in[0,\pi/2]$ be the constant such that $\sin\epsilon_0=\epsilon_0/2$. By Lemma \ref{lma:f1-f51}, for any $\epsilon\leq 2\epsilon_0$, we can construct an $\epsilon$-net $\cS=\{\alpha_i\}_{i=1}^N$ of $\cM=\SSS^{m}$ with 
    \$
    N\leq \frac{\vol(\SSS^m)}{\vol(\SSS^{m-1})} \cdot 4^mm\epsilon^{-m}.
    \$
    Using this set  $\cS$, we can construct the following set
    \$
    \cS_{\cF}=\{f_{\vMF}(x;\alpha_i,\beta)\mid \alpha_i\in\cS\},
    \$
    whose cardinality is also $N$. For any $\alpha\in\SSS^m$, there exists an $\alpha_i\in\cS$ such that $d_g(\alpha,\alpha_i)\leq \epsilon$. By the Lipschitz property, Lemma \ref{lma:f3-f52}, we have
    \$
    d_{\infty}(f_{\vMF}(x;\alpha,\beta),f_{\vMF}(x;\alpha_i,\beta))&\leq \beta e^{2\beta}\vol(\SSS^m)\cdot d_g(\alpha,\alpha_i)\\
    &\leq \beta e^{2\beta}\vol(\SSS^m)\cdot \epsilon,
    \$
    Therefore, $\cS_{\cF}$ is a $(\vol(\SSS^m)\beta e^{2\beta}\epsilon)$-net of $\cF$ under $d_\infty$ and the following entropy estimate holds:
    \$
    \cN\left(\vol(\SSS^m)\beta e^{2\beta}\epsilon,\cF,d_{\infty}\right)\leq |\cS_{\cF}|\leq \frac{\vol(\SSS^m)}{\vol(\SSS^{m-1})} \cdot 4^mm\epsilon^{-m}.
    \$
    By rescaling $\epsilon$, we obtain that for a sufficiently small $\epsilon$,
    \#\label{equ:N-F-sec-f}
    \cN(\epsilon,\cF,d_{\infty})\leq \frac{\vol(\SSS^m)}{\vol(\SSS^{m-1})} \cdot\left(4\beta e^{2\beta}\vol(\SSS^m)\right)^mm\cdot \epsilon^{-m}.
    \#
    To proceed, we let $\{f_i\}_{i=1}^{N_*}$ be a $\epsilon$-net of $\cF$ under  $d_\infty$ with 
    \$
    N_*=\cN(\epsilon,\cF,d_\infty).
    \$
    Then we construct the brackets $[l_i,u_i]$ as follows:
    \$
    l_i=\max\{0,f_i-\epsilon\},\quad u_i=f_i+\epsilon.
    \$
    Then it is clear that $\cF\subseteq\cup_{i}[l_i,u_i]$. In addition,
    \$
    d_1(l_i,u_i)\leq 2\epsilon \cdot \vol(\SSS^m).
    \$
    Therefore, $\{[l_i,u_i]\}$ are $2\epsilon\cdot\vol(\SSS^m)$-brackets covering $\cF$. It implies that
    \$
    \cN_{B}(2\epsilon\cdot\vol(\SSS^m),\cF,d_1)&\leq \cN(\epsilon,\cF,d_{\infty})\\
    &\leq \frac{\vol(\SSS^m)}{\vol(\SSS^{m-1})} \cdot\left(4\beta e^{2\beta}\vol(\SSS^m)\right)^mm\cdot \epsilon^{-m},
    \$
    where we use \eqref{equ:N-F-sec-f}. By rescaling $\epsilon$, we obtain that for a sufficiently small $\epsilon$,
    \#\label{equ:NB-F-sec-f}
    \cN_{B}(\epsilon,\cF,d_1)\leq \frac{\vol(\SSS^m)}{\vol(\SSS^{m-1})} \cdot\left(8\beta e^{2\beta}\right)^m\left(\vol(\SSS^m)\right)^{2m}m\cdot \epsilon^{-m}.
    \#
    Since $\cN_B(\epsilon,\cF,d_h)\leq \cN_{B}(\epsilon^2,\cF,d_1)$, we obtain that 
    \#\label{equ:NB-F-sec-f-2}
    \cN_B(\epsilon,\cF,d_h)\leq  \frac{\vol(\SSS^m)}{\vol(\SSS^{m-1})} \cdot\left(8\beta e^{2\beta}\right)^m\left(\vol(\SSS^m)\right)^{2m}m\cdot \epsilon^{-2m}.
    \#
    To further simplify \eqref{equ:N-F-sec-f}, \eqref{equ:NB-F-sec-f}, and \eqref{equ:NB-F-sec-f-2}, we observe that
    \$
    \vol(\SSS^m)=\int_0^\pi\sin^{m-1}rdr\cdot\vol(\SSS^{m-1})\leq \pi\cdot \vol(\SSS^{m-1}).
    \$ 
    Also, since $\vol(\SSS^m)$ tends to zero as $m$ tends to infinity, one can show that $\vol(\SSS^m)\leq C_0$ for some finite constant $C_0$ independent of $m$. Thus,  \eqref{equ:N-F-sec-f}, \eqref{equ:NB-F-sec-f}, and \eqref{equ:NB-F-sec-f-2} can be rewritten as follows:
    \$
    \cN(\epsilon,\cF,d_{\infty})&\leq \pi m\left(4C_0\beta e^{2\beta}\right)^m\cdot \epsilon^{-m}\leq e^{Cm}\epsilon^{-m},\\
    \cN_{B}(\epsilon,\cF,d_1)&\leq \pi m\left(8C_0^2\beta e^{2\beta}\right)^{m}\cdot\epsilon^{-m}\leq e^{Cm}\epsilon^{-m},\\
    \cN_B(\epsilon,\cF,d_h)&\leq \pi m\left(8C_0^2\beta e^{2\beta}\right)^{m}\cdot\epsilon^{-2m}\leq e^{Cm}\epsilon^{-m},
    \$
    where $C>0$ be a finite constant such that  
    \$
    \max\left\{\pi m \left(4 C_0\beta e^{2\beta}\right)^m,\pi m\left(8C_0^2\beta e^{2\beta}\right)^m\right\}\leq e^{Cm}.
    \$
    Since $C$ is chosen independent of $m$ and $\epsilon$, we can show that for a sufficiently small $\epsilon$, 
    \$
    \log \cN(\epsilon,\cF,d_{\infty})&\leq Cm+m\log(\frac{1}{\epsilon})\leq 2m\log(\frac{1}{\epsilon}),\\
    \log \cN_{B}(\epsilon,\cF,d_1)&\leq 2m\log(\frac{1}{\epsilon}),\\
    \log\cN_B(\epsilon,\cF,d_h)&\leq 2m\log(\frac{1}{\epsilon}),
    \$
    which concludes the proof.
\end{proof}

\subsection{Proof of Corollary \ref{corollary:7.2}}\label{sec:f72}

\begin{proof}
     By the bracketing entropy estimates in Theorem \ref{thm:7.1}, the bracketing entropy integral satisfies that
    \$
    J_B(\delta,\cF,d_h)&\coloneqq\int_0^\delta \sqrt{\log\cN_{B}(u,\cF,d_h)}du\\
    &\leq \sqrt{2m}\int_0^{\delta} \log^{1/2}(\frac{1}{u})du\\
    &\leq 2\sqrt{m}\delta\log(\frac{1}{\delta}),
    \$
    for sufficiently small $\delta$. Then by Theorem 2,  \citet{wong1995probability}, we obtain that for sufficiently large $n$, the following inequality
    \$
    d_h(f_{\vMF}(x;\alpha^{\MLE},\beta),f_{\vMF}(x;\alpha^{\tr},\beta))\leq C\cdot \frac{\sqrt{m}\log n}{\sqrt{n}}
    \$
    holds with probability at least $1-ce^{-cm\log^2n}$, 
    where $c,C>0$ are constants independent of $m$ and $n$. 
    This proves \eqref{equ:dh-cor-7.2}. Also, since the $L^1$ distance  is upper bounded by twice the hellinger distance, the inequality \eqref{equ:d1-cor-7.2}  follows from the inequality \eqref{equ:dh-cor-7.2}. This proves the corollary.  
\end{proof}

\subsection{Proof of Theorem \ref{thm:7.3}}\label{sec:f73}

\begin{proof}
    Using the same argument in Corollary \ref{corollary:4.4}, we know that for sufficiently large $n$, the distance $d_g(\hat\alpha^{\MLE},\alpha^{\tr})$ is sufficiently small with high probability. Then we can combine the  identifiability result, Lemma \ref{lma:f4-f53}, with Corollary \ref{corollary:7.2} to prove the desired result. 
\end{proof}

\subsection{Proof of Theorem \ref{thm:7.4}}\label{sec:f74}

\begin{proof}
    Let $\epsilon_0\in[0,\pi/2]$ be a constant such that
    \$
    \sin\epsilon_0=\epsilon_0/2.
    \$
    Then we pick a $2\delta$-packing  $\cS=\{\alpha_i\}_{i=1}^N$ of the set $\cA=\cB_{\SSS^m}(\alpha,8\delta)$ for some $\delta\leq\epsilon_0/8$. By Lemma \ref{lma:f2-f51}, we know $\cS$ can be chosen such that $N\geq 2^m$. Then by the Fano's method, i.e., inequality \eqref{equ:5.4}, we can show that
    \#\label{equ:mini-risk-lower-bound-proof-f}
    \cR_{n,\vMF}(\beta)\geq \delta\cdot\left\{1-\frac{n\cdot\sup_{\alpha,\alpha'\in\cS}D_{\KL}(P_{\alpha}\|P_{\alpha'})+\log 2}{m\log 2}\right\},
    \#
    where $P_{\alpha}$ represents the distribution $f_{\vMF}(x;\alpha,\beta)$, and $D_{\KL}$ is the KL divergence. Notice that for all $\alpha,\alpha'\in\cS$, the inequality $d_g(\alpha,\alpha')\leq 16\delta$ holds. Therefore, by Lemma \ref{lma:f5-f54}, we have
    \$
    \sup_{\alpha,\alpha'\in\cS}D_{\KL}(P_{\alpha}\|P_{\alpha'})\leq \frac{C\delta^2}{m},
    \$
    where $C>0$ is a constant independent of $\delta,\cS,$ and $m$. By substituting this into \eqref{equ:mini-risk-lower-bound-proof-f}, we obtain that 
    \$
    \cR_{n,\vMF}(\beta)\geq \delta\cdot\{1-\frac{n\cdot C\delta^2/m+\log 2}{m\log 2}\}.
    \$
    Thus, by taking $Cn\delta^2=0.01m^2$, we obtain the following inequality
    \$
    \cR_{n,\vMF}(\beta)\geq \frac{C'm}{\sqrt{n}},
    \$
    where $C'>0$ is a constant independent of $n$ and $m$.
\end{proof}

\subsection{Technical lemmas}\label{sec:f75}

\subsubsection{Geometric lemma}

Lemma \ref{lma:f1-f51} establishes an upper bound on the covering number $\cN(\epsilon,\SSS^m,d_g)$. 

\begin{lemma}\label{lma:f1-f51}
    Suppose $\cM=\SSS^{m}$ is the unit $m$-sphere. Let $\epsilon_0\in[0,\pi/2]$ be the constant such that 
    \$
    \sin\epsilon_0=\epsilon_0/2. 
    \$
    Then for any $\epsilon\leq 2\epsilon_0$, it holds that
    \#
    \cN(\epsilon,\SSS^m,d_g)&\leq \frac{\vol(\SSS^m)}{\vol(\SSS^{m-1})} \cdot 4^mm\epsilon^{-m},\label{equ:N-S-f5}
    \#
    where $\cN(\epsilon,\SSS^m,d_g)$ represents the covering number.
\end{lemma}

\begin{proof}
    Construct an $\epsilon$-net $\cS$ of $\cM$ by the following greedy procedure. Let $x_1\in\cM$ be an arbitrary point. Suppose $x_1,\ldots,x_s$ have been chosen. If the set $\{x\in\cM\mid d_g(x,x_i)>\epsilon,i\leq s\}$ is not empty, we pick an arbitrary point $x_{s+1}$ in this set and add it to $\cS$. Otherwise, we end the construction of the set $\cS$. Such set $\cS$ is easily shown to be an $\epsilon$-net of $\cM$ as well as an $\epsilon$-packing of $\cM$. 
    
    To prove the lemma, it suffices to provide a suitable upper  bound on the cardinality of the set $\cS$. 
    Since $\cS$ is an $\epsilon$-packing,  the geodesic balls $\cB_{\cM}(x_i,\epsilon/2)$ and $\cB_{\cM}(x_j,\epsilon/2)$ are disjoint for any different $x_i,x_j\in\cS$. 
        Thus, the volume of $\cup_{x_i\in\cS}\cB_{\cM}(x_i,\epsilon/2)$ is equal to the sum of the volumes of these geodesic balls. In addition, 
        since $\cM$ is a sphere, we have 
    \#\label{equ:cV-f}
    \vol(\cB_{\cM}(x_i,\epsilon/2))=\vol(\cB_{\cM}(x_j,\epsilon/2))\eqqcolon\cV(\epsilon/2)
    \#
    for any pair of points $x_i,x_j\in\cM$ and any $\epsilon>0$. It follows that 
    \$
    \vol\left(\cup_{x_i\in\cS}\cB_{\cM}(x_i,\epsilon/2)\right)=|\cS|\cdot\cV(\epsilon/2).
    \$
    Since $\cup_{x_i\in\cS}\cB_{\cM}(x_i,\epsilon/2)\subseteq\SSS^m$, we can show that
    \#\label{equ:S-upperbound-f5}
    |\cS|\leq \frac{\vol(\SSS^m)}{\cV(\epsilon/2)}.
    \#
    Using the integral formula on $\SSS^m$, one can show that
    \#\label{equ:cV-integral-f}
    \cV(\delta)\coloneqq\vol(\cB_{\cM}(x,\delta))=\int_0^\delta \sin^{m-1}rdr\cdot \vol(\SSS^{m-1}).
    \#
    By the choice of $\epsilon_0$, we know that 
    \$
    \sin r\geq r/2,\quad\forall r\in[0,\epsilon_0].
    \$
    Then for all $\epsilon\leq 2\epsilon_0$, it holds that
    \$
    \int_0^{\epsilon/2}\sin^{m-1}rdr\geq \int_0^{\epsilon/2}r^{m-1}dr/2^{m-1}=\frac{\epsilon^m}{2^{2m-1}m}\geq \frac{\epsilon^m}{4^mm}.
    \$
    By combining this with \eqref{equ:S-upperbound-f5} and \eqref{equ:cV-integral-f}, we obtain that 
    \$
    \cN(\epsilon,\SSS^m,d_g)\leq |\cS|\leq \frac{\vol(\SSS^m)}{\vol(\SSS^{m-1})} \cdot 4^mm\epsilon^{-m},
    \$
    for all $\epsilon\leq2\epsilon_0$. This proves the lemma. 
\end{proof}

Lemma \ref{lma:f2-f51} lower bounds the packing number $M(2\delta,\cA,d_g)$, where $\cA=\cB_{\SSS^m}(\alpha,8\delta)$. 

\begin{lemma}\label{lma:f2-f51}
    Suppose $\cM=\SSS^m$ is the unit $m$-sphere. Let $\epsilon_0\in[0,\pi/2]$ be the constant such that
    \$
    \sin\epsilon_0=\epsilon_0/2.
    \$
    Let $\cA=\cB_{\cM}(\alpha,8\delta)$ for some $\alpha\in\cM$. Then for any $\delta\leq \epsilon_0/8$, we have
    \$
    M(2\delta,\cA,d_g)\geq 2^m,
    \$
    where $M(\cdot,\cA,d_g)$ denotes the packing number. 
\end{lemma}

\begin{proof}
    We construct a $2\delta$-packing $\cS$ of $\cA$ by the following greedy procedure. Let $x_1\in\cA$ be an arbitrary point. Suppose $x_1,\ldots,x_s$ have been chosen. If the set $\{x\in\cA\mid d_g(x,x_i)>2\delta,i\leq s\}$ is not empty, we pick an arbitrary point $x_{s+1}$ in this set and add it to $\cS$. Else, we end the construction of the set $\cS$. Such set $\cS$ is easily shown to be a $2\delta$-net of $\cA$ as well as a $2\delta$-packing of $\cA$. 
    
    To prove the lemma, it suffices to provide a suitable lower bound on the cardinality of the set $\cS$. Since $\cS$ is a $2\delta$-net of $\cA$, the set $\cup_{x_i\in\cS}B_{\cM}(x_i,2\delta)$ covers $\cA$. Thus, we have 
    \$
    \sum_{x_i\in\cS}\vol(\cB_{\cM}(x_i,2\delta))\geq \vol(\cup_{x_i\in\cS}\cB_{\cM}(x_i,2\delta))\geq \vol(\cA).
    \$
    Since $\cM=\SSS^m$, we have 
    \#\label{equ:cV-f-lmaf2}
    \vol(\cB_{\cM}(x_i,\epsilon))=\vol(\cB_{\cM}(x_j,\epsilon))\eqqcolon\cV(\epsilon)
    \#
    for any pair of points $x_i,x_j\in\cM$ and any $\epsilon>0$. Therefore, it follows that
    \$
    |\cS|\cdot \cV(2\delta)\geq \cV(8\delta).
    \$
    By using the integral formula on $\SSS^m$, we can show that 
    \$
    \cV(\epsilon)=\int_0^\epsilon\sin^{m-1}rdr\cdot \vol(\SSS^{m-1}).
    \$
    Hence, the cardinality of the set $\cS$ can be lower bounded as follows:
    \$
    |\cS|\geq \frac{\int_0^{8\delta}\sin^{m-1}rdr}{\int_0^{2\delta}\sin^{m-1}rdr}.
    \$
    Since $8\delta\leq\epsilon_0$, where $\sin\epsilon_0=\epsilon_0/2$, we have that
    \$
    \int_0^{8\delta}\sin^{m-1}rdr\geq \int_0^{8\delta}r^{m-1}2^{-m}dr=\frac{(8\delta)^m}{m2^m}.
    \$
    Also, since $\sin r\leq r$ for all $r\geq 0$, we have
    \$
    \int_0^{2\delta}\sin^{m-1}rdr\leq \int_0^{2\delta}r^{m-1}dr=\frac{(2\delta)^m}{m}.
    \$
    Consequently,
    \$
    M(2\delta,\cA,d_g)\geq|\cS|\geq 2^m,
    \$
    which concludes the proof.
\end{proof}

\subsubsection{Lipschitz properties}

\begin{lemma}\label{lma:f3-f52}
    Let $f_{\vMF}(x;\alpha,\beta)$ be the von Mises-Fisher distribution on the $m$-sphere $\SSS^m$, defined in \eqref{equ:vMF-sec7}. Then for any $\alpha,\alpha'\in\SSS^m$, the following inequality holds:
    \$
    d_{\infty}(f_{\vMF}(x;\alpha,\beta),f_{\vMF}(x;\alpha',\beta))\leq \beta e^{2\beta}\vol(\SSS^m)\cdot d_g(\alpha,\alpha').
    \$
\end{lemma}

\begin{proof}
    By definition of $f_{\vMF}(x;\alpha,\beta)$, we have
    \$
    |f_{\vMF}(x;\alpha,\beta)-f_{\vMF}(x;\alpha',\beta)|=\frac{1}{Z_{\vMF}(\beta)}\cdot|e^{\beta \cos d_g(x,\alpha)}-e^{\beta \cos d_g(x,\alpha')}|,
    \$
    where $Z_{\vMF}(\beta)$ is the normalizing constant given by \eqref{equ:vMF-sec7}. Observe that the function $h(r)=e^{\beta\cos r}$ is a Lipschitz continuous function with a Lipschitz constant $\beta e^{\beta}$. Thus, 
    \$
    |f_{\vMF}(x;\alpha,\beta)-f_{\vMF}(x;\alpha',\beta)|&\leq \frac{\beta e^{\beta}}{Z_{\vMF}(\beta)}\cdot |d_g(x,\alpha)-d_g(x,\alpha')|\\
    &\leq \frac{\beta e^{\beta}}{Z_{\vMF}(\beta)}\cdot d_g(\alpha,\alpha'),
    \$
    where the second inequality uses the triangle inequality. By simple estimation, we can show that
    \$
    Z_{\vMF}(\beta)\geq e^{-\beta}\vol(\SSS^m).
    \$
    Thus, 
    \$
    |f_{\vMF}(x;\alpha,\beta)-f_{\vMF}(x;\alpha',\beta)|\leq \beta e^{2\beta}\vol(\SSS^m)\cdot d_g(\alpha,\alpha').
    \$
    By taking supremum over $x$, we conclude the proof of this lemma. 
\end{proof}

\subsubsection{Identifiability}

Lemma \ref{lma:f4-f53} characterizes the parameter identifiability for the von Mises-Fisher distribution. 

\begin{lemma}\label{lma:f4-f53}
    Let $f_{\vMF}(x;\alpha,\beta)$ be the von Mises-Fisher distribution on $\SSS^m$. Then for a sufficiently small $d_g(\alpha,\alpha_0)$, the following lower bound holds:
    \$
    d_1(f_{\vMF}(x;\alpha,\beta),f_{\vMF}(x;\alpha_0,\beta))\geq \frac{C d_g(\alpha,\alpha_0)}{\sqrt{m}},
    \$
    where $C>0$ is a constant independent of  $\alpha,\alpha_0,$ and $m$. 
\end{lemma}

\begin{proof}
    It suffices to prove this lemma for sufficiently large $m$, since Theorem \ref{thm:4.3} already addresses the case where $m$ is small. For convenience, we denote by 
    \$
    D(\alpha,\alpha_0)=d_1(f_{\vMF}(x;\alpha,\beta),f_{\vMF}(x;\alpha_0,\beta)).
    \$
    By definition of von Mises-Fisher distribution, we have
    \#\label{equ:int-lma-f4}
    D(\alpha,\alpha_0)=\frac{1}{Z_{\vMF}(\beta)}\int_{\SSS^m}\left|e^{\beta\cos d_g(x,\alpha)}-e^{\beta\cos d_g(x,\alpha_0)}\right|\dvol(x),
    \#
    where $Z_{\vMF}(\beta)$ is the normalizing constant defined in \eqref{equ:vMF-sec7}. Denote by $r_0=d_g(\alpha,\alpha_0)$. Without loss of generality, we assume that $m>2$. Using the spherical geometry, we can rewrite  the integral in \eqref{equ:int-lma-f4} as follows:
    \#
    &\int_{\SSS^m}\left|e^{\beta\cos d_g(x,\alpha)}-e^{\beta\cos d_g(x,\alpha_0)}\right|\dvol(x)\notag\\
    =&\int_{0}^{\pi}\int_0^\pi  h(\varphi_1,\varphi_2)\sin^{m-1}\varphi_1\sin^{m-2}\varphi_2d\varphi_1d\varphi_2 \cdot \vol(\SSS^{m-2}),\label{equ:int-h-lma-f4}
    \#
    where 
    \$
    h(\varphi_1,\varphi_2)&=\left|e^{\beta(\cos\varphi_1\cos r_0+\sin\varphi_1\cos\varphi_2\sin r_0)}-e^{\beta\cos\varphi_1}\right|\\
    &=e^{\beta\cos\varphi_1}\cdot\left|e^{\beta(\cos\varphi_1(\cos r_0-1)+\sin\varphi_1\cos\varphi_2\sin r_0)}-1\right|.
    \$
    There exists a sufficiently small constant $r_*>0$ such that for all $r_0\leq r_*$,  
    \$
    h(\varphi_1,\varphi_2)&\geq e^{-\beta}\cdot \frac{\beta}{2}\cdot \left|\cos\varphi_1(\cos r_0-1)+\sin\varphi_1\cos\varphi_2\sin r_0\right|\\
    &\geq e^{-\beta}\cdot \frac{\beta}{2}\cdot \left(\frac{r_0}{2}\left|\sin\varphi_1\cos\varphi_2\right|-\frac{r_0^2}{2}\left|\cos\varphi_1 \right|\right).
    \$
    Therefore, for all $r_0\leq r_*$, it holds that
    \#
    &\frac{2e^{\beta}}{\beta}\cdot \int_{0}^{\pi}\int_0^\pi  h(\varphi_1,\varphi_2)\sin^{m-1}\varphi_1\sin^{m-2}\varphi_2d\varphi_1d\varphi_2 \notag \\
    \geq &\int_{0}^{\pi}\int_0^\pi  \frac{r_0}{2}\left|\sin\varphi_1\cos\varphi_2\right|\sin^{m-1}\varphi_1\sin^{m-2}\varphi_2d\varphi_1d\varphi_2\notag \\
    -&\int_{0}^{\pi}\int_0^\pi  \frac{r_0^2}{2}\left|\cos\varphi_1\right|\sin^{m-1}\varphi_1\sin^{m-2}\varphi_2d\varphi_1d\varphi_2\notag \\
    =&\ r_0\int_0^\pi \sin^m\varphi_1d\varphi_1\int_0^{\pi/2}\sin^{m-2}\varphi_2\cos\varphi_2d\varphi_2 \notag\\
    -&\ r_0^2\int_0^{\pi/2}\sin^{m-1}\varphi_1\cos \varphi_1d\varphi_1\int_0^\pi \sin^{m-2}\varphi_2d\varphi_2 \notag\\
    =&\ \frac{r_0}{m-1}  \int_0^\pi\sin^{m}\varphi d\varphi - \frac{r_0^2}{m}\int_0^\pi\sin^{m-2}\varphi d\varphi.\label{equ:f11-lma-f4}
    \#
    When $m$ is sufficiently large, we have 
    \$
    \int_0^{\pi}\sin^m\varphi d\varphi&=\int_0^\pi\sin^{m-2}\varphi d\varphi-\int_0^\pi\sin^{m-2}\varphi\cos^2\varphi d\varphi\\
    &\geq \frac{1}{2}\int_0^\pi \sin^{m-2}\varphi d\varphi.
    \$
    Then for sufficiently large $m$ and $r_0\leq \min\{r_*,1/4\}$, we have
    \$
    \frac{r_0}{m-1}  \int_0^\pi\sin^{m}\varphi d\varphi - \frac{r_0^2}{m}\int_0^\pi\sin^{m-2}\varphi d\varphi&\geq \left(\frac{r_0}{2m}-\frac{r_0^2}{m}\right)\int_0^\pi \sin^{m-2}\varphi d\varphi\\
    &\geq \frac{r_0}{4m}\int_0^\pi\sin^{m-2}\varphi d\varphi.
    \$
    Combining this with \eqref{equ:f11-lma-f4}, \eqref{equ:int-h-lma-f4}, and \eqref{equ:int-lma-f4}, we obtain that for a sufficiently large $m$ and $r_0\leq\min\{r_*,1/4\}$, the following bound holds:
    \#\label{equ:f13-lma-4}
    D(\alpha,\alpha_0)\geq \frac{1}{Z_{\vMF}(\beta)}\cdot \vol(\SSS^{m-2}) \cdot \frac{\beta}{2e^\beta}\cdot \frac{r_0}{4m}\int_0^\pi\sin^{m-2}\varphi d\varphi. 
    \#
    By definition, we can show that 
    \$
    Z_{\vMF}(\beta)\leq e^\beta\vol(\SSS^m)=e^\beta\int_0^\pi\sin^{m-1}\varphi_1d\varphi_1\int_0^\pi\sin^{m-2}\varphi_2 d\varphi_2\cdot\vol(\SSS^{m-2}).
    \$
    Substituting this into \eqref{equ:f13-lma-4}, we obtain that
    \$
    D(\alpha,\alpha_0)\geq \frac{\beta r_0}{8me^{2\beta}\int_0^\pi\sin^{m-1}\varphi d\varphi}.
    \$
    Using Stirling formula, one can show that
    \$
    L\coloneqq \sup_{m}\sqrt{m}\int_0^\pi\sin^{m}\varphi d\varphi>0
    \$
    is a finite constant independent of $m$. Therefore,
    \$
    D(\alpha,\alpha_0)\geq \frac{Cr_0}{\sqrt{m}},
    \$
    where $C>0$ is a constant independent of $m$ and $r_0$. 
\end{proof}

\subsubsection{Lemmas for minimax analysis}

\begin{lemma}\label{lma:f5-f54}
    Let $f_{\vMF}(x;\alpha,\beta)$ be the von Mises-Fisher distribution on the $m$-sphere $\SSS^m$. Then 
    \$
    D_{\KL}(P_{\alpha}\|P_{\alpha'})\leq \frac{Cd_g^2(\alpha,\alpha')}{2m},
    \$
    where $P_{\alpha}$ denotes the distribution $f_{\vMF}(x;\alpha,\beta)$ and $C>0$ is a finite constant independent of $\alpha,\alpha'$, and $m$.
\end{lemma}

\begin{proof}
    It suffices to consider a sufficiently large $m$, since when $m$ is small, Proposition \ref{prop:5.8} already proves the result. Without loss of generality, we assume that $m>2$. Using the definition of the von Mises-Fisher distribution the KL divergence, we have that
    \#\label{equ:KL-lma-f5-f54}
    D_{\KL}(P_{\alpha}\|P_{\alpha'})=\frac{\beta}{Z_{\vMF}(\beta)}\int_{\SSS^m}(\cos d_g(x,\alpha)-\cos d_g(x,\alpha'))e^{\beta\cos d_g(x,\alpha)}\dvol(x)
    \#
    Denote by $r_0=d_g(\alpha,\alpha')$. Then using the spherical geometry, we have 
    \#
    &\int_{\SSS^m}(\cos d_g(x,\alpha)-\cos d_g(x,\alpha'))e^{\beta\cos d_g(x,\alpha)}\dvol(x)\notag\\
    =& \int_0^\pi\int_0^\pi h(\varphi_1,\varphi_2)\sin^{m-1}\varphi_1\sin^{m-2}\varphi_2d\varphi_1d\varphi_2\cdot\vol(\SSS^{m-2})\notag\\
    \eqqcolon &J,\label{equ:int-f-13-f54}
    \#
    where
    \$
    h(\varphi_1,\varphi_2)=
    (\cos\varphi_1-(\cos\varphi_1\cos r_0+\sin\varphi_1\cos\varphi_2\sin r_0))e^{\beta \cos\varphi_1}.
    \$
    By simple calculation, we find that
    \$
    \int_0^\pi\sin^{m-2}\varphi_2\cos\varphi_2d\varphi_2=0.
    \$
    Thus, the integral \eqref{equ:int-f-13-f54} reduces to 
    \#\label{equ:int-2-f-14-f54}
    J&=\int_0^\pi\int_0^\pi h_1(\varphi_1)\sin^{m-1}\varphi_1\sin^{m-2}\varphi_2d\varphi_1d\varphi_2\cdot \vol(\SSS^{m-2})\notag\\
    &=\int_0^\pi h_1(\varphi_1)\sin^{m-1}\varphi_1 d\varphi_1\cdot\vol(\SSS^{m-1}),
    \#
    where 
    \$
    h_1(\varphi_1)=\cos\varphi_1(1-\cos r_0)e^{\beta \cos\varphi_1}. 
    \$
    Since $0\leq 1-\cos r_0\leq r_0^2/2$, we know that
    \$
    J\leq \frac{r_0^2}{2}\int_0^\pi e^{\beta\cos\varphi}\cos \varphi \sin^{m-1}\varphi d\varphi\cdot \vol(\SSS^{m-1}).
    \$
    Substituting this into \eqref{equ:KL-lma-f5-f54}, we obtain that 
    \#\label{equ:final-equ}
    D_{\KL}(P_{\alpha}\|P_{\alpha'})\leq \frac{\beta}{Z_{\vMF}(\beta)}\frac{r_0^2}{2}\int_0^\pi e^{\beta\cos\varphi}\cos \varphi \sin^{m-1}\varphi d\varphi\cdot \vol(\SSS^{m-1}).
    \#
    Observe that
    \$
    Z_{\vMF}(\beta)=\int_0^{\pi}e^{\beta\cos \varphi}\sin^{m-1}\varphi d\varphi\cdot \vol(\SSS^{m-1}). 
    \$
    In addition, 
    \$
    \int_0^\pi e^{\beta\cos\varphi}\cos\varphi \sin^{m-1}\varphi d\varphi&=\frac{\beta}{m}\int_0^\pi e^{\beta\cos\varphi}\sin^{m+1}\varphi d\varphi\\
    &\leq \frac{\beta}{m}\int_0^\pi e^{\beta\cos\varphi}\sin^{m-1}\varphi d\varphi.
    \$
    By substituting these observations into \eqref{equ:final-equ}, we obtain that
    \$
    D_{\KL}(P_{\alpha}\|P_{\alpha'})\leq \frac{\beta^2r_0^2}{2m},
    \$
    which concludes the proof.
\end{proof}

\end{document}